\documentclass{amsart}
\usepackage[OT2,T1]{fontenc}
\usepackage[utf8]{inputenc}
\usepackage{tikz}
\usepackage{tikz-cd}
\usepackage{amsmath, amssymb}
\usepackage{stmaryrd} 
\usepackage{fourier}
\usepackage{hyperref}
\usepackage[all]{xy}
\usepackage{fullpage}
\usepackage[disable]{todonotes}

\usepackage{xcolor}

\tikzset{
  symbol/.style={
    draw=none,
    every to/.append style={
      edge node={node [sloped, allow upside down, auto=false]{$#1$}}}
  }
}
\tikzset{ext/.style={circle, draw,inner sep=1pt},int/.style={circle,draw,fill,inner sep=1pt},nil/.style={inner sep=1pt}}
\tikzset{exte/.style={circle, draw,inner sep=3pt},inte/.style={circle,draw,fill,inner sep=3pt}}
\tikzset{diagram/.style={matrix of math nodes, row sep=3em, column sep=2.5em, text height=1.5ex, text depth=0.25ex}}
\tikzset{diagram2/.style={matrix of math nodes, row sep=0.5em, column sep=0.5em, text height=1.5ex, text depth=0.25ex}}

\newcommand{\iid}{\mathrm{id}}
\newcommand{\op}{\operatorname}
\newcommand{\Map}{\op{Map}}

\newcommand{\flower}{\text{\floweroneright}}
\renewcommand{\U}{\mathcal{U}}
\newcommand{\C}{\mathcal{C}}

\newcommand{\Cyc}{\op{Cyc}}
\newcommand{\Com}{\mathsf{Com}}

\newcommand{\Lie}{\mathsf{Lie}}
\newcommand{\R}{\mathbb{R}}
\newcommand{\Q}{\mathbb{Q}}
\newcommand{\Graphs}{\mathsf{Graphs}}
\newcommand{\Diff}{\mathrm{Diff}}

\newcommand{\MAAA}{\mathcal M_A}
\newcommand{\dgVect}{dg\mathcal Vect}

\newcommand{\bCyc}{\overline{Cyc}}
\newcommand{\IBL}{IBL}
\newcommand{\fg}{{\mathfrak g}}
\newcommand{\K}{{\mathbb K}} 
\newcommand{\Tot}{\op{Tot}}
\newcommand{\FM}{\mathsf{FM}}
\newcommand{\D}{{\mathbf D}}
\newcommand{\Z}{{\mathbb Z}}
\newcommand{\GC}{\mathsf{GC}}
\DeclareSymbolFont{cyrletters}{OT2}{wncyr}{m}{n}
\DeclareMathSymbol{\Sha}{\mathbin}{cyrletters}{"78}

\newtheorem{Thm}{Theorem}[section]
\newtheorem{Prop}[Thm]{Proposition}
\newtheorem{Lem}[Thm]{Lemma}
\newtheorem{Rem}[Thm]{Remark}

\newcommand{\Th}{\mathit{Th}}
\newcommand{\evd}{f} 
\newcommand{\tp}{T} 

\newcommand{\hCom}{\widehat{\Com}}
\newcommand{\bCC}{\overline{CC}}
\newcommand{\bF}{{\mathbb F}}
\newcommand{\Der}{\mathrm{Der}}

\renewcommand{\L}{\mathbb{L}}

\title{String topology and configuration spaces of two points}
\author{Florian Naef}
\address{Department of Mathematics \\ Massachusetts Institute of Technology \\
182 Memorial Dr \\
Cambridge, MA 02142, USA}
\email{naeffl@mit.edu}
\author{Thomas Willwacher}
\address{Department of Mathematics \\ ETH Zurich \\
R\"amistrasse 101 \\
8092 Zurich, Switzerland}
\email{thomas.willwacher@math.ethz.ch}
\date{}

\begin{document}
\begin{abstract}
    Given a closed manifold $M$. We give an algebraic model for the Chas-Sullivan product and the Goresky-Hingston coproduct. In the simply-connected case, this admits a particularly nice description in terms of a Poincaré duality model of the manifold, and involves the configuration space of two points on $M$. We moreover, construct an $IBL_\infty$-structure on (a model of) cyclic chains on the cochain algebra of $M$, such that the natural comparison map to the $S^1$-equivariant loop space homology intertwines the Lie bialgebra structure on homology. The construction of the coproduct/cobracket depends on the perturbative partition function of a Chern-Simons type topological field theory.
    Furthermore, we give a construction for these string topology operations on the absolute loop space (not relative to constant loops) in case that $M$ carries a non-vanishing vector field and obtain a similar description.
    Finally, we show that the cobracket is sensitive to the manifold structure of $M$ beyond its homotopy type. More precisely, the action of $\Diff(M)$ does not (in general) factor through $\op{aut}(M)$.
\end{abstract}

\maketitle

\tableofcontents

\section{Introduction}
String topology is the study of algebraic structures on the the free loop space $LM$ of a smooth manifold $M$, as initiated by Chas and Sullivan \cite{ChasSullivan}.
In this paper we will consider the following subset of operations on the real or rational (co)homology.
First, the BV operator $\Delta:H_\bullet(LM)\to H_{\bullet+1}(LM)$ is the just the action of the fundamental chain of $S^1$, using the $S^1$ action on $LM$ by reparameterization of the loop, $S^1\times LM\to LM$.
Second, one has the string product $H_\bullet(LM)\otimes H_\bullet(LM)\to H_{\bullet-n}(LM)$, where $n$ is the dimension of the smooth manifold $M$.
The product is obtained by intersecting the loops at their basepoints, see below for more details. 
The product and the operation $\Delta$ generate a Batalin-Vilkovisky algebra structure on the homology $H_{\bullet+n}(LM)$.
Third, there is the string coproduct $H_{\bullet+n-1}(LM,M)\to H_\bullet(LM,M)\otimes H_\bullet(LM,M)$, by taking a self-intersection of the loop \cite{GoreskyHingston}. It is defined only on the relative homology with respect to constant loops.
We refer to section \ref{sec:stringtopology} below for a more detailed description of these operations.

Finally, we consider the homotopy $S^1$-quotient of $LM$, which we denote by $LM_{S^1}$.
Chas and Sullivan \cite{ChasSullivan2} describe a Lie bialgebra structure on the equivariant homology $H_\bullet^{S_1}(LM,M)$ of $LM$ relative to the constant loops, extending earlier work by Turaev \cite{Turaev}. More precisely, the Lie bialgebra structure is degree shifted, such that both bracket and cobracket have degree $2-n$.
Similarly, given a Poincaré duality model for $M$, a Lie bialgebra structure on the reduced equivariant homology $\bar H_\bullet^{S^1}(LM)$ has been described in \cite{ChenEshmatovGan}.
It should be remarked that the latter construction is completely algebraic and "formal", while the Chas-Sullivan definition is "topological", by intersecting loops.

By duality we obtain the corresponding dual operations on the cohomology of the loop spaces considered above.
The (co)homology can furthermore be efficiently computed.
To this end let $A$ be a differential graded commutative algebra (dgca) model for $M$, i.e., $A$ is quasi-isomorphic to the dgca of differential forms on $M$.
Then an iterated integral construction yields a map 
\begin{equation}\label{equ:ii1}
HH_{\bullet}(A,A) \to H^{-\bullet}(LM)
\end{equation}
from the Hochschild homology of $A$ with coefficients in $A$ to the cohomology of $LM$. Similarly, one obtains maps 
\begin{equation}\label{equ:ii2}
\overline{HH}_{\bullet}(A,A) \to H^{-\bullet}(LM,M)
\end{equation}
from a reduced version of the Hochschild homology and 
\begin{equation}\label{equ:ii3}
H\bCyc_{\bullet}(\bar A) \to \bar H^{-\bullet}_{S^1}(LM)
\end{equation}
from the homology of the (reduced) cyclic words in $A$.

\begin{Thm}[\cite{Jones,CohenJones02,ChenEshmatovGan}]
If $M$ is a simply connected closed manifold then the maps \eqref{equ:ii1}, \eqref{equ:ii2}, \eqref{equ:ii3} are isomorphisms.
\end{Thm}

\subsection{Statement of results}

The purpose of this paper is to understand the string topology operations described above on the objects on the left-hand side of \eqref{equ:ii1}, \eqref{equ:ii2}, \eqref{equ:ii3}, in the case that $M$ is a smooth compact manifold without boundary.
More concretely, our results are as follows.
First, we provide a (slightly) new version of the construction of the string product and coproduct, using the compactified configuration spaces $\FM_M(2)$ of two points on $M$, together with the inclusion of the boundary
$\partial \FM_M(2)=UTM\to \FM_M(2)$. This inclusion can be seen as the simplest instance of the action of the little disks operad on the configuration space of (framed) points on $M$. Our approach allows us to use existing models for the configuration spaces of points \cite{Idrissi,CamposWillwacher} to conduct computations in string topology.

We shall work in the cohomological setting (on $LM$), i.e., we consider the Hochschild homology, not cohomology, of our dgca model $A$ for $M$. Suppose first that $M$ is simply connected.
Then we may take for $A$ a Poincar\'e duality model for $M$ \cite{LambrechtsStanley}. In particular, $A$ comes equipped with a diagonal $\D=\sum \D'\otimes \D''\in A\otimes A$, such that $\sum a\D'\otimes \D''=\sum \pm \D'\otimes a\D''$ for all $a\in A$.
We may use this to construct a (degree shifted) co-BV structure on the reduced Hochschild complex $\bar C(A)=\bar C(A,A)$.
Concretely, the BV operator is the Connes-Rinehart differental on the Hochschild complex, given by the formula (cf. \cite[(2.1.7.3)]{Loday})
\begin{equation}\label{equ:connesB}
B(a_0,\dots, a_n) = \sum_{j=0}^n \pm (1, a_j, a_{j+1},\dots, a_0, a_1, \dots , a_n)
\end{equation}
The coproduct dual to the string product and the cup product on Hochschild cohomology is 
\begin{equation}\label{equ:intro product}
(a_0,\dots, a_n) \mapsto 
\sum_{j=0}^n
\sum
\pm 
(a_0\D',\dots, a_j) \otimes (\D'',a_{j+1},\dots, a_n).
\end{equation}
As a first application we then obtain another proof of the following result of Cohen and Voronov.
\begin{Thm}[\cite{CohenVoronov}]\label{thm:main_1}
For $M$ a closed simply connected manifold the map \eqref{equ:ii1} is an isomorphism of co-BV algebras.
\end{Thm}

As a second application we consider the string coproduct.
On the Hochschild chains the corresponding product operation is given by the formula
\begin{equation}\label{equ:intro coproduct}
(a_0,a_1,\dots,a_m) \otimes (b_0,b_1,\dots,b_n)
\to 
\sum\pm (b_0\D' a_0,a_1,\dots,a_m,\D'',b_1,\dots , b_n).
\end{equation}
We then show the following result, conjectured in (some form in) \cite{Abbaspour, Klamt2}, cf. also \cite{ChenEshmatovGan}.
\begin{Thm}\label{thm:main_2}
For $M$ a closed simply connected manifold the map \eqref{equ:ii2} respects the coproducts.
\end{Thm}

The Lie bialgebra structure on the $S^1$-equivariant (co)homology of $LM$ can be constructed from the string product and coproduct. Hence it also follows that the map \eqref{equ:ii3} respects the Lie bialgebra structures in the simply connected situation. 

Note that so far all string topology operations considered depend on $M$ only through the real (or rational) homotopy type of $M$, as encoded in the dgca model $A$.
This is in accordance with the result of \cite{CamposWillwacher,Idrissi} that the real homotopy type of the configuration spaces of points on $M$ only depends on the real homotopy type of $M$, for simply connected $M$. (And our construction depends only on a model for the configuration spaces, with the boundary inclusion.)  

However, we can also use our approach to "compute" the string topology operations for non-simply connected manifolds. In this case the maps \eqref{equ:ii1}-\eqref{equ:ii3} are no longer quasi-isomorphisms and "compute" has to be understood 
as providing algebraic operations on the left-hand sides that are preserved by those maps.
In this setting, one notably sees some indications of dependence of the string topology operations on $M$ beyond its real or rational homotopy type.
In this situation, we use the dgca model for the configuration space of points on $M$ constructed in \cite{CamposWillwacher}.
In that construction a central role is played by a dg Lie algebra of graphs $\GC_M$, whose elements are series of connected graphs with vertices decorated by elements of $\bar H_{\bullet}(M)$.
\begin{equation}\label{equ:GCMex}
\begin{tikzpicture}[scale=.7,baseline=-.65ex]
\node[int,label=90:{$\alpha$}] (v1) at (90:1) {};
\node[int,label=180:{$\beta$}] (v2) at (180:1) {};
\node[int] (v3) at (270:1) {};
\node[int,label=0:{$\gamma$}] (v4) at (0:1) {};
\draw (v1) edge (v4) edge (v2) (v3) edge (v2) edge (v4) edge (v1);
\end{tikzpicture}
\end{equation}
The dgca model for the configuration space is then completely encoded by a Maurer-Cartan element $Z\in \GC_M$.
The tree (i.e., loop-order-0-)part of $\GC_M$ can be identified with (almost) the Lie algebra encoding the real homotopy automorphisms of $M$. Similarly, the tree part $Z_0$ of $Z$ just encodes the real homotopy type of $M$. The higher loop orders hence encode potential dependence on $M$ beyond its real homotopy type.

We can use these models for configuration spaces to describe the string topology operations on the images of the morphisms \eqref{equ:ii1}-\eqref{equ:ii3}, and get explicit formulas.
However, since the formula for the coproduct is a bit ugly (see section \ref{sec:graphical version}), we will restrict to the equivariant situation and describe the string bracket and cobracket there.
So we consider again the loop space $LM$, with the goal of studying the Lie bialgebra structure on its $S^1$-equivariant cohomology.
First note that in the non-simply connected case the maps \eqref{equ:ii1}-\eqref{equ:ii3} still exist, but they are generally not quasi-isomorphisms.
Nevertheless, we can ask for a Lie bialgebra structure to put on the left-hand side of \eqref{equ:ii3} that makes \eqref{equ:ii3} into a morphism of Lie bialgebras.
We may also replace the reduced cyclic words $\bCyc \bar A$ in our dgca model $A$ of $M$ by the reduced cyclic words $\bCyc \bar H$ in the cohomology $H:=H^\bullet(M)$ of $M$.
In this case we have a non-trivial $\Com_\infty$-structure on $H$, and accordingly a differential on $\bCyc \bar H$ encoding the $\Com_\infty$-structure.
Our result is then the following, partially conjectured in \cite[Conjecture 1.11]{CFL}.
\begin{Thm}\label{thm:liebialg}\label{thm:main_3}
Let $M$ be a closed connected oriented manifold.
Then there is a (degree $2-d$-)homotopy involutive Lie bialgebra structure on the chain complex $\bCyc \bar H$, explicitly constructed below using only the Maurer-Cartan element $Z\in \GC_M$ above, such that the induced map in (co)homology
\[
H_\bullet\bCyc(\bar H) \to \bar H^{-\bullet}_{S^1}(LM)
\]
respects the Lie bracket and Lie cobracket.
\end{Thm}
We remark that for this theorem the Lie bracket and cobracket on $\bar H^{-\bullet}_{S^1}(LM)$ are defined using the string product and coproduct, roughly following \cite{GoreskyHingston, HingstonWahl}, see section \ref{sec:stringbracketdef} below. For the explicit formulas for the bracket and cobracket see section \ref{sec:thm3proof}, in particular Theorems \ref{thm:bracket} and \ref{thm:cobracket reduced}.

We also note that in dimensions $\neq 3$ the Maurer-Cartan element $Z$ can be taken without terms of loop orders $>1$, i.e., $Z=Z_0+Z_1$, with $Z_1$ of loop order 1.
In this case the homotopy involutive Lie bialgebra structure in the above theorem is in fact a strict Lie bialgebra structure.
Furthermore, the induced involutive Lie bialgebra structure on $H\bCyc(\bar H)$ depends only on the loop order 0 and 1 parts of $Z$ in all dimensions.
Finally, $Z_1$ is not easy to compute, and the authors do not have an example of a concrete manifold for which $Z_1$ is computable and known to be nontrivial. 

However, we expect the loop order 1 part of $\GC_M$ to correspond to nontrivial terms in (a Lie algebra model of) $\Diff(M)$ arising from topological Hochschild homology. 
We show that in families the string cobracket witnesses a dependence on $M$ beyond its real homotopy type.
In particular, we show that in the simply-connected case the following diagram commutes
$$
\begin{tikzcd}[column sep=tiny]
\pi_*(\Diff_1(M)) \ar[rr] \ar[d] && \Der_{[\cdot,\cdot], \delta}(\bar{H}^{S^1}_\bullet(LM)) \ar[d] \\
\pi_*(aut_1(M)) \ar[rr] \ar[dr] && \Der_{[\cdot,\cdot]}(\bar{H}^{S^1}_\bullet(LM)) \\
& \bar{H}^{S^1}_\bullet(LM) \ar[ur, "\op{ad}"'],
\end{tikzcd}
$$
and that in examples, the right vertical arrow is far from being surjective. Hence, in contrast to the string bracket, the cobracket gives a non-trivial condition (in general) on elements in $\pi_*(aut_1(M))$ to be in the image of $\pi_*(\Diff_1(M))$. In that sense we obtain that the string coproduct is not homotopy invariant.
For a more detailed discussion we refer to the concluding remarks in section \ref{sec:discussion}.

Let us also consider the case of $M$ being $1$-framed, that is, equipped with a nowhere vanishing vector field.
A necessary condition for this to exist is, of course, the vanishing of the Euler characteristic $\chi(M)=0$.
In this setting one may construct the string coproduct already on the cohomology of the loop space $H(LM)$, as opposed to on $H(LM,M)$ for general $M$, see section \ref{sec:framedcoprod} below.
By similar methods as above we then obtain:
\begin{Thm}\label{thm:main_4}
Let $M$ be a closed orientable $1$-framed manifold.
If $M$ is simply connected (and we hence have a Poincaré duality model) the map \eqref{equ:ii1} is compatible with the string coproduct (cohomology product), where the cohomology product on the left-hand side is defined by the same formula (which now makes sense on absolute chains).

For $M$ (potentially) non-simply connected the map \eqref{equ:ii3} intertwines the string bracket and cobracket on the right-hand side with the corresponding operations on the left-hand side, given by the same formulas as in Theorem \ref{thm:main_3}, cf. section \ref{sec:thm3proof} below.
\end{Thm}

We finally note that for the string topology operations considered here, only the configuration spaces of up to two points play a role, and in the graph complex only diagrams of loop order $\leq 1$. In light of Theorem \ref{thm:liebialg} it is hence reasonable to expect that a similar discussion of higher order string topology operations involves configuration spaces of more points, and diagrams of higher loop order.

\subsection*{Acknowledgements}
We are grateful for discussions, suggestions and support by Anton Alekseev, Ricardo Campos, Matteo Felder, Alexander Kupers, Pavel Safronov, Nathalie Wahl, and others. 
While working on this project we were made aware of similar results in preparation by Kaj B\"orjeson \cite{Borjeson}. We are grateful for him sharing his drafts.

The first author is partially supported by the Postdoc Mobility grant P400P2\_183900 of the Swiss National Science Foundation.
The second author is partially supported by the European Research Council under the ERC starting grant StG 678156 GRAPHCPX.

\section{Notation and recollections}

\subsection{Conventions on (co)chain complexes}
We generally work with cohomological degree conventions, that is, differentials in differential graded (dg) vector spaces have degree +1.
If we want to emphasize the cohomological nature of a dg vector space, we sometimes write it as $V^\bullet$, while $V_\bullet$ shall refer to homological conventions. Note that we often omit the $(-)^\bullet$, for example $H(M)=H^\bullet(M)$ is the cohomology of the manifold $M$.
Furthermore, all dg vector spaces will be over a field $\K$ of characteristic zero. For a (big) part of the results we need to restrict to either $\K=\R$ or $\K=\Q$. 
For $V$ a dg vector space we denote by $V[k]$ the $k$-fold degree shifted dg vector space, defined such that for an element $v\in V$ of degree $j$, the corresponding element of $V[k]$ has degree $j-k$. Such degree shifts can also be indicated on the degree placeholder $\bullet$ like $V^{\bullet +1}:=V^\bullet[1]$.

We define the tensor coalgebra of a cohomologically graded complex $V^\bullet$ with a (non-standard) degree shift
\[
TV^\bullet := \bigoplus_{k\geq 0} (V^{\bullet}[1])^{\otimes k}.
\]
This is to remove clutter when working with the Hochschild complex $TA\otimes A$ or similarly defined complexes later.

\subsection{Operads, $\Com_\infty$- and $\hCom_\infty$-algebras}
\label{sec:intro operads}
We denote by $\Com$ the commutative operad.
It has the standard Koszul resolution 
\[
\Com_\infty = \Omega(\Com^\vee)\to \Com
\]
as the cobar construction of the Koszul dual cooperad $\Com^\vee=\Lie^*\{1\}$, with $\Lie$ the Lie operad.
A $\Com_\infty$-algebra structure on a (differential) graded vector space $A$ can be encoded as a codifferential $D_A$ on the cofree $\Com^\vee$-coalgebra
\[
A = \bF^c_{\Com^\vee}(A[1]).
\]
A $\Com_\infty$-map $A\to B$ is by (our) definition a map of $\Com^\vee$-coalgebras
\[
\left( \bF^c_{\Com^\vee}(A[1]),D_A\right) \to \left( \bF^c_{\Com^\vee}(B[1]),D_B\right).
\]
One can strictify a $\Com_\infty$-algebra $A$ to a free $\Com$-algebra quasi-isomorphic to $A$
\[
\hat A := (\bF_{\Com}(\bF^c_{\Com^\vee}(A[1])[-1]), D).
\]
Concretely, the $\Com_\infty$-quasi-isomorphism $A\to \hat A$ is given by the inclusion
\[
\bF^c_{\Com^\vee}(A[1]) 
\subset
\bF_{\Com}(\bF^c_{\Com^\vee}(A[1])[-1])[1]
\subset 
\bF^c_{\Com^\vee}(\bF_{\Com}(\bF^c_{\Com^\vee}(A[1])[-1])[1]).
\]

We shall also use below the bar-cobar resolution 
\[
\hCom_\infty :=\Omega(B(\Com)) \xrightarrow{\simeq}\Com
\]
where $B(-)$ stands for the operadic bar construction.
In the same manner as above one defines $\hCom_\infty$-algebras and $\hCom_\infty$-morphisms.
We just replace $\Com^\vee$ by $B\Com$, or in other words $\Lie$ by $\Lie_\infty$.

There is a canonical quasi-isomorphism of cooperads $\Com^\vee\to B\Com$ and accordingly a canonical quasi-isomorphism
\[
\Com_\infty \to \hCom_\infty.
\]
Hence any $\hCom_\infty$-algebra is in particular a $\Com_\infty$-algebra, and a $\hCom_\infty$-map between two $\hCom_\infty$-algebras induces a $\Com_\infty$-map between the corresponding $\Com_\infty$-algebras.

\subsection{Hochschild and cyclic complex}\label{sec:intro hochschild}
For an associative or more generally $A_\infty$-algebra $A$ we consider the reduced Hochschild complex $\bar C(A)$ of $A$. All of our algebras will be augmented, and in this case we can write
\[
\bar C(A) = \left( \bigoplus_{k\geq 0} (\bar A[1])^{\otimes k} \otimes A, d_H\right)
\]
where $d_H$ is the Hochschild differential and $\bar A$ is the augmentation ideal.

The negative cyclic complex of $A$ is 
\[
(\bar C(A)[[u]], d_H+uB),
\]
where $u$ is a formal variable of degree $+2$ and $B$ is the Connes-Rinehart differential, see \eqref{equ:connesB}.

We similarly define the reduced negative cyclic complex to be cyclic homology relative to $\R \to A$
\[
\bCC(A) = \left(\bigoplus_{k\geq 0}  (\bar A[1])^{\otimes k} \otimes A[[u]], d_H+uB\right)/ \R[[u]],
\]
since our unit is split, this will differ by $\R[[u]]$ from the original one.
Let 
\[
\bCyc(\bar A) = \bigoplus_{k\geq 1} \left(  (\bar A[1])^{\otimes k} \right)_{S_k}
\]
be the reduced cyclic words.
Then there is a natural map of complexes, and a quasi-isomorphism as we will see in Proposition \ref{prop:cycqiso} below,
\[
\bCyc(\bar A) \to \bCC(A),
\]
essentially by applying the operator $B$. More concretely
\[
(a_1,\cdots ,a_k)
\mapsto 
\sum_{j=1}^k \pm (a_j,\dots,a_{j-1},1).
\]

Now consider an $A_\infty$-map $f: A\to B$ between (unital) $A_\infty$-algebras.
The constructions above are functorial in $A$, and one natural maps between the Hochschild and cyclic complexes of $A$ and $B$, induced by $f$.
Furthermore, any $\Com_\infty$-algebra is an $A_\infty$-algebra, and a $\Com_\infty$-map is an $A_\infty$ map, so the same applies to $\Com_\infty$-maps $f:A\to B$.
By the previous subsection, we may also replace (a fortiori) $\Com_\infty$ by $\hCom_\infty$. 

When dealing with words $\alpha=(a_1,\dots,a_p),\beta=(b_1,\dots,b_q)\in TX$ in the tensor algebra of a vector space $X$, we denote by 
\[
\alpha\beta=(a_1,\dots,a_p,b_1,\dots,b_q)\in TX
\]
their concatenation.
Similarly, we denote by $\Sha$ the shuffle product
\[
\alpha\Sha\beta
=\sum_{\sigma \in \mathit{Sh}(p,q)} \sigma\cdot (\alpha\beta),
\]
where the sum runs over all $(p,q)$-shuffle permutations.
For example, the reduced Hochschild complex of a commutative algebra has a commutative product given by the formula
\[
(\alpha,\alpha_0)\Sha (\beta,\beta_0)
=
\pm (\alpha\Sha \beta, \alpha_0\beta_0).
\]
Note that here $\alpha_0\beta_0$ is the product of the two elements in $A$, not the juxtaposition.

\subsection{Pullback-pushout lemma}
We will later make use of the following result, see for example \cite[Theorem 2.4]{HessNotes} or \cite[Proposition 15.8]{FHT2}.

\begin{Thm}
Consider the pullback diagram
\[
\begin{tikzcd}
E\times_B X \ar{r} \ar{d}&  E\ar{d}{p} \\
X\ar{r} & B 
\end{tikzcd}
\]
where $p$ is a Serre fibration, $E$ is path connected and $X$ and $B$ are simply connected. Let $A_X\leftarrow A_B\rightarrow A_E$ be a rational dgca model for the lower right zigzag in the diagram.
Then the homotopy pushout
\[
A_X \otimes^h_{A_B} A_E 
\]
is a dgca model for the pullback $E\times_B X$.
\end{Thm}

\subsection{Fulton-MacPherson-Axelrod-Singer compactification of configuration spaces}\label{sec:FM}
Consider an oriented manifold $M$. Axelrod and Singer \cite{AxelrodSinger} defined compactifications of the configuration spaces of points on $M$, by iterated real bordification. We denote the thus created compactified configuration space of $r$ points by $\FM_M(r)$.
We shall not recall the details of the compactification procedure here. We just note that a point in $\FM_M(r)$ can be seen as a decorated tree with $r$ leaves, with the root node decorated by a configuration of points in $M$, and the other nodes by (essentially) configurations of points in tangent spaces of $M$.
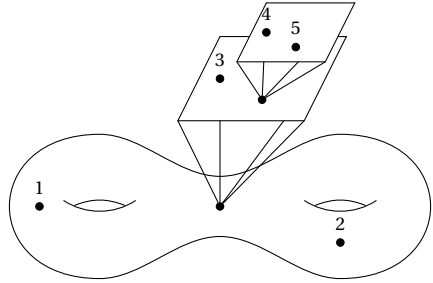
\[
  \begin{tikzpicture}[baseline=1cm, scale=.8]
    \draw (0,0)
    .. controls (-.7,0) and (-1.3,-.7) .. (-2,-.7)
    .. controls (-4,-.7) and (-4,1.7) .. (-2,1.7)
    .. controls (-1.3,1.7) and (-.7,1) .. (0,1)
    .. controls  (.7,1) and (1.3,1.7) .. (2,1.7)
    .. controls (4,1.7) and (4,-.7) .. (2,-.7)
    .. controls (1.3,-.7) and (.7,0)  .. (0,0);
    \begin{scope}[xshift=-2cm, yshift=.6cm, scale=1.2]
      \draw (-.5,0) .. controls (-.2,-.2) and (.2,-.2) .. (.5,0);
      \begin{scope}[yshift=-.07cm]
        \draw (-.35,0) .. controls (-.1,.1) and (.1,.1) .. (.35,0);
      \end{scope}
    \end{scope}
    \begin{scope}[xscale=-1, xshift=-2cm, yshift=.6cm, scale=1.2]
      \draw (-.5,0) .. controls (-.2,-.2) and (.2,-.2) .. (.5,0);
      \begin{scope}[yshift=-.07cm]
        \draw (-.35,0) .. controls (-.1,.1) and (.1,.1) .. (.35,0);
      \end{scope}
    \end{scope}
    \node [int, label={$\scriptstyle 1$}] (v1) at (-3,.5) {};
    \node [int, label={$\scriptstyle 2$}] (v2) at (2,-.1) {};
    \node [int] (va) at (0,.5) {};
    \begin{scope}[scale = .7, xshift=-1cm,yshift=2.75cm]
      \draw (va) -- (0,0) (va) -- (3,0) (va)-- (4,2) (va) -- (1,2);
      \draw[fill=white] (0,0)--(3,0)--(4,2)--(1,2)--cycle;
      \node [int, label={$\scriptstyle 3$}] (v3) at (1,1) {};
      \node [int] (vb) at (2,.5) {};
      \begin{scope}[scale = .7, xshift=2cm,yshift=2cm]
        \draw (vb) -- (0,0) (vb) -- (3,0) (vb)-- (4,2) (vb) -- (1,2);
        \draw[fill=white] (0,0)--(3,0)--(4,2)--(1,2)--cycle;
        \node [int, label={$\scriptstyle 4$}] (v4) at (1,1) {};
        \node [int, label={$\scriptstyle 5$}] (v5) at (2,.5) {};
      \end{scope}
    \end{scope}
  \end{tikzpicture}
\]

Similarly, one defines a version of the little disks operad $\FM_n$ assembled from the compactified configuration spaces of $r$ points in $\R^n$.
From this one may finally build a fiberwise version of the little disks operad 
\[
\FM_n^M = \mathit{Fr}_M \times_{SO(n)} \FM_n,
\]
where $\mathit{Fr}_M$ is the oriented orthonormal frame bundle for some (irrelevant) choice of metric on $M$.
The collection $\FM_n^M$ can be seen either as an operad in spaces over $M$, or as a colored operad with colors $M$.
The collection of spaces $\FM_M(r)$ then assembles into an operadic right module over $\FM_n^M$.

For our purposes it shall suffice to understand the situation in arities $r\leq 2$.
We have 
\begin{align*}
\FM_M(1)&=M = \FM_n^M(1) &\text{and} & & \FM_n^M(2) &= UTM,
\end{align*}
where $UTM$ is the unit tangent bundle of $M$.
The simplest instance of the operadic right action (and in fact the only instance we need) is the composition
\[
\FM_M(1) \times_M \FM_n^M(2) = UTM \to \FM_M(2), 
\]
which is just the inclusion $UTM\cong\partial\FM_M(2) \to \FM_M(2)$.

\subsection{The Lambrechts-Stanley model of configuration space}
We shall need cochain and dg commutative algebra models for configurations spaces of (up to 2) points. The simpler version of these are models proposed by Lambrechts and Stanley.
Concretely, for a simply connected closed manifold one can find a Poincar\'e duality model $A$, see \cite{LambrechtsStanley}. This is a dg commutative algebra quasi-isomorphic to the differential forms $\Omega(M)$, exhibiting Poincar\'e duality on the cochain level. In particular, we have a coproduct $\Delta : A\to A\otimes A$ of degree $n$.
The Lambrechts-Stanley model \cite{LambrechtsStanley2} is a (tentative) dg commutative algebra model for $\FM_M(r)$, which can be built out of $A$.
In particular for $r=2$ this is just
\[
\mathit{cone}(A\xrightarrow{\Delta} A\otimes A). 
\]
It has been shown in \cite{LambrechtsStanley3} that for $r\leq 2$ and 2-connected closed $M$ the proposed Lambrechts-Stanley model is indeed a model, i.e., quasi-isomorphic to $\Omega(\FM_M(2))$.
This has been extended by \cite{Idrissi, CamposWillwacher} to arbitrary $r$ and simply connected closed manifolds, provided $n\geq 4$.

\subsection{Graph complex models for configuration spaces}\label{sec:GraphsM}
For non-simply connected manifolds we cannot guarantee the existence of Poincar\'e duality models.
However, by work of Campos-Willwacher \cite{CamposWillwacher} one can still write down explicit, albeit more complicated models $\Graphs_M(r)$ of configuration spaces, for $M$ connected closed and oriented.
We shall only sketch the construction.
The dg vector space $\Graphs_M(r)$ consists (essentially) of linear combinations of (isomorphism classes of) diagrams with $r$ ``external'' vertices labelled $1,\dots,r$, and an arbitrary (finite) number of internal vertices. In addition, all vertices may be decorated by zero or more elements of the reduced cohomology $\bar H(M)=H^{\geq 1}(M)$. Finally, each connected component of a graph must contain at least one external vertex.
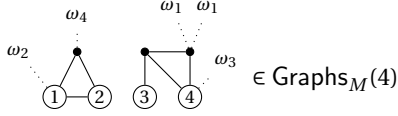
\[
  \begin{tikzpicture}[scale=1.2]
    \node[ext] (v1) at (0,0) {$\scriptstyle 1$};
    \node[ext] (v2) at (.5,0) {$\scriptstyle 2$};
    \node[ext] (v3) at (1,0) {$\scriptstyle 3$};
    \node[ext] (v4) at (1.5,0) {$\scriptstyle 4$};
    \node[int] (w1) at (.25,.5) {};
    \node[int] (w2) at (1.5,.5) {};
    \node[int] (w3) at (1,.5) {};
    \node (i1) at (1.7,1) {$\scriptstyle \omega_1$};
    \node (i2) at (1.3,1) {$\scriptstyle \omega_1$};
    \node (i3) at (-.4,.5) {$\scriptstyle \omega_2$};
    \node (i4) at (1.9,.4) {$\scriptstyle \omega_3$};
    \node (i5) at (0.25,.9) {$\scriptstyle \omega_4$};
    \draw (v1) edge (v2) edge (w1) (w1)  edge (v2) (v3) edge (w3) (v4) edge (w3) edge (w2) (w2) edge (w3);
    \draw[dotted] (v1) edge (i3) (w2) edge (i2) edge (i1) (v4) edge (i4) (w1) edge (i5);

    \node at (3,.2) {$\in \Graphs_M(4)$};
  \end{tikzpicture}
\]
For concreteness, we pick a homogeneous basis $(e_q)$ of $H(M)$, and we denote by $e^q = (e_q^*)$ the Poincar\'e-dual basis, such that the diagonal in $H(M\times M)$ is represented by the element $e_q \otimes e^q = \sum_q e_q\otimes e_q^*$.

Then the differential in our graph complex acts by edge contraction, and replacing each edge by a diagonal in $H(M)\otimes H(M)$ (cutting the edge).
\begin{equation}\label{equ:edgecsplit}
 d\, 
 \begin{tikzpicture}[baseline=-.65ex]
 \node[int] (v) at (0,0) {};
 \node[int] (w) at (0.7,0) {};
 \draw (v) edge +(-.5,0) edge +(-.5,.5) edge +(-.5,-.5) edge (w)
   (w) edge +(.5,0) edge +(.5,.5) edge +(.5,-.5);
 \end{tikzpicture}
 =
  \begin{tikzpicture}[baseline=-.65ex]
 \node[int] (v) at (0,0) {};
 \node[int] (w) at (0,0) {};
 \draw (v) edge +(-.5,0) edge +(-.5,.5) edge +(-.5,-.5) edge (w)
   (w) edge +(.5,0) edge +(.5,.5) edge +(.5,-.5);
 \end{tikzpicture}
 +
 \sum_q
  \begin{tikzpicture}[baseline=-.65ex]
 \node[int] (v) at (0,0) {};
 \node[int] (w) at (1,0) {};
     \node (i1) at (0.3,.5) {$\scriptstyle e_q$};
      \node (i2) at (0.7,-0.5) {$\scriptstyle e_{q}^*$};
 \draw (v) edge +(-.5,0) edge +(-.5,.5) edge +(-.5,-.5) 
   (w) edge +(.5,0) edge +(.5,.5) edge +(.5,-.5);
   \draw[dotted] (v) edge (i1) (w) edge (i2); 
 \end{tikzpicture}
\end{equation}
Here, if $e_q$ or $e_q^*$ is the unit $1\in H^0(M)$ (or a multiple thereof), we just drop the corresponding decoration. Alternatively, we can say that we could also define our graph complex with decorations in the cohomology $H(M)$ instead of the reduced cohomology $\bar H(M)$, and then impose the relation that units may be dropped from any (internal or external) vertex
\begin{equation}\label{equ:onerelation}
\begin{tikzpicture}[baseline=-.65ex]
 \node[int,label=90:1] (v) at (0,0) {};
 \draw (v) edge +(-.5,-.5) edge +(0,-.5) edge +(.5,-.5);
 \end{tikzpicture}
 =
\begin{tikzpicture}[baseline=-.65ex]
 \node[int] (v) at (0,0) {};
 \draw (v) edge +(-.5,-.5) edge +(0,-.5) edge +(.5,-.5);
 \end{tikzpicture}\, .
\end{equation}
In fact, this relation shall be seen as of "cosmetic" origin, eventually yielding a smaller but quasi-isomorphic complex.

More severely, we shall also note that the cutting operation in \eqref{equ:edgecsplit} might produce a graph with a connected component without external vertices, violating the connectivity condition above.
In that case, the cut-off-subgraph is formally mapped to a number, given a map from such graphs to numbers. That latter map is called partition function $Z$, and combinatorially encodes the real homotopy type of the configuration spaces of points.
In fact, we may understand $Z$ as a Maurer-Cartan element in a dual graph complex $\GC_{\bar H(M)}$, whose elements are formal series in graphs without external vertices such as \eqref{equ:GCMex}, see the following subsection. Furthermore one can separate graphs of various loop orders present in $Z$:
 \begin{equation}\label{equ:Zsplit}
 Z=Z_{tree}+Z_1+Z_2+\cdots,
\end{equation}
where $Z_{tree}$ is the tree piece, $Z_1$ contains only the $1$-loop graphs etc.
The piece $Z_{tree}$ encodes precisely the real homotopy type of $M$, i.e., the $\Com_\infty$ structure on $H(M)$.
The higher corrections $Z_{\geq 2}$ vanish if the dimension $n$ of $M$ satisfies $n\neq 3$.
The piece $Z_1$ also vanishes for degree reasons if $H^1(M)=0$ and can be made to vanish if $n=2$.
\todo[inline]{FlorianX}

Below we shall see that the piece $Z_1$ of the partition function $Z$ appears in our formula for the string cobracket.
This in itself is not a contradiction to the conjectured homotopy invariance of the string topology operations.  However, as we will see below this has the consequence that the $\Diff(M)$-action on string topology does not factor through the homotopy automorphisms.

\subsection{Graph complex (Lie algebra) $\GC_H$ and $\GC_M$, following \cite{CamposWillwacher}}\label{sec:GCM}
We shall need below a more explicit definition of the graph complex in which $Z$ above is a Maurer-Cartan element.
Generally, consider a finite dimensional graded vector space $H$ with a non-degenerate pairing $\epsilon:H\otimes H\to \R$ of degree $-n$. Our main example will be $H=H^\bullet(M)$, the cohomology of a closed oriented connected manifold, with the pairing provided by Poincar\'e duality. With this example in mind, we assume that the subspaces of degree 0 and $n$ are one-dimensional, and there is a distinguished element $1\in H^0$, which in our case will be the unit of the cohomology algebra. We denote the dual element of degree $n$ by $\omega$.
We denote by $\bar H=H^{\neq 0}$ the corresponding reduced version of $H$. 
We also use the notation $1^*,\omega^*$ to denote dual elements in the dual space $H^*$.
For concreteness, we also pick a basis $f_q$ of $H^*$ in degrees $\neq 0,n$, and denote by $f_q^*$ the Poincar\'e-dual basis.

Then we may consider a dg Lie algebra $\GC_{\bar H}'$ whose elements are series of (isomorphism classes of) connected graphs, with vertices carrying decorations by $H^*$.
There is a differential by splitting vertices, or connecting two decorations
\begin{align}\label{equ:GCMdelta}
 \delta \, 
  \begin{tikzpicture}[baseline=-.65ex]
 \node[int] (v) at (0,0) {};
 \node[int] (w) at (0,0) {};
 \draw (v) edge +(-.5,0) edge +(-.5,.5) edge +(-.5,-.5) edge (w)
   (w) edge +(.5,0) edge +(.5,.5) edge +(.5,-.5);
 \end{tikzpicture}
 &=
\sum
\begin{tikzpicture}[baseline=-.65ex]
 \node[int] (v) at (0,0) {};
 \node[int] (w) at (0.7,0) {};
 \draw (v) edge +(-.5,0) edge +(-.5,.5) edge +(-.5,-.5) edge (w)
   (w) edge +(.5,0) edge +(.5,.5) edge +(.5,-.5);
 \end{tikzpicture}
 &
\delta\, 
 \begin{tikzpicture}[baseline=-.65ex]
 \node[int] (v) at (0,0) {};
 \node[int] (w) at (1,0) {};
     \node (i1) at (0.3,.5) {$\scriptstyle \alpha$};
      \node (i2) at (0.7,-0.5) {$\scriptstyle \beta$};
 \draw (v) edge +(-.5,0) edge +(-.5,.5) edge +(-.5,-.5) 
   (w) edge +(.5,0) edge +(.5,.5) edge +(.5,-.5);
   \draw[dotted] (v) edge (i1) (w) edge (i2); 
 \end{tikzpicture}
&=
\epsilon(\alpha,\beta)\, 
 \begin{tikzpicture}[baseline=-.65ex]
 \node[int] (v) at (0,0) {};
 \node[int] (w) at (0.7,0) {};
 \draw (v) edge +(-.5,0) edge +(-.5,.5) edge +(-.5,-.5) edge (w)
   (w) edge +(.5,0) edge +(.5,.5) edge +(.5,-.5);
 \end{tikzpicture}\, .
\end{align}
which should be seen as the dual to the edge contraction and edge splitting \eqref{equ:edgecsplit} above. 
Also, there is a Lie bracket, which is again given by pairing two decorations, replacing them by an edge, schematically:
\begin{equation}\label{equ:GCMbracket}
\left[
 \begin{tikzpicture}[baseline=-.65ex]
 \node[int] (v) at (0,0) {};
 \node[draw,circle, minimum size=7mm] (w) at (-1,0) {$\Gamma$};
     \node (i1) at (0.3,.5) {$\scriptstyle \alpha$};
 \draw (v) edge (w.north east) edge (w) edge (w.south east);
   \draw[dotted] (v) edge (i1);
 \end{tikzpicture}
 ,
 \begin{tikzpicture}[baseline=-.65ex]
 \node[int] (v) at (1,0) {};
 \node[draw,circle, minimum size=7mm] (w) at (2,0) {$\Gamma'$};
      \node (i1) at (0.7,-0.5) {$\scriptstyle \beta$};
 \draw (v) edge (w.north west) edge (w) edge (w.south west);
   \draw[dotted] (v) edge (i1);
 \end{tikzpicture}
 \right]
 =
 \epsilon(\alpha,\beta)\, \,
  \begin{tikzpicture}[baseline=-.65ex]
 \node[int] (v) at (0,0) {};
 \node[draw,circle, minimum size=7mm] (w) at (-1,0) {$\Gamma$};
 \node[draw,circle, minimum size=7mm] (w2) at (1.5,0) {$\Gamma$};
 \node[int] (v2) at (.5,0) {};
 \draw (v) edge (w.north east) edge (w) edge (w.south east) edge (v2);
  \draw (v2) edge (w2.north west) edge (w2) edge (w2.south west);
 \end{tikzpicture}
 \, .
\end{equation}
There are also sign and degree conventions, which we shall largely ignore here, but refer the reader to the original reference \cite{CamposWillwacher} instead.
There is a Maurer-Cartan element 
\begin{equation}\label{equ:GCHz}
z=
\sum_{j\geq 0} \frac 1 {j!} \left(
\begin{tikzpicture}[baseline=-.65ex]
 \node[int] (v) at (0,0) {};
     \node (i1) at (0.3,.5) {$\scriptstyle (1^*)^j$};
     \node (i2) at (0.3,-0.5) {$\scriptstyle \omega^*$};
   \draw[dotted] (v) edge (i1) edge (i2); 
 \end{tikzpicture}
 +
 \frac 1 2 \,
 \sum_q\,
\begin{tikzpicture}[baseline=-.65ex] 
 \node[int] (v) at (0,0) {};
     \node (i3) at (0.3,.5) {$\scriptstyle (1^*)^j$};
     \node (i1) at (0.3,-.5) {$\scriptstyle f_q$};
     \node (i2) at (-0.3,-0.5) {$\scriptstyle f_q^*$};
   \draw[dotted] (v) edge (i1) edge (i2); 
 \end{tikzpicture}
 \right)
\end{equation}
where the notation $(1^*)^j$ shall indicate that $j$ copies of $1^*$ are present decorating the vertex. We shall define $\GC_H:= (\GC_H')^z$ as the twist by this MC element.

It is convenient to also define a cosmetic variant, getting rid of decorations by $1^*$ in graphs. (This is also done in \cite{CamposWillwacher}.)
More precisely, we construct a dg Lie algebra $\GC_{\bar H}$ by repeating the construction of $\GC_{H}$, except for the following differences:
\begin{itemize}
    \item We only allow decorations by $\bar H^*$ in graphs.
    \item In the differential \eqref{equ:GCMdelta} and bracket \eqref{equ:GCMbracket} we tacitly assume that every vertex is decorated by copies of elements $1^*$, i.e., a decoration $\omega^*$ is replaced by an edge to any vertex.
    \item In the MC element $z$ of \eqref{equ:GCHz} we merely drop all terms involving decorations by $1^*$, leaving only the $j=0$-term in the outer sum.
\end{itemize}

Thus we obtain a dg Lie algebra $\GC_{\bar H}$.
There is a natural map of dg Lie algebras
\begin{equation}\label{equ:GCbarHGCH}
\GC_{\bar H} \to \GC_H
\end{equation}
by sending a graph to all possible graphs obtainable by adding $1$-decorations to all vertices.
Formally, to each vertex, we do the following operation
\[
\begin{tikzpicture}[baseline=-.65ex]
 \node[int] (v) at (0,0) {};
  \draw (v) edge +(-.5,0) edge +(-.5,.5) edge +(-.5,-.5);    
 \end{tikzpicture}
 \, \,
\mapsto
\, \,
\sum_{j\geq 0} \frac 1 {j!}\,
 \begin{tikzpicture}[baseline=-.65ex]
   \node[int] (v) at (0,0) {};
   \node (i1) at (0.3,.5) {$\scriptstyle (1^*)^j$};
   \draw[dotted] (v) edge (i1);
   \draw (v) edge +(-.5,0) edge +(-.5,.5) edge +(-.5,-.5);
 \end{tikzpicture}\, .
\]
(The map is in fact a quasi-isomorphism, though we shall not use this.)

The Maurer-Cartan element $Z$ of \cite{CamposWillwacher} takes values in $\GC_{\bar H(M)}$.
However, given the map \eqref{equ:GCbarHGCH} we may map it to another MC element $Z\in \GC_H$, which we shall denote by the same letter, abusing notation.

There is one important observation, for which we refer to \cite{CamposWillwacher}.
The MC element $Z\in \GC_{\bar H(M)}$ may be taken to be composed of graphs which are at least trivalent, i.e., the valency of any vertex in any graph occurring is $\geq 3$.
The valency of a vertex is defined to be the number of elements in its star, which is in turn the set of half-edges and decorations incident at that vertex.
Later on we shall need this observation in the following form.
If we consider the MC element $z+Z\in \GC'_{H(M)}$, and we consider only its trivalent part, then the only graphs which ever contain any decorations by $1^*$ are those from the part $z$, and can explicitly be read off from \eqref{equ:GCHz} above.

\section{String topology operations}
\label{sec:stringtopology}
The goal of this section is to introduce the construction of the string product and coproduct, in the form we will be using them.
We will generally work in the cohomological setting, i.e., we will define the operations on the cohomology of the loop space, not on homology as usual.

\subsection{Preliminaries}
Let $M$ be a closed oriented manifold. Let $\FM_M(2)$ denote the compactified configuration space of two (labelled) points on $M$. It can be constructed as the real oriented blowup of $M \subset M \times M$, and thus its boundary can be identified with the unit tangent bundle $UTM$. It naturally fits into the following commuting square
\begin{equation}\label{diag:preucolim}
\begin{tikzcd}
UTM \ar[r] \ar[d]& \FM_M(2) \ar[d] \\
M \ar[r] & M \times M,
\end{tikzcd}
\end{equation}
which is actually a homotopy pushout as the map $UTM \to \FM_M(2)$ is a cofibration. The vertical homotopy cofibers are two versions of the Thom space of $M$. The diagram hence gives a homotopy equivalence between the Thom space $DM / UTM$ (with $DM\to M$ the unit disk bundle) and $M \times M / (M \times M \setminus M)$ that does not depend on a tubular neighborhood embedding. We will exploit this fact in our construction of string topology operations.

As an example, consider the map $H_\bullet(M\times M) \to H_{\bullet-n}(M)$ obtained by intersecting with the diagonal. Given the above observation we may realize the dual map on cohomology by the zigzag
\begin{equation}\label{equ:example_diag}
\begin{tikzcd}
H^\bullet(M\times M) & \ar[l] H^\bullet( M \times M , \FM_M(2)) \ar[d, "\simeq"] &  & \\
& H^\bullet(M, UTM) & \ar[l, "\wedge Th"] H^{\bullet-n}(M),
\end{tikzcd}
\end{equation}
where the last arrow is the Thom isomorphism (multiplication by the Thom form) and the fact that the vertical arrow is an isomorphism uses that \eqref{diag:preucolim} is a homotopy pushout.
While this zigzag might not be the the simplest expression for the intersection with the diagonal, it has the advantage that it is relatively straightforward to realize on the cochain complex level, without many "artificial choices".
The construction of the string topology operations prominently involves the intersection with the diagonal, and hence we will be using this example below. We shall need a slight extension, allowing for pullbacks.
To this end, let us note the following.
\begin{Lem}
\label{lem:ucolim}
Let $E \to M \times M$ be a fibration. Then the following diagram (of pullbacks) is a homotopy pushout.
\begin{equation}
\label{diag:ucolim}
\begin{tikzcd}
E|_{UTM} \ar[r] \ar[d]& E|_{\FM_M(2)} \ar[d] \\
E|_{M} \ar[r] & E,
\end{tikzcd}
\end{equation}
in particular the maps between cofibers are equivalences.
\end{Lem}
\begin{proof}
The diagram is clearly a pushout, and the map $E|_{UTM} \to E|_{\FM_M(2)}$ is a cofibration. Alternatively, this is Mather's cube theorem, visualizing \eqref{diag:ucolim} as the top face of a cube with bottom face \eqref{diag:preucolim}.
\end{proof}

Let $PM \to M \times M$ be the path space fibration of $M$. We will denote by $LM$ the free loop space $M^{S^1} = PM \times_{M \times M} M$. 

\subsection{String product (Cohomology coproduct)}
We define the string product
$$
H^\bullet(LM) \otimes H^\bullet(LM) \longleftarrow H^{\bullet -n}(LM),
$$
to be the composite of the maps
\begin{equation}\label{equ:product_zigzag}
\begin{tikzcd}
H^\bullet(LM) \otimes H^\bullet(LM) & \ar[l] H^\bullet( LM \times LM , LM \times^\prime LM) \ar[d, "\simeq"] &  & \\
& H^\bullet(\Map(8), \Map^\prime(8)) & \ar[l, "\wedge Th"] H^{\bullet-n}(\Map(8)) & \ar[l] H^{\bullet-n}(LM),
\end{tikzcd}
\end{equation}
which we will describe now. We apply Lemma \ref{lem:ucolim} to the fibration $LM \times LM \to M \times M$ to obtain the following homotopy pushout diagram
$$
\begin{tikzcd}
\Map^\prime(8) \ar[r] \ar[d]& LM \times^\prime LM \ar[d] \\
\Map(8) \ar[r] & LM \times LM,
\end{tikzcd}
$$
where each entry is defined to be the fiber product of $LM \times LM$ with the corresponding term in \eqref{diag:preucolim} over $M \times M$, for example
\[
\Map'(8):= (LM\times LM) \times_{M\times M} UTM,
\]
which can be thought of as the space of figure eights in $M$ together with a tangent vector at the node of the eight.
This defines all the spaces in the definition of the string product and explains the vertical isomorphism in \eqref{equ:product_zigzag}. The third map in \eqref{equ:product_zigzag} is given by multiplying with the image of the Thom class $Th \in H^n(M, UTM)$ in $H^n(\Map(8), \Map^\prime(8))$. The last map is induced by the natural map $\Map(8) \to LM$, traversing both "ears" of the figure 8, in a fixed order.

\begin{Rem}
Since later it will be more natural to work in the cohomological setting, but the string topological operation have a geometric meaning, we still call the above map a product, even though it has the signature of a coproduct. To make up for this we will often write the maps from right to left.
\end{Rem}

The diagram \eqref{equ:product_zigzag} is obtained from the chain-level version of diagram \eqref{equ:example_diag} by taking fiber product with a certain fibration. More precisely, let us consider the following diagram of pairs of spaces
$$
\begin{tikzcd}
M \times M \ar[r] &  (M \times M, \FM_M(2)) & \\
& \ar[u, "\simeq"] (M, UM) \ar[r, dashed, "Th"] & M,
\end{tikzcd}
$$
where the dashed map is the Thom isomorphism and exists on chain level.
In particular, this induces the corresponding maps upon taking the fiber product with $LM \times LM \to M \times M$.
The map
$$
\begin{tikzcd}
H^\bullet(\Map(8), \Map^\prime(8)) & \ar[l, "\wedge Th"'] H^{\bullet-n}(\Map(8)),
\end{tikzcd}
$$
on chain level is given by taking the cup product with the pullback of the Thom class $Th \in C^n(M, UTM)$ along the map of pairs
$$
\begin{tikzcd}
(\Map(8), \Map^\prime(8)) \ar[r] & (M , UTM).
\end{tikzcd}
$$

\subsection{String coproduct (cohomology product)}
The coproduct operation we are interested in is defined on loops relative to constant loops, i.e.
$$
H^\bullet(LM, M) \longleftarrow H^\bullet(LM, M) \otimes H^\bullet(LM, M) [n-1].
$$
We will later see that it admits a natural lift to $H^\bullet(LM)$ in case that there is a non-vanishing vector field on $M$ (in particular $M$ has Euler characteristic $0$). The coproduct will be constructed in a similar fashion to the product, but this time we apply Lemma \ref{lem:ucolim} to the fibration $E = PM \times_{M\times M} PM \to M \times M$. We will identify $PM \times_{M\times M} PM$ with $\Map(\bigcirc_2)$, the space of two-pointed loops in $M$. Taking the fiber product with the diagram \eqref{diag:preucolim} we obtain the homotopy pushout
$$
\begin{tikzcd}
\Map^\prime(8) \arrow{r} \arrow{d} & \Map^\prime(\bigcirc_2) \arrow{d} \\
\Map(8) \arrow{r} & \Map(\bigcirc_2).
\end{tikzcd}
$$
Let us define the spliting/reparametrization map
\begin{equation}\label{equ:splittingmap_def}
\begin{aligned}
    s : I \times LM &\longrightarrow \Map(\bigcirc_2) \\
    (t, \gamma) &\longmapsto (s \mapsto \begin{cases} \gamma(\tfrac{1}{2}st) &\mbox{if } 0 \leq s \leq \tfrac{1}{2} \\ \gamma(2t(1-s) + 2(s-\tfrac{1}{2}) & \tfrac{1}{2} \leq s \leq 1 \end{cases}.
\end{aligned}
\end{equation}
This map sends the boundary of the interval into $F := LM \coprod_M LM \subset \Map(\bigcirc_2)$, that is the space of loops with two marked points where either one of the intervals is mapped to a constant. We thus obtain a map
$$
H^\bullet(LM, M) \overset{s^\bullet}{\longleftarrow} H^{\bullet+1}(\Map(\bigcirc_2), LM \coprod_M LM).
$$
Naturally $F := LM \coprod_M LM \to M$ is a fibration. We are in the situation of having a fibration $E \to M \times M$ together with a map $F \to E|_{M}$ and we seek to construct a map
$$
H^\bullet(E,F) \to H^{\bullet -n} (E|_{M}, F),
$$
by "intersecting with the diagonal".
Given a Thom class $Th \in H^n(M\times M, \FM_M(2))$, the authors of \cite{GoreskyHingston,HingstonWahl} 
construct such a map by multiplying a relative cochain in $C^\bullet(E,F)$ with a representative of the Thom class which has support in a tubular neighborhood of $M \subset U_\epsilon \subset M \times M$, thus obtaining a cochain in $C^{\bullet + n}( E|_{U_\epsilon}, F)$ and then composing with a retraction of the bundle $E|_{U_\epsilon}$ to $E|_M$.

We shall rather use an extension of \eqref{equ:example_diag}. Recall that the vertical homotopy cofibers in the diagram
$$
\begin{tikzcd}
M \times M & \ar{l} M \\
\FM_M(2) \ar{u} & \ar{l} \ar{u} UTM
\end{tikzcd}
$$
are both copies of the Thom space $DM / UTM$, and the induced map is an equivalence. We will again pull back the fibrations $E$ and $F$ along these maps, to obtain the following maps of pairs (we denote cofibers by ordinary quotients here)
$$
\begin{tikzcd}
(E,F) \ar{r} & (E / E|_{\FM_M(2)}, F / F|_{UTM} ) \\
& \ar{u}{\simeq} ( E|_M / E|_{UTM}, F / F|_{UTM} ).
\end{tikzcd}
$$
Now we are left with a pair of fibrations over the Thom pair $(M , UTM)$ and we can apply the Thom isomorphism, that is, we multiply with a Thom form $Th \in \Omega^\bullet(M, UTM)$. 
Note that we can identify $E|_M$ with $\Map(8)$ and under this identification, $F$ corresponds to the space of figure 8's with at least one ear constant. We define the loop coproduct to be the composite
\begin{equation}\label{equ:string coproduct zigzag}
\begin{tikzcd}
H^\bullet(LM,M)  & H^{\bullet+1}(\Map(\bigcirc_2), LM \coprod_M LM) \ar[l]& H^{\bullet+1-n}(\Map(8), F) \ar[l] & H^\bullet(LM,M)^{\otimes 2}[n-1] \ar[l]
\end{tikzcd}
\end{equation}
where the last map is induced by the map $\Map(8) \to LM \times LM$. On the "space"-level the diagram is given by
\begin{equation}
\label{eqn:defredcop}
\begin{tikzcd}
\frac{LM}{M} \ar[dashed]{r}{\text{suspend}} &\frac{I \times LM}{\partial I \times LM\cup I\times M} \ar[r, "s"] & \frac{\Map(\bigcirc_2)}{F} \ar[r] & \frac{\Map(\bigcirc_2) / \Map^\prime(8)}{ F/ F|_{UTM}} & \\
&&& \frac{\Map(8) / \Map^\prime(8)}{F / F|_{UTM}} \ar[u, "\simeq"] \ar[r, dashed, "Th"] & \frac{\Map(8)}{F},
\end{tikzcd}
\end{equation}
where we again wrote cofibers and cofibers of cofibers as fractions and pullbacks as restrictions. Note that the main part of this diagram is induced by maps on the base after taking fiber product with $E$ and $F$ over $M \times M$ and $M$, respectively. More concretely, we have the following diagram of pairs of pairs
$$
\begin{tikzcd}
\frac{M \times M}{M} \ar[r] & \frac{(M \times M, \FM_M(2) )}{( M, UTM)} & \\
& \frac{(M, UTM)}{(M, UTM)} \ar[u, "\simeq"] \ar[r, dashed, "Th"] & \frac{M}{M},
\end{tikzcd}
$$
where in the numerator we have pairs of spaces over $M \times M$ and in the denominator pairs of spaces over $M$. To get the previous diagram \eqref{eqn:defredcop}, one takes fiber product of the numerator with $E$ over $M\times M$ and of the denominator with $F$ over $M$ and realizes the corresponding cofibers.

\begin{Rem}
If one wishes to invert the above homotopy equivalence to write down an "actual" map, one needs to write down an inverse map of pairs of pairs that induces a homotopy inverse after taking the fiber products and taking the cofibers. Since we realize cofibers later, a "map" of pairs does not need to be defined everywhere. For instance, a "map" $M \times M \to (M, UTM)$ can be given by describing the map on a tubular neighborhood of the diagonal and providing a map to $UTM$ on the punctured tubular neighborhood. The problem is that the pair $(M, UTM)$ in the numerator is not fibrant as spaces over $M\times M$, thus we need to find a fibrant replacement, for instance $(M,UTM) \to (PM \times_M PM, PM\times_M UTM \times_M PM)$. Then a "map" $M \times M \to (PM \times_M PM, PM\times_M UTM \times_M PM)$ is obtained by connecting points that are close and providing the corresponding vector if they are not equal.
\end{Rem}

\subsection{String coproduct ($\chi(M) = 0$ case)}
\label{sec:framedcoprod}
Let $M$ be a closed manifold with trivialized Euler class. In particular, we can assume that the Thom class is represented by a fiberwise volume form on $UTM$, that is $Th \in C^{n-1}(UTM) \subset C^{n}(M, UTM)$ is closed. In this case, one can lift the coproduct to a map
$$
H^\bullet(LM) \longmapsfrom H^\bullet(LM)^{\otimes 2}[d-1]
$$
defined by the following zig-zag
$$
\begin{tikzcd}
H^\bullet(LM)  & \arrow{l}{s^*} H^{\bullet-1}(\Map(\bigcirc_2),\Map(8)) \arrow{d}{\simeq} \\
 & H^{\bullet-1}(\Map^\prime(\bigcirc_2),\Map^\prime(8))  &  \arrow{l}{\delta}  H^{\bullet-2}(\Map^\prime(8))  & \arrow[l, "\wedge Th"] H^{\bullet-n-1}(\Map(8)).
\end{tikzcd}
$$
It is induced by the following maps of spaces
$$
\begin{tikzcd}
(I, \partial I) \times LM \arrow[r, "s"] & (\Map(\bigcirc_2),\Map(8)) \\
 & (\Map^\prime(\bigcirc_2),\Map^\prime(8)) \arrow{u}{\simeq} \arrow{r} & \Sigma\Map^\prime(8) \arrow{r} & Th\Map(8),
\end{tikzcd}
$$
where the last map is induced by the Thom collapse along the embedding $M \to UTM$ (implicitly given by the trivialization of the Euler class). On chain level it is given by multiplying with the pullback of the fiberwise volume form $Th \in C^{n-1}(UTM)$ along the map $\Map'(8) \to UTM$.

The vertical excision isomorphism follows again from Lemma \ref{lem:ucolim} applied to the fibration $\Map(\bigcirc_2) = PM \times PM \to M \times M$, where this time we consider the horizontal cofibers.

The entire zig-zag except for the splitting map is obtained from the following zig-zag by taking fiber product with the fibration $\Map(\bigcirc_2) = PM \otimes PM \to M \times M$.
\begin{equation}
\label{equ:1frdefcop}
\begin{tikzcd}
(M \times M, M) \\
\ar[u] (\FM_M(2), UTM) \ar[r] & \Sigma UTM = (pt, UTM)  \ar[r, dashed, "Th"] & M
\end{tikzcd}
\end{equation}

\subsection{Definition of string bracket and cobracket}
\label{sec:stringbracketdef}
The original version of the string bracket and cobracket \cite{ChasSullivan,ChasSullivan2} was defined on the equivariant (co)homology of $LM$ relative to the constant loops, $H^\bullet_{S^1}(LM,M)$.
For our purposes, we consider a version of the definition using the string coproduct, provided essentially by Goresky and Hingston \cite[section 17]{GoreskyHingston}, see also \cite{HingstonWahl}.
To this end consider the $S^1$-bundle $\pi:LM\simeq LM\times ES^1\to LM_{S^1}$.
This gives rise to the Gysin long exact sequence for equivariant cohomology
\[
\cdots \to  H^\bullet(LM) \xrightarrow{\pi^*} H_{S^1}^\bullet(LM)
\to H_{S^1}^{\bullet+2}(LM) \xrightarrow{\pi_!}
H^{\bullet+1}(LM) \to \cdots
\]
One has a similar sequence for reduced (equivariant) cohomology.
Now, we define the string bracket (cohomology cobracket) operation (up to sign) as the composition 
\begin{equation}\label{equ:defbracket}
\bar H^\bullet_{S^1}(LM)
\xrightarrow{\pi^*}
\bar H^\bullet(LM)
\xrightarrow{\cdot}
(\bar H^\bullet(LM)\otimes \bar H^\bullet(LM))[n]
\xrightarrow{\pi_!\otimes \pi_!}
(\bar H^\bullet_{S^1}(LM)\otimes \bar H^\bullet_{S^1}(LM))[n-2].
\end{equation}
Here $\cdot$ is the string product. 

For the string cobracket (cohomology bracket) we similarly use the composition (up to sign)
\begin{equation}\label{equ:defcobracket}
\bar H^\bullet_{S^1}(LM)\otimes \bar H^\bullet_{S^1}(LM)
\xrightarrow{\pi^*\otimes \pi^*}
\bar H^\bullet(LM)\otimes \bar H^\bullet(LM)
\to 
 H^\bullet(LM,M)\otimes H^\bullet(LM,M)
\xrightarrow{*}
H^\bullet(LM,M)[n-1]
\to \bar H^\bullet(LM)[n-1]
\xrightarrow{\pi_!}
\bar H^\bullet_{S^1}(LM)[n-2],
\end{equation}
where the map $H^\bullet(LM,M) \to \bar{H}^\bullet(LM)$ uses that $M \subset LM$ is a retract.
\begin{Rem}
Note that the map $H^\bullet(LM,M) \to \bar H^\bullet(LM)$ does not depend on choosing a basepoint. More precisely, in our convention $\bar H^\bullet(LM) = H^\bullet( \text{pt}, LM)[1]$, i.e. it is the cohomology of the cofiber of the map $LM \to \text{pt}$ shifted by one (and not relative to a basepoint). The map above then comes from the fact that the long exact sequence associated to the cofiber diagram
$$
\begin{tikzcd}
\frac{LM}{M} \ar[r] &\frac{\text{pt}}{M} \ar[r] &  \frac{\text{pt}}{LM}
\end{tikzcd}
$$
splits since $\frac{\text{pt}}{M} \to \frac{\text{pt}}{LM}$ is a retract.
\end{Rem}
\todo[inline]{FlorianX: It seems it does not depend on basepoints.}

In the $\chi(M) = 0$ case, there is no need to work relative to constant loops and on reduced homologies. The bracket and cobracket are defined on $H^\bullet_{S^1}(LM)$ in this case.

\section{Cochain complex models}\label{sec:cochain models}
Having defined our version of the string topology operations on cohomology the goal of this section and the next is to introduce concrete cochain complexes that compute these cohomologies. Furthermore we will find the concrete maps on cochain complex that realize, for example, the zigzag \eqref{equ:product_zigzag} for the string product and its variant for the coproduct.
These chain complex models will eventually allow us to prove the main theorems \ref{thm:main_1}-\ref{thm:main_3} in later sections, by explicitly tracing cochains through the zigzags.

We also note that for non-simply-connected situation our "models" are not actually models, i.e., their cohomology is generally not the same as that of the loop spaces considered.
However, mind that in this situation (Theorem \ref{thm:main_3}) our only goal is to check that the map in one direction from the cyclic chains to the cohomology $\bar{H}_{S^1}(LM)$ respects the Lie bialgebra structure, not that the map is an isomorphism, which would be wrong.

\subsection{Iterated integrals and model for path spaces}\label{sec:it int}
Let for now $A := \Omega^\bullet(M)$ denote the algebra of differential forms on $M$. We denote by $B = B(A,A,A)$ the two sided bar construction, namely
$$
B(A,A,A) = \oplus_{n\geq 0} A \otimes \bar{A}[1]^{\otimes n} \otimes A.
$$
We recall that since $A$ is commutative, $B$ is a commutative dg algebra with the shuffle product. It is moreover a coalgebra in $A$-bimodules under the natural deconcatenation coproduct
\begin{align*}
B &\longrightarrow B \otimes_A B \\
\alpha = (\alpha_0 | \alpha_1 \dots \alpha_k | \alpha_{-1}) &\longmapsto \alpha' \otimes \alpha'' = \sum_i (\alpha_0 | \alpha_1 \dots \alpha_i | 1 | \alpha_{i-1} \dots \alpha_k | \alpha_{-1}),
\end{align*}
with counit $\epsilon: B \to A$.

Following Chen \cite{Chen}, we define
\begin{align*}
    ev_n \colon \Delta^n \times PM &\longrightarrow M \times M^{n} \times M \\
    ((t_1, \ldots, t_n), \gamma) &\longmapsto (\gamma(0), \gamma(t_1), \ldots, \gamma(t_n), \gamma(1)),
\end{align*}
where $\Delta^n = \{ (t_1, \ldots, t_n ) \ | \ 0 \leq t_1 \leq \ldots \leq t_n \leq 1\}$ is the standard simplex, and the fiber integral
$$
\int_{\Delta^n} \colon \Omega^\bullet( \Delta^n \times PM) \to \Omega^\bullet(PM).
$$
And hence
\begin{align*}
\int : B \to \Omega^\bullet(PM)\\
\int = \bigoplus_{n \geq 0} \int_{\Delta_n} \circ \ ev_n^*.
\end{align*}

\begin{Lem}
\label{lem:barpath}
The map $\int: B \to \Omega^\bullet(PM)$ is a quasi-isomorphism of algebras. Furthermore, the following diagrams commute
$$
\begin{tikzcd}
B \ar{r}{\int} & \Omega^\bullet(PM) \\
A \otimes A \ar{u} \ar{r} & \Omega^\bullet(M \times M) \ar{u}
\end{tikzcd}
\quad
\begin{tikzcd}
B \ar{d}{\epsilon} \ar{r}{\int} & \Omega^\bullet(PM) \ar{d}{const^*} \\
A  \ar{r} & \Omega(M)
\end{tikzcd}
\quad
\begin{tikzcd}
B \ar{d} \ar{r}{\int} & \Omega^\bullet(PM) \ar{d} \\
B \otimes_A B  \ar{r} & \Omega(PM \times_M PM)
\end{tikzcd}
$$
\end{Lem}

We may hence take the bar construction $B$ as our model for $PM$, in the sense that B is a dgca with a quasi-isomorphism to $\Omega(PM)$.

Furthermore, we shall later be flexible and replace $A=\Omega(M)$ by different dgca models for $M$, denoted also by $A$, slightly abusing notation. We can then still use the two-sided bar construction $B:=B(A,A,A)$ as our model of $PM$.
In practice, we shall not use the dgca structure, but only the $A$-bimodule structure on $B$, and use that $B$ is cofibrant as an $A$-bimodule.

\subsection{Splitting map}
\label{sec:splittingmap}
Let $\mathcal{I} = \R (1-t) \oplus \R t \oplus \R dt \subset \Omega^\bullet(I)$ denote the space of Whitney forms on the interval with projection
\begin{align*}
     \Omega^\bullet(I) &\longrightarrow \mathcal{I} \\
    f &\longmapsto f(0)(1-t) + f(1) t + (\int_I f) dt
\end{align*}

The splitting map $s$ (see \eqref{equ:splittingmap_def}) can be derived from a map
$$
I \times PM \to PM \times_M PM,
$$
given by the same formula which induces
$$
\Omega^\bullet(PM \times_M PM) \to \mathcal{I} \otimes \Omega^\bullet(PM).
$$
\begin{Prop}
The following diagram commutes
$$
\begin{tikzcd}
\mathcal{I} \otimes B \ar{d} & \ar{l} \ar{d} B \otimes_A B \\
\mathcal{I} \otimes \Omega^\bullet(PM) & \ar{l} \Omega^\bullet(PM \times_M PM),
\end{tikzcd}
$$
where
\begin{align*}
    B \otimes_A B &\longrightarrow \mathcal{I} \otimes B \\
    (x | \alpha | y | \beta | z) \longmapsto &(1-t) \otimes (\epsilon(x\alpha y)| \beta | z)  \\
    & + t \otimes (x | \alpha | \epsilon(y \beta z)) \\
    & + dt \otimes (-1)^{x \alpha y} (x | \alpha y \beta | z)
\end{align*}
\end{Prop}

\begin{proof}
The formula for the first two components follows from Lemma \ref{lem:barpath}. For the third component we note that the following diagram commutes
$$
\begin{tikzcd}
\Delta^n \times PM \times \Delta^m \times PM \ar{r} & M \times M^n \times M \times M \times M^m \times M \\
\Delta^n \times \Delta^m \times I \times PM \ar{u} \ar{r} \ar{d}& M \times M^n \times M \times M^m \times M \ar{u} \ar{d}\\
\Delta^n \star \Delta^m \times PM \ar{r} & M \times M^{m+n+1} \times M.
\end{tikzcd}
$$
\end{proof}

Note that the original splitting map \eqref{equ:splittingmap_def}
$$
I \times LM \longrightarrow \Map(\bigcirc_2)
$$
is obtained from $I \times PM \to PM \times_M PM$ by pulling back along the diagonal $M \to M \times M$. Hence we obtain the following

\begin{Prop}
\label{prop:splitmap}
The following diagram commutes.
$$
\begin{tikzcd}
\Omega^\bullet(LM) & \arrow{l}{s^*} \Omega^\bullet(\Map(\bigcirc_2)) \oplus \Omega^\bullet(F)[1]& \arrow{l} \Omega^\bullet(\Map(\bigcirc_2)) \oplus \Omega^\bullet(\Map(8))[1] \\
B \otimes_{A^e} A \arrow{u}& \arrow{u} \arrow{l} (B \otimes_A B) \otimes_{A^e} A \ \oplus \ ((B \overset{A}{\oplus} B) \otimes_{A^{\otimes 4}} A)[1]& \arrow{u} \arrow{l} (B \otimes_A B) \otimes_{A^e} A \ \oplus \ ((B \otimes B) \otimes_{A^{\otimes 4}} A)[1]
\end{tikzcd}
$$
\end{Prop}
where the map on the lower left row is given by
\begin{equation}\label{equ:splitmap cochains}
\begin{aligned}
    B \otimes_{A^e} A  &\longleftarrow (B \otimes_A B) \otimes_{A^e} A \quad \oplus \quad ((B \overset{A}{\oplus} B) \otimes_{A^{\otimes 4}} A)[1] \\
    \pm (\alpha x \beta | y) &\longmapsfrom ((\alpha | x | \beta | y), 0) \\
    (\alpha |x ) - (\beta | x) &\longmapsfrom (0, (\alpha \otimes 1 + 1 \otimes \beta | x))
\end{aligned}
\end{equation}
and on the lower right is the natural projection $B \otimes B \to (B \otimes A) \oplus (A \otimes B) \to B \overset{A}{\oplus} B$

\begin{Rem}
The vertical maps are quasi-isomorphisms in the simply-connected case.
\end{Rem}

\subsection{Connes differential}\label{sec:Connesdiff}
The action of $S^1$ on $LM = PM \times_{M \times M} M$ is induced by the splitting map $s$ as follows.
$$
\begin{tikzcd}
I \times PM \times_{M \times M} M \ar[r, "s"] & ( PM \times_M PM) \times_{M \times M} M \ar[r, "m^\text{op}"] & PM \times_{M \times M} M, 
\end{tikzcd}
$$
where $m^\text{op}$ is concatenation in the opposite order. That is rotating a loop is the same as splitting it at every point and concatenating the two pieces in the opposite order. Under the map $B \otimes_{A^e} A \to \Omega^\bullet(LM)$ the action of the fundamental class of the circle is then computed as
$$
\begin{tikzcd}
B \otimes_{A^e} A \ar[r, "\Delta^\text{op}"] & (B \otimes_A B) \otimes_{A^e} A \ar[r] & B \otimes_{A^e} A \\
(\alpha | x)  \ar[r, mapsto] & \pm (\alpha''|x|\alpha'|1) \ar[r, mapsto] & \pm (\alpha'' x \alpha' | 1),
\end{tikzcd}
$$
which is the standard formula for Connes' B-operator.

\subsection{Space of figure eights}
\label{sec:figeight}
In the definition of the loop product we also need the map 
$$
\Map(8) \to LM
$$
that concatenates the two loops. Since this is induced by the map $PM \times_M PM \to PM$ which is modelled by the deconcatenation coproduct on $B$, we get that the map $\Map(8) \to LM$ is modelled by
\begin{align}
\label{equ:concatloop}
\begin{split}
(B \otimes_{A^e} A)  &\longrightarrow (B \otimes B) \otimes_{A^{\otimes 4}} A \\
(\alpha|x) &\longmapsto (\alpha' \otimes \alpha'' | x).
\end{split}
\end{align}

To get the usual coproduct we have to compose with the map
$$
\Map(8) \to LM \times LM,
$$
that reads off each ear separately. This is just $M \to M \times M$ fiber product with $LM \times LM$ and hence modelled by
\begin{align}
\label{equ:splitloop}
\begin{split}
(B \otimes_{A^e} A) \otimes (B \otimes_{A^e} A)  &\longrightarrow (B \otimes B) \otimes_{A^{\otimes 4}} A \\
(\alpha|x) \otimes (\beta|y) &\longmapsto \pm (\alpha \otimes \beta| xy).
\end{split}
\end{align}

\subsection{$S^1$-equivariant loops}
It has been shown by Jones \cite{Jones} that the $S^1$-equivariant cohomology of the loop space of a simply connected manifold can be computed via the negative cyclic homology of a dgca model $A$ of the manifold,
\[
HC^-_{-\bullet}(A)\xrightarrow{\cong} H^\bullet_{S^1}(LM).
\]
More precisely, let $(B \otimes_{A^e} A [[u]], d + uB)$ be the negative cyclic complex, where $B$ is Connes operator. Then there is a map $(B \otimes_{A^e} A [[u]], d + uB) \to \Omega^\bullet(LM \times_{S^1} ES^1)$ (for the standard bar resolution $LM \times_{S^1} ES^1$)  that induces an isomorphism on cohomology in the case where $M$ is simply connected. Moreover, the Gysin sequence induced by the fiber sequence $S^1 \to LM \to LM_{S^1}$ is modelled by
$$
\begin{tikzcd}
B \otimes_{A^e} A \ar[r, "B"] \ar[d] & B \otimes_{A^e} A[[u]] \ar[r, "\cdot u"] \ar[d] & B \otimes_{A^e} A[[u]] \ar[r, "u = 0"] \ar[d] & B \otimes_{A^e} A  \ar[d] \\
\Omega^{\bullet-1}(LM) \ar[r, "\pi_!"] & \Omega^\bullet(LM_{S^1}) \ar[r] & \Omega^{\bullet+2}(LM_{S^1}) \ar[r, "\pi^*"] & \Omega^{\bullet+2}(LM),
\end{tikzcd}
$$
(see for instance \cite[Theorem 4.3]{Loday2}).
Let us moreover recall the following result
\begin{Prop}\label{prop:cycqiso}
Let $A$ be a connected graded algebra. Then
$$
B \otimes_{A^e} A \oplus \R[[u]] \overset{B}{\longrightarrow} (B \otimes_{A^e} A [[u]], d + uB)$$
factors through cyclic coinvariants $B \otimes_{A^e} A \to \Cyc(\bar{A})$ and induces a quasi-isomorphism
$$
\bCyc(\bar{A}) \oplus \R[[u]] = \Cyc(\bar{A}) \oplus u\R[[u]] \overset{B}{\longrightarrow} (B \otimes_{A^e} A [[u]], d + uB).
$$
\end{Prop}
\begin{proof}
This can for instance be extracted from \cite{Goodwillie} who shows that periodic cyclic homology is that of a point in this case and hence negative cyclic is essentially cyclic homology which can be computed in terms of cyclic words (see for instance \cite{Loday}).
\end{proof}
Under this quasi-isomorphism we get the following descriptions for $\pi_!$ and $\pi^*$ of section \ref{sec:stringbracketdef}.
\begin{Lem}
\label{lem:equivtohoch}
The following diagrams commute
$$
\begin{tikzcd}
B \otimes_{A^e} A \ar[r, "pr"] \ar[d] & \Cyc(\bar{A}) \oplus u\R[[u]] \ar[d] \\
\Omega^{\bullet}(LM) \ar[r, "\pi_!"] & \Omega^{\bullet+1}(LM_{S^1})
\end{tikzcd}
\quad
\begin{tikzcd}
\Cyc(\bar{A}) \oplus u\R[[u]] \ar[r, "\iota"] \ar[d] & B \otimes_{A^e} A \ar[d] \\
\Omega^\bullet(LM_{S^1}) \ar[r, "\pi^*"] & \Omega^{\bullet}(LM),
\end{tikzcd}
$$
where $pr : B \otimes_{A^e}A \to \Cyc(\bar{A})$ is the natural projection (where the element $1$ is sent to the empty cyclic word), and $\iota: \Cyc(\bar{A}) \to $ sends a cyclic word $x_1\dots x_n \mapsto \sum_i \pm (x_{1+i}x_{2+i}\ldots x_{n+i} | 1)$.
\end{Lem}

We also consider the following variant of this. One has the reduced equivariant cohomology $\bar H_{S^1}(LM)$, which fits into a split short exact sequence
\[
0 \to  \bar H_{S^1}(LM) \to H^\bullet_{S^1}(LM) \to H^\bullet_{S^1}(*) \to 0.
\]

By naturality of the above construction we get a natural map
\[
\bCyc(\bar A) \to \bar H_{S^1}(LM),
\]
from the reduced cyclic complex, which is a quasi-isomorphism for $M$ simply-connected (see also \cite{ChenEshmatovGan}).

\section{Cochain zigzags for the string product and bracket}\label{sec:cochain zigzags}
In this section we shall use the cochain models of section \ref{sec:cochain models} to obtain explicit descriptions of the string product and coproduct, as defined in section \ref{sec:stringtopology}.

\subsection{Product}
The zigzag \eqref{equ:product_zigzag} defining the string product on cohomology is realized on cochains by the zigzag
\[
\begin{tikzcd}
\Omega^\bullet(LM\times LM) & \ar{l} \Tot\left( \Omega^\bullet(LM\times LM) \to \Omega^\bullet(LM\times' LM) \right) \ar{d}{\simeq} & \\
& 
\Tot\left( \Omega^\bullet(\Map(8)) \to \Omega^\bullet(\Map'(8)) \right)
& \ar{l} \Omega^{\bullet -n}(\Map(8))
& \ar{l} \Omega^{\bullet -n}(LM)
\end{tikzcd}
\]
We first rewrite this to the zigzag
\[
\begin{tikzcd}
\Omega^\bullet(LM)\otimes \Omega^\bullet(LM) & \ar{l} \Tot\left( \Omega^\bullet(LM)\otimes \Omega^\bullet(LM) \to (\Omega^\bullet(LM)\otimes \Omega^\bullet(LM))\otimes_{A^{\otimes 2}} \Omega^\bullet(\FM_M(2)) \right) \ar{d}{\simeq} & \\
& 
\Tot\left( \Omega^\bullet(LM)\otimes_A \Omega^\bullet(LM) \to \Omega^\bullet(LM) \otimes_A \Omega^\bullet(LM) \otimes_A \otimes \Omega^\bullet(UTM) \right)
& \ar{l} \Omega^\bullet(LM) \otimes_A \Omega^\bullet(LM)[-n]
&  \\
& &  \ar{u} \Omega^\bullet(LM)[-n],
\end{tikzcd}
\]
which clearly comes with a map to the original zigzag.

Now suppose we have a model for the boundary inclusion $UTM\to \FM_M(2)$ compatible with the maps to $M \times M$, i.e., we have a commutative diagram of $A$-bimodules
\[
\begin{tikzcd}
C \ar{r} \ar{d}{\simeq} & U \ar{d}{\simeq} \\
\Omega^\bullet(\FM_M(2)) \ar{r} & \Omega^\bullet(UTM). 
\end{tikzcd}
\]
Then, using the models of the previous section, we can rewrite the diagram again to 
\[
\begin{tikzcd}
B\otimes_{A^2} A \otimes B\otimes_{A^2} A  
& \ar{l} \Tot\left(B\otimes_{A^2} A \otimes B\otimes_{A^2} A
\to (B\otimes B)\otimes_{A^{\otimes 4}} C \right) \ar{d}{\simeq} & \\
& 
\left( B\otimes_{A^4} B \to B \otimes_{A^2} B \otimes_{A^2}U \right)
& \ar{l}{\wedge Th} B\otimes_{A^4} B[-n]
& \ar{l} B\otimes_{A^2} A [-n] \, ,
\end{tikzcd}
\]
and again we retain a map to the original diagram.
Now the left- and right-hand end of the diagram are Hochschild complexes.
We can hence compute the string product on the Hochschild homology $HH(A,A)$ by starting with a cocycle on the right and tracing its image through the zigzag.
The main difficulty here is however crossing the vertical map.
We shall hence further simplify the zigzag slightly.

To this end note that our zigzag is obtained from the zigzag of $A$-bimodules 
\begin{equation}\label{equ:prod_mod_zigzag}
\begin{tikzcd}
A \otimes A   
& \ar{l} A\otimes A/C
\ar{d}{\simeq} & \\
& 
A/U
& \ar{l}{\wedge Th} A[-n]
\end{tikzcd}
\end{equation}

by tensoring with $B^{\otimes 2}$ over $A^{\otimes 4}$.

We will then use the following result, which is obvious, but nevertheless stated for comparison with the later Lemma \ref{lem:hocop}.
\begin{Lem}\label{lem:hoprod}
Suppose that $QA\to A$ is a cofibrant resolution of $A$ as a bimodule and 
\[
\begin{tikzcd}
QA[d] \ar{d}\ar{r}{g} & A\otimes A/C \ar{d} \\
A[-n] \ar{r}{\wedge Th} \ar[Rightarrow, ur, "h"] & A/U
\end{tikzcd}
\]
is a homotopy commutative square. Then the zigzag 
\[
A\otimes A/C \rightarrow A/U \xleftarrow{\wedge Th} A[-n]
\]
is homotopic to the zigzag
\[
A\otimes A/C \xleftarrow{g} QA[-n]\to A[-n]  \, .
\]
\end{Lem}
Note that eventually we are interested in the composition \[
A[-n]\leftarrow QA[-n] \rightarrow A\otimes A/C \to A\otimes A,
\]
which one may interpret as a (derived) coproduct.

\subsection{Coproduct (relative to constant loops case)}\label{sec:coprodrel}
Taking differential forms of the right part of the diagram \eqref{eqn:defredcop} we obtain
$$
\begin{tikzcd}
\Tot( \Omega^\bullet(E) \to \Omega^\bullet(F)) & \ar{l} \Tot\left(    \begin{tikzpicture}
        \node(1) at (0,.5) {$\Omega^\bullet(E)$};
        \node(2) at (2.2,.5) {$\Omega^\bullet(F)$};
        \node(3) at (0,-.5) {$\Omega^\bullet(E|_{\FM_M(2)})$};
        \node(4) at (2.2,-.5) {$\Omega^\bullet(F|_{UTM})$};
        \draw[->](1)-- (2);
        \draw[->](3)-- (4);
        \draw[->](1)-- (3);
        \draw[->](2)-- (4);
    \end{tikzpicture} \right) \ar{d}{\simeq}
&  \\
& \Tot\left(    \begin{tikzpicture}
        \node(1) at (0,.5) {$\Omega^\bullet(E|_M)$};
        \node(2) at (2.2,.5) {$\Omega^\bullet(F)$};
        \node(3) at (0,-.5) {$\Omega^\bullet(E|_{UTM})$};
        \node(4) at (2.2,-.5) {$\Omega^\bullet(F|_{UTM})$};
        \draw[->](1)-- (2);
        \draw[->](3)-- (4);
        \draw[->](1)-- (3);
        \draw[->](2)-- (4);
    \end{tikzpicture} \right) & \ar{l}{\wedge Th} \Tot( \Omega^\bullet(E|_M) \to \Omega^\bullet(F))[d]
\end{tikzcd}
$$
Here $\Tot(\cdots)$ refers to the total complex of the diagram, for example $\Tot( \Omega^\bullet(E) \to \Omega^\bullet(F))$ is the mapping cone of the map of complexes $\Omega^\bullet(E) \to \Omega^\bullet(F)$.
By replacing fiber products with tensor product we obtain a map from the following diagram into the above diagram (that is a quasi-isomorphism in the simply connected situation at each step).
\begin{equation}
\label{equ:zigzag_coprod2}
\begin{tikzcd}
\Tot( \Omega^\bullet(E) \to \Omega^\bullet(F)) & \ar{l}
\Tot\left(    \begin{tikzpicture}
        \node(1) at (0,.5) {$\Omega^\bullet(E)$};
        \node(2) at (3.5,.5) {$\Omega^\bullet(F)$};
        \node(3) at (0,-.5) {$\Omega^\bullet(E) \otimes_{A^{\otimes 2}} \Omega^\bullet(\FM_M(2))$};
        \node(4) at (3.5,-.5) {$\Omega^\bullet(F) \otimes_A \Omega^\bullet(UTM)$};
        \draw[->](1)-- (2);
        \draw[->](3)-- (4);
        \draw[->](1)-- (3);
        \draw[->](2)-- (4);
    \end{tikzpicture} \right) \ar{d}{\simeq}
&  \\
&\Tot\left(  \left(\Omega^\bullet(E)\otimes_{A^{\otimes 2}} A \to \Omega^\bullet(F) \right) \otimes_A \left(\begin{tikzpicture}
        \node(1) at (0,.5) {$A$};
        \node(3) at (0,-.5) {$\Omega^\bullet(UTM)$};
        \draw[->](1)-- (3);
    \end{tikzpicture} \right) \right) & \ar{l}{\wedge Th} \Tot\left(\Omega^\bullet(E)\otimes_{A^{\otimes 2}} A \to \Omega^\bullet(F)[n]\right)
\end{tikzcd}
\end{equation}
This diagram is obtained from
\begin{equation}
\label{diag:mor}
\begin{tikzcd}
A[-n] \ar[d, equal] \ar{r}{\wedge Th} & \ar[d, equal] A / \Omega^\bullet(UTM) & \ar{d} \ar{l}{\simeq} (A \otimes A) / \Omega^\bullet(\FM_M(2))  \ar{r} & A \otimes A \ar[d]\\
A[-n] \ar{r}{\wedge Th} & A / \Omega^\bullet(UTM) & \ar[l,equal] A / \Omega^\bullet(UTM) \ar{r} & A
\end{tikzcd}
\end{equation}
by tensoring the first line with $\Omega^\bullet(E)$ over $A^{\otimes 4}$ and the second line with $\Omega^\bullet(F)$ over $A^{\otimes 2}$. Here we again wrote (co)cones as quotients. In particular, if we wish to invert the quasi-isomorphism $A / \Omega^\bullet(UTM) \overset{\simeq}{\longrightarrow} (A \otimes A) / \Omega^\bullet(\FM_M(2))$ we have to do so as $A^{\otimes 4}$ module and compatible with the projection to $A / \Omega^\bullet(UTM)$. Since all our $A$-module structures come in pairs that factor through a single $A$-module structure it is enough to talk about $A$-bimodules and $A$-modules.

Let us summarize the above situation as follows. Let us define the category $\MAAA$ with objects pairs $(M \to N)$ of an $A$-bimodule $M$ and an $A$-module $N$, together with an $A$-bimodule map between them. Morphisms are homotopy commuting squares, i.e., squares
\begin{equation}\label{equ:MAAAmorph}
\begin{tikzcd}
M \ar{r}{f} \ar{d}{\pi} & M' \ar{d}{\pi'}\\
N \ar[Rightarrow, ur, "h"] \ar{r}{g}& N' 
\end{tikzcd},
\end{equation}
where $f$ is a map of $A$-bimodules, $g$ is a map of $A$-modules and $h:M\to N'[1]$ is a map of bimodules, such that $d_{N'}h+hd_M=\pi'f-g\pi$.

We define a functor from $\MAAA$ into dg vector spaces
\[
T_{F,E} : \MAAA \to \dgVect
\]
that sends an object $(M \overset{f}{\to} N)$ to the complex
$$
\Omega^\bullet(E) \otimes_{A^{\otimes 2}} M \oplus \Omega^\bullet(F) \otimes_A N [1],
$$
with differential of the form
$$
d = \begin{pmatrix}
d & 0 \\
f & d
\end{pmatrix}
=
\begin{pmatrix}
d_{\Omega^\bullet(E)} \otimes 1 + 1 \otimes d_M & 0 \\
\iota_{\Omega^\bullet(E) \to \Omega^\bullet(F)} \otimes f & d_{\Omega^\bullet(F)} \otimes 1 + 1 \otimes d_N
\end{pmatrix}.
$$
This defines a functor. Concretely, to a morphism \eqref{equ:MAAAmorph} in $\MAAA$ we associate the morphism 
$$
\Omega^\bullet(E) \otimes_{A^{\otimes 2}} M \oplus \Omega^\bullet(F) \otimes_A N [1]
\to \Omega^\bullet(E) \otimes_{A^{\otimes 2}} M' \oplus \Omega^\bullet(F) \otimes_A N' [1]
$$
given by the matrix
$$
\begin{pmatrix}
\iid_\Omega^\bullet(E) \otimes f & 0\\
\iota_{\Omega^\bullet(E) \to \Omega^\bullet(F)}\otimes h & \iid_\Omega^\bullet(F) \otimes g 
\end{pmatrix}\, .
$$
We note in particular, that the homotopy in the morphism \eqref{equ:MAAAmorph} in $\MAAA$ is part of the data and appears non-trivially in the image under $T_{F,E}$.
The functor sends componentwise quasi-isomorphisms into quasi-isomorphisms (we note that $\Omega^\bullet(E)$ and $\Omega^\bullet(F)$ are cofibrant as $A^{\otimes 2}$ and $A$ modules, respectively). 
Our zigzag \eqref{equ:zigzag_coprod2} of quasi-isomorphisms in $\dgVect$ is obtained from the zigzag \eqref{diag:mor} in $\MAAA$ (with all homotopies $=0$) by applying the functor $T_{F,E}$.

The category $\MAAA$ can in fact be extended to a dg category and $T_{F,E}$ to a dg functor.
Hence one can talk about homotopic morphisms in $\MAAA$, and homotopic morphisms are send to homotopic morphisms between complexes by $T_{F,E}$. Also, one can consider zigzags of quasi-isomorphisms such as \eqref{diag:mor} as a morphism in the derived category.

Our goal is then to replace the zigzag \eqref{diag:mor} by a homotopic zigzag in $\MAAA$ that is computationally simpler.
To this end we will use later the following result.

\begin{Lem}
\label{lem:hocop}
Let for that purpose $QA \to A$ denote a cofibrant replacement of $A$ in $A^{\otimes 2}$-modules.
Let $g$ and $h$ be any $A$-bimodule maps forming a homotopy commuting square
\begin{equation}\label{equ:hocopsquare}
\begin{tikzcd}
QA[-n] \ar[d] \ar[r, "g"] & (A \otimes A) / \Omega^\bullet(\FM_M(2)) \ar[d]  \\
A[-n] \ar[r, "\wedge Th"] \ar[Rightarrow, ur, "h"] & A / \Omega^\bullet(UTM),
\end{tikzcd}
\end{equation}
then
$$
\begin{tikzcd}
A[-n] \ar[d, equal] & \ar[l] QA[-n] \ar[d] \ar[r, "g"] & (A \otimes A) / \Omega^\bullet(\FM_M(2)) \ar[d] \\
A[-n] & \ar[l, equal] A[-n] \ar[Rightarrow, ur, "h"] \ar[r, "\wedge Th"] & A / \Omega^\bullet(UTM),
\end{tikzcd}
$$
and
$$
\begin{tikzcd}
A[-n] \ar[d, equal] \ar{r}{\wedge Th} & \ar[d, equal] A / \Omega^\bullet(UTM) & \ar{d} \ar{l}{\simeq} (A \otimes A) / \Omega^\bullet(\FM_M(2))  \\
A[-n] \ar{r}{\wedge Th} & A / \Omega^\bullet(UTM) & \ar[l,equal] A / \Omega^\bullet(UTM)
\end{tikzcd}
$$
define homotopic morphisms.
\end{Lem}
\begin{proof}
It is enough to show that the composites 
$$
\begin{tikzcd}
QA[-n] \ar[d] \ar[r, "g"] & A \otimes A / \Omega^\bullet(\FM_M(2)) \ar[d] \ar{r}{\simeq}& \ar{d}  A / \Omega^\bullet(UTM) \\
A[-n] \ar[Rightarrow, ur, "h"] \ar[r, "\wedge Th"] & A / \Omega^\bullet(UTM)  & \ar[l,equal] A / \Omega^\bullet(UTM),
\end{tikzcd}
$$
and 
$$
\begin{tikzcd}
QA[-n] \ar[r]  \ar[d]&A[-n] \ar[d, equal] \ar{r}{\wedge Th} & \ar[d, equal] A / \Omega^\bullet(UTM)   \\
A[-n] \ar[r, equal] &A[-n] \ar{r}{\wedge Th} & A / \Omega^\bullet(UTM)
\end{tikzcd}
$$
are homotopic. A homotopy is given by $H := \begin{pmatrix}h & 0 \\ 0 & 0\end{pmatrix}$ as shown by the computation
\begin{align*}
[d,H] &= \begin{pmatrix} d & 0 \\ 1 & d \end{pmatrix} \begin{pmatrix}h & 0 \\ 0 & 0 \end{pmatrix} - \begin{pmatrix}h & 0 \\ 0 & 0 \end{pmatrix} \begin{pmatrix} d & 0 \\ \pi & d \end{pmatrix} \\
&= \begin{pmatrix} [d,h] & 0 \\ h & 0 \end{pmatrix}= \begin{pmatrix} mg-(\wedge Th) \pi & 0 \\ h & 0 \end{pmatrix} \\
&= \begin{pmatrix} mg & 0 \\ h & \wedge Th \end{pmatrix} - \begin{pmatrix} (\wedge Th) \pi & 0 \\ 0 & \wedge Th \end{pmatrix} 
\end{align*}
where $m$ denotes the map $A \otimes A / \Omega^\bullet(C_2(M)) \to A / \Omega^\bullet(UTM)$ and $\pi: QA \to A$ is the canonical projection and we used the assumption
$$
[d,h] = m g - (\wedge Th)\pi.
$$
\end{proof}

We remark that the final datum that enters \eqref{diag:mor} is the homotopy commuting square
$$
\begin{tikzcd}
QA[-n] \ar[d] \ar[r] & A \otimes A \ar[d]  \\
A[-n] \ar[Rightarrow, ur] \ar[r] & A,
\end{tikzcd}
$$
obtained by postcomposing \eqref{equ:hocopsquare} with the maps from the relative to the non-relative complexes. We note that the lower horizontal map is multiplication with the Euler element. The needed data is thus a bimodule map
$$
g: QA[-n] \to A \otimes A,
$$
such that
$$
QA[-n] \to A \otimes A \to A
$$
is homotopic (with a specified homotopy $h$) to a map that descends along $QA \to A$ to an $A$-module map $A[-n] \to A$.

\begin{Rem}
Note that for any cyclic $A_\infty$ (or right Calabi-Yau) algebra $A$ there is a canonical "coproduct" map $A \to A \otimes A$ in the derived category of $A$-bimodules from which one constructs an "Euler class" $A \to A$ by postcomposing with the multiplication. This is an element in Hochschild cohomology $HH^d(A,A)$. In our case, this element is in the image of the map $Z(A) \to HH^\bullet(A,A)$. In general, this gives an obstruction for $A$ to come from a manifold, which is an element in $HH^{d,\geq 1}(A,A)$. In case this element vanishes, choices of lifts are given by elements in $HH^{d-1}(A,A) = HH_1(A,A)^*$, so that in case $M$ is 2-connected, there is no room for additional data. We will see below that we can trivialize the element in $HH^{d,\geq 1}(A,A)$ using $Z_1$ thus reducing the space of additional data to $HC^-_1(A,A)^*$ which is zero for $M$ simply-connected.
\todo[inline]{ThomasX: Where exactly do we see that?}
\end{Rem}

\subsection{Coproduct ($\chi(M) = 0$ case)}\label{sec:chi 0 case}
In the 1-framed case, we recall that the coproduct is obtained by fiber product with the fibration $\Map(\bigcirc_2) = PM \otimes PM \to M \times M$ of the zig-zag
$$
\begin{tikzcd}
(M \times M, M) \\
\ar[u] (\FM_M(2), UTM) \ar[r] & \Sigma UTM = (pt, UTM)  \ar[r, dashed, "Th"] & M.
\end{tikzcd}
$$
and precomposing with the splitting map $(I, \partial I) \times LM \to (\Map(\bigcirc_2), \Map(8))$.
In this case the Thom form can be lifted to a closed element $Th \in \Omega^{n-1}(UTM)$. A cochain description can thus be obtained from the $A^{\otimes 2}$-bimodule map
$$
(A \otimes A)/A \longrightarrow \Omega^\bullet(\FM_M(2))/ \Omega^\bullet(UTM) \overset{\wedge Th}{\longleftarrow} A,
$$
by tensoring with $B \otimes B$ over $A^{\otimes 4}$. Thus to get a description, we merely need to find a homotopy inverse to the map $(A \otimes A)/A \leftarrow \Omega^\bullet(\FM_M(2))/ \Omega^\bullet(UTM)$ on the image of $\wedge Th$. More concretely, we have the following
\begin{Lem}
\label{lem:frcop}
Assume $\Psi : QA \to (A \otimes A)/A$ and a homotopy making the following diagram homotopy commute
$$
\begin{tikzcd}
QA \ar[r, "\Psi"] \ar[dr, "\wedge Th"', ""{name=Th}] & (A \otimes A) / A \ar[d] \arrow[Rightarrow, to=Th] \\
& \Omega^\bullet(\FM_M(2))/ \Omega^\bullet(UTM),
\end{tikzcd}
$$
in $A$-bimodules are given. Then the diagram
$$
\begin{tikzcd}
(B \otimes B) \otimes_{A^{\otimes 4}} (A \otimes A / A) \ar[d] & \ar[l, "\Psi"] (B \otimes B) \otimes_{A^{\otimes 4}} QA \ar[d] \ar[r, "\simeq"] & (B \otimes B) \otimes_{A^{\otimes 4}} A \ar[dl] \\
\Omega^\bullet(\Map(\bigcirc_2), \Map(8)) & \ar[l] \Omega^\bullet(\Map(8)))
\end{tikzcd}
$$
commutes.
\end{Lem}
\todo[inline]{I changed this a bit. Check for repetition later. Decide whether to mention shifts.}

\section{Simply-connected case, and proofs of Theorems \ref{thm:main_1} and \ref{thm:main_2}, and the first part of Theorem \ref{thm:main_4}} 
\label{sec:thm12proofs}
Let us assume that $A$ is a Poincaré duality model for $M$. For $M$ simply-connected this exists by \cite{LambrechtsStanley}. In that case we obtain the following models for $UTM$ and $\FM_M(2)$, cf. \cite{LambrechtsStanley3, Idrissi, CamposWillwacher},
\begin{align*}
cone( A \overset{ m \circ \Delta}{\to} A) & \simeq \Omega^\bullet(UTM) \\
cone( A \overset{\Delta}{\to} A \otimes A) & \simeq \Omega^\bullet(\FM_M(2)).
\end{align*}
The map $\Omega^\bullet(\FM_2(M))\to\Omega^\bullet(UTM)$ (restriction to the boundary) is modelled by the map 
\begin{align*}
cone( A \overset{\Delta}{\to} A \otimes A) 
&\to 
cone( A \overset{ m \circ \Delta}{\to} A)
\\
(a,b\otimes c) &\mapsto (a,bc),
\end{align*}
i.e., just by the multiplication $m:A\otimes A\to A$.

A representative of the Thom class in $A / \Omega^\bullet(UTM)$ is given by the pair $(m\circ\Delta(1), (1,0))$. In this case we can easily write down the maps $g$ and $h$ appearing in Lemma \ref{lem:hocop}. Namely, $h=0$ and the map $g$ is given by $x \mapsto (\Delta(x),x, 0)$, that is the following diagram (strictly) commutes
$$
\begin{tikzcd}
A[-n] \ar[r, "g"] \ar[dr, "\wedge Th"'] & A\otimes A/ \Omega^\bullet(\FM_M(2)) =  \Tot\left(    \begin{tikzpicture}
        \node(2) at (1.5,.5) {$A \otimes A$};
        \node(3) at (0,-.5) {$A$};
        \node(4) at (1.5,-.5) {$A \otimes A$};
        \draw[double equal sign distance](2) to (4);
        \draw[->](3) to node[above]{$\Delta$} (4);
    \end{tikzpicture} \right)
\ar[d] \\
&  A/ \Omega^\bullet(UTM)= \Tot\left(    \begin{tikzpicture}
        \node(2) at (1.5,.5) {$A$};
        \node(3) at (0,-.5) {$A$};
        \node(4) at (1.5,-.5) {$A$};
        \draw[double equal sign distance](2) to (4);
        \draw[->](3) to node[above]{$m\circ\Delta$} (4);
    \end{tikzpicture} \right),
\end{tikzcd}\qquad
\begin{tikzcd}
x \ar[mapsto, r, "g"] \ar[mapsto, dr, "\wedge Th"'] & \left(    \begin{tikzpicture}
        \node(2) at (1.5,.5) {$\Delta(x)$};
        \node(3) at (0,-.5) {$x$};
        \node(4) at (1.5,-.5) {$0$};
    \end{tikzpicture} \right)
\ar[mapsto, d] \\
&  \left(    \begin{tikzpicture}
        \node(2) at (1.5,.5) {$m\Delta(x)$};
        \node(3) at (0,-.5) {$x$};
        \node(4) at (1.5,-.5) {$0$};
    \end{tikzpicture} \right).
\end{tikzcd}
$$
Thus the resulting diagram \eqref{diag:mor} is equivalent to
\begin{equation}\label{equ:sec6diag}
\begin{tikzcd}
A[-n] \ar[d, equal]\ar[r, "\Delta"] & A \otimes A \ar[d] \\
A[-n] \ar[r, "m\Delta"] & A.
\end{tikzcd}
\end{equation}
Note that here we did not even need to pass to a cofibrant replacement $QA$ of the $A$-bimodule $A$, as in Lemma \ref{lem:hocop}. More pedantically, we can take, for example, $QA=B(A,A,A)=:B$, the bar resolution of the $A$-bimodule $A$ as in section \ref{sec:it int}, but then all maps from $QA$ in Lemma \ref{lem:hocop} factor through the canonical projection $QA\to A$.

From this we can compute the string coproduct (cohomology product).
To this end, we trace back and follow the zigzag \eqref{equ:string coproduct zigzag}.

\newcommand{\vE}{M_E}
\newcommand{\vF}{M_F}
\newcommand{\veight}{M_8}

Let us first make explicit all cochain complex models we use. Our model for $E=\Map(\bigcirc_2)\to M\times M$ (loops with two marked points) is 
\[
\vE = B\otimes_{A^{\otimes 2}} B \cong \bigoplus_{p,q\geq 0} (\bar A[1])^{\otimes p} \otimes A \otimes (\bar A[1])^{\otimes q} \otimes A
\]
considered as an $A^{\otimes 2}$-module, and equipped with the natural Hochschild-type differential.
Similarly, the cochain complex modelling $\Map(8)\to M$ is \[
\veight = B\otimes_{A^{\otimes 4}} B
\cong
\bigoplus_{p,q\geq 0} (\bar A[1])^{\otimes p} \otimes A \otimes (\bar A[1])^{\otimes q}.
\]
The subspace $F\subset \Map(8)$ (figure 8 loops with one ear trivial) is modelled by the quotient of $\veight$ by the summands with $p>0$ and $q>0$
\[
\vF =\veight / \bigoplus_{p,q> 0} (\bar A[1])^{\otimes p} \otimes A\otimes (\bar A[1])^{\otimes q}
\cong \bigoplus_{p,q\geq 0, pq=0} (\bar A[1])^{\otimes p} \otimes A\otimes (\bar A[1])^{\otimes q}.
\]
The cochain complex computing $H(LM)$ is the Hochschild complex
\[
C(A) = B\otimes_{A^{\otimes 2}} / A \cong \bigoplus_{k\geq 0} (\bar A[1])^k \otimes A.
\]
The cochain complex computing $H(LM,M)$ in this case is the reduced Hochschild complex 
\[
\Tot(C(A) \to A) \xleftarrow{\simeq} \bar C(A) :=\bigoplus_{k\geq 1}  (\bar A[1])^k \otimes A.
\]
We start with two cocycles 
\begin{align*}
a &=(\alpha_1,\cdots,\alpha_k, x)=:(\alpha,x) \\
b &=(\beta_1,\cdots,\beta_l, y)=:(\beta,y)
\end{align*} 
in this complex $\bigoplus_{k\geq 1} (\bar A[1])^k \otimes A$, representing cohomology classes in $H(LM,M)$. Our goal is to produce a cocycle representing their pullback under the string coproduct. 

First, consider the rightmost map in \eqref{equ:string coproduct zigzag}. 
The rightmost map in \eqref{equ:string coproduct zigzag} is realized on cochains by (cf \eqref{equ:concatloop})
\[
(a,b) \mapsto c_1 := \pm (\alpha, xy, \beta)\in 
\bigoplus_{p,q\geq 1} (\bar A[1])^{\otimes p} \otimes A \otimes  (\bar A[1])^{\otimes q} 
\xrightarrow{\simeq} 
\Tot(\veight \to \vF).
\]

Next consider the map $H^{\bullet+1-n}(\Map(8),F)\to H^{\bullet+1}(E,F)$, i.e., the middle map in \eqref{equ:string coproduct zigzag}. This map is, according to section \ref{sec:coprodrel} and the discussion in the beginning of this section, represented on cochains by the map 
\[
\Tot(\veight \to \vF) \to \Tot(\vE\to \vF) 
\]
obtained from diagram \eqref{equ:sec6diag} by tensoring the first row with $\vE$ over $A^{\otimes 2}$ and the second by $\vE$ over $A$.
Our cocycle $c_1$ above is hence mapped to the cocycle 
\[
c_2 := \sum \pm (\alpha, \D', \beta, x\D''y)
\in
\vE \subset \Tot(\vE\to \vF),
\]
where $\sum \pm \D' \otimes x\D''y = \Delta(\alpha_0 \beta_0)$.

Finally, we apply the pullback via the splitting map $s$ to obtain a map 
\[
H^\bullet(E,F) \to H^{\bullet -1}(LM,M).
\]
According to Proposition \ref{prop:splitmap} this map on cochains is given by formula \eqref{equ:splitmap cochains}. Applied to our cochain $c_2$ the final result reads
\[
\sum \pm (\alpha,\D',\beta, x\D''y) 
\in 
\bar C(A),
\]
in agreement with \eqref{equ:intro coproduct}.
Summarizing, we hence obtain Theorem \ref{thm:main_2}, i.e., the following result.
\begin{Thm}
For $M$ simply-connected, and $A$ a Poincaré duality model for $M$, the natural map
$$
\overline{HH}_\bullet(A,A) \to H^\bullet(LM, M),
$$
intertwines the string coproducts.
\end{Thm}

We can compute the string product on $H(LM)$ by a similar, albeit simpler computation. 
To this end we begin with a cochain in the Hochschild complex 
\[
(\alpha, x) \in C(A),
\]
representing a cohomology class in $H(LM)$.
We then trace this class through the zigzag \eqref{equ:product_zigzag}.
The image in our model $\veight$ for $H(\Map(8))$ is given by formula \eqref{equ:concatloop} as
\[
p_1 := (\alpha',x,\alpha'')
\in
\veight.
\]
The remaining maps of \eqref{equ:product_zigzag}, namely $H^{\bullet-n}(\Map(8))\to H^\bullet(LM\times LM)$ have been discussed in section \ref{sec:figeight}. According to Lemma \ref{lem:hoprod} and the discussion preceding it they are given on chains by tensoring the upper row ($g=\Delta$) of \eqref{equ:sec6diag} with $B\otimes B$ over $A^{\otimes 4}$ to obtain a map of cochain complexes
\[
\veight[-n] \to C(A)\otimes C(A).
\]
Concretely, on our cochain $p_1$ this produces the cochain
\[
p_2 = \sum (\alpha', x')\otimes (\alpha'', x'')
=
\sum (\alpha', x\D')\otimes (\alpha'', \D'').
\]
this agrees with \eqref{equ:intro product}.
It is furthermore obvious from the construction of the map via iterated integrals that the BV operator is also preserved (see Section \ref{sec:Connesdiff}).
Hence we have shown Theorem \ref{thm:main_1}.

\begin{proof}[Proof of first part of Theorem \ref{thm:main_4}]
Similarly, we can prove the first part of Theorem \ref{thm:main_4}. For the string product, the same computations as above work. For the coproduct we use the 
description of section \ref{sec:chi 0 case}.
Proceeding otherwise as in the relative case above we obtain the string coproduct (cohomology product) of the cocycles $(\alpha,x)$ and $(\beta,y)$ in $\bar C(A)$ as the composition:
\[
\begin{tikzcd}
B\otimes_{A^e}A \otimes B\otimes_{A^e}A
\ar{r}
&
(B\otimes B)\otimes_{A^{\otimes 4}} A
\ar{r}{\Delta}
&
(B\otimes B)\otimes_{A^{\otimes 4}} (A\otimes A/A)
\ar{r}{s}
&
B\otimes_{A^2}A \\
((x,\alpha), (y,\beta))
\ar[symbol=\in]{u}
\ar[mapsto]{r}
&
\pm (\bar\alpha,\bar\beta,xy)
\ar[symbol=\in]{u}
\ar[mapsto]{r}
&
\sum \pm (\alpha, \beta, \D'xy, D'', 0)
\ar[symbol=\in]{u}
\ar[mapsto]{r}
&
\sum \pm (\bar\alpha,\D'',\bar\beta,\D'\alpha_0\beta_0)
\ar[symbol=\in]{u}
\end{tikzcd}.
\]
\end{proof}

\section{Involutive homotopy Lie bialgebra structure on cyclic words}

\subsection{Lie bialgebra of cyclic words}
For a graded vector space $H$ with a non-degenerate pairing of degree $-n$ it is well known (see \cite{Gonzalez} for an elementary review) that the cyclic words 
\begin{equation}\label{equ:cycdef}
\Cyc(H^*)= \bigoplus_{k\geq 0} (H^*[-1])^{\otimes k}_{C_k}
\end{equation}
carry an involutive dg Lie bialgebra structure with bracket $[,]$ and cobracket $\Delta$ of degree $n-2$.
Concretely, one has the following formulas (cf. \cite[section 3.1]{ChenEshmatovGan})
\begin{align*}
\Delta(a_1\cdots a_k) &= \sum_{i<j}\pm \epsilon(a_i,a_j) (a_1\cdots a_{i-1}a_{j+1}\cdot a_k) \wedge (a_{i+1}\cdots a_{j-1}) \\
[(a_1\cdots a_k),(b_1\cdots b_m)]
&=
\sum_{i,j} \pm \epsilon(a_i,b_j)((a_{i+1}\cdots a_k a_1\cdots a_{i-1}b_{j+1}\cdots b_mb_1\cdots b_{j-1}))\, .
\end{align*}

There are also slight variants.
First, we have the quotient Lie bialgebra
\[
\bCyc(H^*) = \Cyc(H^*)/\K,
\]
dropping the $k=0$-term from the direct sum \eqref{equ:cycdef}.
Second, using the setting and notation of section \ref{sec:GCM}, we have the Lie sub-bialgebra
\[
\Cyc(\bar H^*)\subset \Cyc(H^*),
\]
dropping the span of $1$ from $H$.
Finally, we have the quotient Lie bialgebra
\[
\bCyc(\bar H^*) = \Cyc(\bar H^*) /\K.
\]

\subsection{Twisting by Maurer-Cartan elements}\label{sec:IBLtwist}
Next, it is furthermore well known that a (degree $n-2$-)involutive dg Lie bialgebra structure on $\fg$ can be encoded as a BV operator (up to degree shifts) on the symmetric algebra
\[
 S(\fg[n-3])[[\hbar]]
\]
with $\hbar$ a formal parameter of degree $6-2n$.
Concretely, on this space we then have a degree $+1$ operator
\[
\Delta_0 = \delta_{\fg} + \delta_{c} + \hbar \delta_b,
\]
where $\delta_{\fg}$ is the internal differential on $\fg$, $\delta_c$ applies the cobracket to one factor in the symmetric product and $\delta_b$ is an differential operator of order 2 applying the bracket to two factors.
All relations of the involutive Lie bialgebra can be compactly encoded into the single equation
\[
\Delta_0^2=0.
\]
More generally, such a BV operator encodes a homotopy involutive Lie bialgebra structure ($\IBL_\infty$-structure), see e.g. \cite[section 5.4]{CamposMerkulovWillwacher} for more details.

Now suppose we have an element
\[
z\in S(\fg[n-3])[[\hbar]]
\]
of degree $|z|=6-2n$.
Then we may twist our involutive Lie bialgebra structure to a different $\IBL_\infty$-structure encoded by the operator
\[
\Delta_z = e^{-\frac z \hbar}\Delta_0 e^{\frac z \hbar}
\]
if the master equation (Maurer-Cartan equation) 
\[
\Delta e^{\frac z \hbar} = 0 
\quad \Leftrightarrow \quad \delta_c z +\frac 1 2 [z,z]=0
\]
is satisfied. Concretely, in our setting
\[
\Delta_z = \Delta_0 +[z,-].
\]
The operator $[z,-]$ is always a derivation, and hence only contributes to the operations of arity $(k,1)$ (1 input and $k\geq 1$ outputs) of our $\IBL_\infty$-structure. In other words, the twisting changes the differential and the (possibly higher genus and higher arity) cobrackets.

\subsection{Example: Lie bialgebra structure of \cite{ChenEshmatovGan}}
Let now $H$ be a (degree $n$-)Poincar\'e-duality algebra, with the Poincar\'e duality encoded by the map 
\[
\epsilon : H\to \K
\]
of degree $-n$.
We want to apply the twisting construction of the previous section to the case $\fg=\bCyc(H^*)$.
Indeed, consider the degree 6-2n element encoding the product in $H$
\[
z=\sum_{p,q,r} \pm \epsilon(e_p e_q e_r) (e_p^*e_q^*e_r^*)
\in \fg[n-3] \subset S(\fg[n-3]) \subset S(\fg[n-3])[[\hbar]],
\]
where the $e_j$ range over a basis of $H$, and $e_j^*$ is the corresponding dual basis of $H^*$.
One can check that this is indeed a Maurer-Cartan element in the sense of the previous subsection.
Continuing to use the notation of the previous subsection the BV operator encoding the $\IBL_\infty$-structure is just
\[
\Delta_0 + [z,-].
\]
Since $z$ is linear (i.e., only has Terms in the $k=1$ summand of \eqref{equ:cycdef}), and since there is no $\hbar$-dependence, the twist in fact only changes the differential of our dg Lie bialgebra structure.
In fact, the altered differential is just the Hochschild (or rather cyclic) differential, rendering $\fg$ into the cyclic complex of the coalgebra $H^*$, computing its cyclic homology.
As usual, the normalized cyclic complex $\bCyc(\bar H^*)\subset \bCyc(H^*)$ is preserved under this differential and we hence have equipped it with an involutive dg Lie bialgebra structure.
This is the structure found in \cite{ChenEshmatovGan}.

\subsection{Ribbon graph complex}\label{sec:ribbonGC}
Recall the definition of the graph complex (dg Lie algebra) $\GC_H$ as in section \ref{sec:GCM}.
Note that the construction leading to $\GC_H$ can be re-iterated for ribbon (fat) graphs instead of ordinary graphs. 
Here a ribbon graph is a graph with a cyclic order prescribed on each star, i.e., on each set of incoming half-edges and decorations at a vertex.
We denote the corresponding graphical dg Lie algebra of ribbon graphs by $\GC_H^{As}$. There is a natural map of dg Lie algebras
\[
\GC_H' \to \GC_H^{\op{As}}
\]
by sending a graph to the sum of all ribbon graphs built from the original graph by imposing cyclic orders on stars.
In fact, this map can be seen as a version of naturality of the Feynman transform of cyclic operads, together with the natural map of cyclic operads $\op{As}\to \op{Com}$.

The main observation is now that there is a map of dg Lie algebras
\[
f: \GC_H^{\op{As}}\to S(\bCyc(H^*)[n-3])[[\hbar]]
\]
defined on a ribbon graph $\Gamma$ as follows.
\begin{itemize}
\item Note that the ribbon graph has an underlying surface whose genus we denote by $g$, and number of boundary components $b$.
\item If $\Gamma$ has vertices of valency $\neq 3$, we set $f(\Gamma)=0$.
\item Otherwise we set 
\[
f(\Gamma) = \pm \hbar^g (c_1)(c_2)\cdots (c_b),
\]
where $c_j$ is the cyclic word obtained by traversing the $j$-th boundary component counterclockwise, and recording the $H$-decorations.
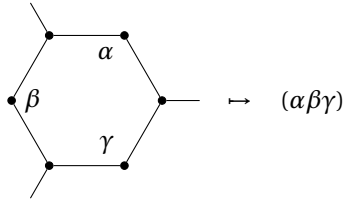
\[
\begin{tikzpicture}[baseline=-.65ex]
\node[int] (v1) at (0:1) {};
\node[int,label=240:{$\alpha$}] (v2) at (60:1) {};
\node[int] (v3) at (120:1) {};
\node[int,label=0:{$\beta$}] (v4) at (180:1) {};
\node[int] (v5) at (240:1) {};
\node[int,label=120:{$\gamma$}] (v6) at (300:1) {};
\draw (v1) edge +(0:.5) edge (v2)
      (v2)  edge (v3)
      (v3) edge +(120:.5) edge (v4) 
      (v4) edge (v5)
      (v5) edge +(240:.5) edge (v6)
      (v6)  edge (v1);
\end{tikzpicture}
\quad
\mapsto
\quad
(\alpha\beta\gamma)
\]
\end{itemize}

In particular, given an MC element $z+Z\in \GC_H'$, as in section \ref{sec:GCM}, we then obtain an MC element $$Z_{Cyc}\in S(\bCyc(H^*)[1])[[\hbar]]$$
via the composition of maps of dg Lie algebras
\[
\GC_H\to \GC_H^{\op{As}}\to S(\bCyc(H^*)[1])[[\hbar]].
\]

\subsection{$\IBL_\infty$-structure on cyclic words}
Now we use our MC element $Z_{Cyc}\in S(\bCyc(H^*)[1])[[\hbar]]$ constructed in the previous section to obtain an $\IBL_\infty$-structure on $\bCyc(H^*)$, by the twisting procedure of section \ref{sec:IBLtwist}.
Specifically the structure is defined as before by the BV operator 
\[
\Delta_{Z_{cyc}} = e^{-Z_{cyc}} \Delta_0 e^{Z_{cyc}}= \Delta_0 + [Z_{cyc}, -],
\]
where $\Delta_0$ is the (untwisted) BV operator as before.
Abusively, we shall denote our $\IBL_\infty$-algebra thus obtained by  $(\bCyc(H^*),\Delta_{Z_{cyc}})$.
We note that this procedure alters the differential on $\bCyc(H^*)$ and alters and extends the Lie cobracket by a series of higher cobrackets, i.e., operations of arity $(r,1)$. The twist however does not affect the Lie bracket, and there are no non-trivial operations of arities $(r,s)$ with $r,s\geq 2$ or with $s>2$.

Overall we obtain the following result
\begin{Prop}
$(\bCyc(H^*), \Delta_{Z_{cyc}})$ is an $IBL_\infty$-algebra whose differential is the natural "Hochschild" differential on the cyclic complex of $H^*$ as $\op{Com}_\infty$-algebra.
\end{Prop}

Finally, we claim that our $\IBL_\infty$-structure on $\bCyc(H^*)$ restricts to the normalized subspace $\bCyc(\bar H^*)$.
That means concretely that the operation $[Z_{cyc},-]$ cannot introduce terms that contain a cyclic word containing the letter $1^*\in H^*$.
Indeed, by the remark at the end of section \ref{sec:GCM} we know that the only pieces of $Z_{cyc}$ that contain the letter $1^*$ at all are those arising from the leading order part $z$ (as in \eqref{equ:GCHz}) of the Maurer-Cartan element.
However, this piece just controls the commutative product on $H$, and hence it leaves the reduced part $\bCyc(\bar H^*)$ invariant. (This is for the same reasons that the reduced Hochschild or cyclic complex is a subcomplex of the full Hochschild or cyclic complex.)
Overall, we obtain the following result, which yields the first statement of Theorem \ref{thm:main_3}.
\begin{Prop}
$(\bCyc(\bar H^*), \Delta_{Z_{cyc}})$ is an $IBL_\infty$-algebra whose cohomology is the reduced cyclic cohomology of the $\op{Com}_\infty$-algebra $H^*$.
\end{Prop}

\subsection{Graded Lie bialgebra structure}
Any $\IBL_\infty$-structure on a graded vector space $V$ induces an ordinary involutive Lie bialgebra structure on the cohomology $H(V)$.
In particular, from the the $\IBL_\infty$-structure on $\Cyc(H^*)$ from the previous subsection we obtain a graded involutive Lie bialgebra structure on the cyclic cohomology $H(\Cyc(H^*))$ of the $\Com_\infty$-algebra $H$.
The remaining statement of Theorem \ref{thm:main_3} (to be shown) is encoded in the following result.
\begin{Thm}\label{thm:main_3b}
The natural map $(H(\op{Cyc}(H))^* \to H^\bullet(LM_{S^1})$ is compatible with the Lie bialgebra structures.
\end{Thm}
The proof will occupy the next two sections, and proceeds by an explicit computation.
At this point, let us just make the Lie bialgebra structure on the left-hand side more explicit. First, the differential on the cyclic complex $\Cyc(H^*)$ is given by the genus zero ($\hbar^0$-) and $1$-boundary-component-part $z_0$ of the MC element $Z_{cyc}$.
This is just given by the tree piece, which is in turn given by the tree piece $Z_{tree}$ of $Z\in \GC_H$ encoding (only) the $\Com_\infty$-structure on $H$.
The Lie bracket is not altered, as we remarked above.
The Lie cobracket is altered, but only receives contributions from the genus-0 and 2-boundary-components-part $z_1$ of $Z_{cyc}$. This is just the 1-loop part, determined by the 1-loop part $Z_1$ of $Z\in \GC_H$, cf. \eqref{equ:Zsplit}.
There are no further corrections to the cohomology Lie bialgebra structure.

\section{Graphical version}\label{sec:graphical version}

Having proven Theorems \ref{thm:main_1}, \ref{thm:main_2} and the first part of Theorem \ref{thm:main_3}, it remains for us to show the remaining statement of Theorem \ref{thm:main_3}, or more precisely Theorem \ref{thm:main_3b}.
To do this we will use the explicit zigzags of section \ref{sec:cochain zigzags} defining the string (co)product (respectively the (co)bracket). However, in the non-simply connected situation we cannot use the Lambrechts-Stanley model for configuration spaces, unfortunately.
Hence we use the more complicated graphical models of configuration spaces of \cite{CamposWillwacher}.
In this section we shall introduce the specific models and some auxiliary results.
Finally, the proof of Theorems \ref{thm:main_3}, \ref{thm:main_3b} will be given in section \ref{sec:thm3proof}.

Technically, the goal of this section is to write down formulas for the homotopy commuting squares in Lemmas \ref{lem:frcop} and \ref{lem:hocop}.

\subsection{Graphical models for configuration spaces}
Recall that the main input to our construction of the string bracket and cobracket is the compactified configuration space of 2 points on $M$, $\FM_M(2)$, together with the boundary inclusion and the forgetful maps
\begin{equation}\label{equ:bdry_inclusion}
\begin{tikzcd}
UTM = \partial \FM_M(2) \ar{d}\ar[hookrightarrow]{r}& \FM_M(2) \ar{d}\\
M \ar{r}{\Delta} & M\times M
\end{tikzcd}\, .
\end{equation}

The goal of this subsection is to construct a real model for the objects and morphisms in this square. More precisely, we require a dgca $A$ quasi-isomorphic to $M$ and two $A\otimes A$-modules modelling the upper arrow in the diagram.

We note that the morphisms in the above square can be interpreted as the simplest non-trivial instances of the action of the little disks operad on the (framed) configuration space of points. Furthermore, combinatorial models (graph complexes) $\Graphs_M$ for configuration spaces of points, with the little disks action have been constructed in \cite{CamposWillwacher, CDIW}, from which a dgca model of the morphism \eqref{equ:bdry_inclusion} can be extracted, cf. section \ref{sec:GraphsM} above. 
However, in our situation their models can be much simplified, essentially by discarding all graphs of loop orders $\geq 2$.
Concretely we make the following definitions.

\begin{itemize}
    \item Our dgca model for $M$ will be the tree part of ${}^*\op{Graphs}_M(1)$, that is the space spanned by rooted (at least trivalent) trees, where each vertex is decorated by an element of $\op{Sym}\bar{H^\bullet}$. We denote this dgca by $A$. It can be identified with 
    $$
    A\cong \op{Com} \circ \op{coLie}_\infty \{1\} \circ \bar{H}^\bullet,
    $$
    that is the bar-cobar resolution of the $\op{Com}_\infty$ algebra $H^\bullet$.
    \item Our model for $\FM_M(2)$ will be the tree part of ${}^*\op{Graphs}_M(2)$, which we denote by $\mathcal{C}$. It has a natural decomposition as a graded vector space
    $$
    \mathcal{C} = B(A, A, A) \oplus A \otimes A,
    $$
    with an additional differential $d_s: B(A,A,A) \to A \otimes A$.
    The graphs in the two summands are schematically depicted as follows.
    \begin{align*}
    &
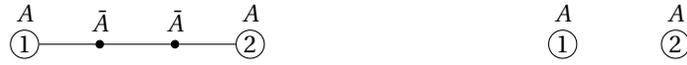
        \begin{tikzpicture}[baseline=-.65ex]
        \node[ext,label=90:{$A$}] (v1) at (0,0) {1};
        \node[ext,label=90:{$A$}] (v2) at (3,0) {2};
        \node[int,label=90:{$\bar A$}] (i1) at (1,0) {};
        \node[int,label=90:{$\bar A$}] (i2) at (2,0) {};
        \draw (v1) edge (i1) (i2) edge (i1) edge (v2);
        \end{tikzpicture}
    &
     \begin{tikzpicture}[baseline=-.65ex]
        \node[ext,label=90:{$A$}] (v1) at (0,0) {1};
        \node[ext,label=90:{$A$}] (v2) at (1.5,0) {2};
        \end{tikzpicture}
    \end{align*}
    where $A$ stands for forests of trees with decorations in $\bar H$ as above.
    The differential $d_s$ (from left to right in the picture) comes from the piece of \eqref{equ:edgecsplit} that cuts an edge.
    
    \item Our model for $UTM$ will be the subspace $\U\subset \Graphs_M(1)$ spanned by graphs of loop order $\leq 1$, where tadpoles are only allowed at the root. Concretely, such graphs can have the following shapes.
    \begin{align}\label{equ:Upics}
    & 
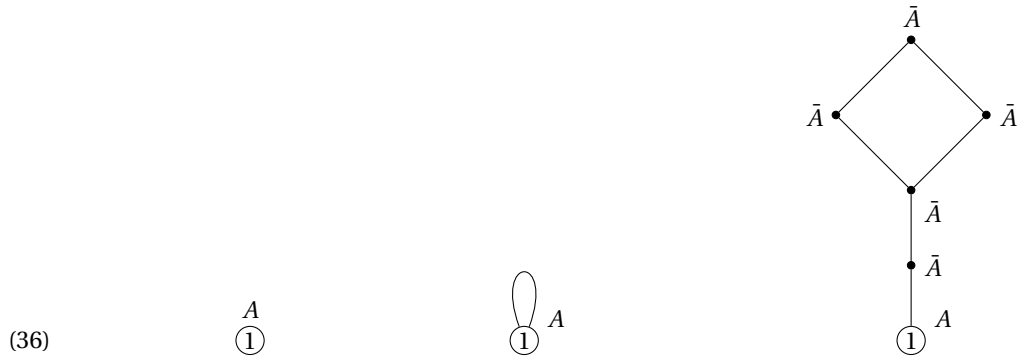
\begin{tikzpicture}[baseline=-.65ex]
        \node[ext,label=90:{$A$}] (v1) at (0,0) {1};
        \end{tikzpicture}  
    &
    &
    \begin{tikzpicture}[baseline=-.65ex, every loop/.style={draw,-}]
        \node[ext,label=20:{$A$}] (v1) at (0,0) {1};
        \draw (v1) edge[loop above, looseness=20, out=110,in=70] (v1);
    \end{tikzpicture}  
    &
    &
    \begin{tikzpicture}[baseline=-.65ex]
        \node[ext,label=20:{$A$}] (v1) at (0,0) {1};
        \node[int,label=0:{$\bar A$}] (i1) at (0,1) {};
        \node[int,label=-30:{$\bar A$}] (i2) at (0,2) {};
        \node[int,label=90:{$\bar A$}] (i4) at (0,4) {};
        \node[int,label=180:{$\bar A$}] (i3) at (-1,3) {};
        \node[int,label=0:{$\bar A$}] (i5) at (1,3) {};
        \draw (v1) edge (i1) (i2) edge (i1) edge (i3) edge (i5) (i4) edge (i3) edge (i5);
        \end{tikzpicture}  
    \end{align}
    Let $T \in \U$ be the cochain given by the tadpole graph and let $Y \in \U$ be the tripod graph decorated by Poincaré dual classes.
    \begin{equation}\label{equ:TYdef}
    T = \,
    \begin{tikzpicture}[baseline=-.65ex, every loop/.style={draw,-}]
        \node[ext] (v1) at (0,0) {1};
        \draw (v1) edge[loop above, looseness=20, out=110,in=70] (v1);
    \end{tikzpicture},
    \qquad
    Y = \,
    \begin{tikzpicture}[baseline=-.65ex]
        \node[ext] (v1) at (0,0) {1};
        \node[int,label=90:{$\sum_q e_q e_q^*$}] (i1) at (0,.5) {};
        \draw (v1) edge (i1);
        \end{tikzpicture} 
    \end{equation}

    In particular, there is a canonical cochain $\nu$ given by the tadpole at the root plus the tripod decorated by Poincaré dual classes, satisfying $d\nu=\chi(M)\omega$, with $\omega\in H^n(M)$ being a volume form, normalized so that $M$ has volume 1.
    \begin{equation}\label{equ:nudef}
    \nu = \,
    \begin{tikzpicture}[baseline=-.65ex, every loop/.style={draw,-}]
        \node[ext] (v1) at (0,0) {1};
        \draw (v1) edge[loop above, looseness=20, out=110,in=70] (v1);
    \end{tikzpicture}  
    \, + \, 
    \begin{tikzpicture}[baseline=-.65ex]
        \node[ext] (v1) at (0,0) {1};
        \node[int,label=90:{$\sum_q e_q e_q^*$}] (i1) at (0,.5) {};
        \draw (v1) edge (i1);
        \end{tikzpicture} 
    \end{equation}
    We can decompose our model as follows into graded vector subspaces
    $$
    \U=  A \oplus A \tp \oplus \U_\bigcirc \oplus \U_\flower,
    $$
    where the terms are as follows.
    \begin{itemize}
        \item The term $A$ corresponds to the first type of graphs in \eqref{equ:Upics}.
        \item The term $AT$ corresponds to the second type of graphs in \eqref{equ:Upics}.
        \item $\U_\bigcirc = ( B(A,A,A) \otimes_{A \otimes A} A)_{\Z_2}$ corresponds to graphs of the third kind with a "stem" of length 0.
        \item $\U_\flower = ( B(A,A,A) {}_A \otimes_{A \otimes A} B(A,A,A))_{\Z_2}$ corresponds to terms of the third kind in \eqref{equ:Upics}, with the "stem" containing at least one edge.
    \end{itemize}
    The differential contains several pieces between our subspaces above.
    \begin{itemize}
        \item There are internal pieces of the differential acting on $A$, by edge contraction and edge splitting. These terms depend on the tree part $Z_{tree}$ of the partition function of section \ref{sec:GCM}.
        \item There is a piece of the differential $d_c : \U_\flower \to \U_\bigcirc$ by contracting of the stem if it has length exactly 1.
        \item There is a piece $d_{s_p} : \U_\flower \to A$ By cutting an edge in the loop.
        \item Similarly, there is a piece $d_s: \U_\bigcirc \to A$.
        \item  Finally there is a piece $d_{s_s}: \U_\flower \to A$ by cutting an edge in the stem.
        This disconnects the loop, and "sends it to a number" by using the loop order one part $Z_1$ of the partition function \eqref{equ:Zsplit}.
    \end{itemize}
    Moreover, one has that $d\nu\in A$ is a single vertex decorated by a top form $\xi \omega$ representing the Euler class of the manifold.    
    
\end{itemize}

We let $\evd : \C \to \U$ denote the natural map obtained by restricting to the boundary.

After choosing a propagator, \cite{CamposWillwacher} construct the following maps of complexes
\begin{align*}
    A &\longrightarrow \Omega_{PA}^\bullet(M) \\
    \mathcal{C} &\longrightarrow \Omega^\bullet_{PA}(\FM_M(2))\\
    \mathcal{U} &\longrightarrow \Omega^\bullet_{PA}(UTM = \partial \FM_M(2)).
\end{align*}
They are compatible with the $A$-module structure and restriction to the boundary. Moreover, the element $\nu \in \U$ is a representative of the fiberwise volume form, such that the Thom class $Th \in A \oplus \U[1]$ is represented by $\chi(M)\omega \oplus \nu$ and a lift in the $\chi(M) = 0$ case is given by $\Th = \nu \in \U$.

\begin{Rem}
We want to think of these graph complexes as constructed similarly to the Hochschild complexes from skeleton graphs, where the Hochschild edge $B$ is replaced by the "graph edge" which can be defined as
$$
A\langle | \rangle / (|^2) \oplus A \otimes A  = B \oplus A \oplus A \otimes A
$$
and differential $d(|) = 1 + |e_i \otimes e^i|$. We can write the total differential as $d = d_\text{Hoch} + d_c + d_s$, where $d_c : B \to A$ is the counit, and $d_s(|) = |e_i \otimes e^i|$. The two extra differentials will correspond to maps between parts of a graph complex with differing skeleta. Namely, the first corresponds to contracting an edge, while the second one corresponds to splitting one.
\end{Rem}

The edge splitting map $d_s: B \to A \otimes A$ can be expressed in two steps, first splitting the Hochschild edge and adding decorations $| e_i \otimes e^i |$ and then identifying the two resulting bar complexes as part of two copies of the bar-cobar resolution $A$
$$
B \to B \otimes B \to A \otimes A.
$$
We will use the notations
\begin{align}
\label{equ:leftmult}
    B &\overset{m_l}{\longrightarrow} A \\
    x | \alpha | y &\longmapsto x (\alpha y)
\end{align}
and 
\begin{align*}
    B &\overset{m_r}{\longrightarrow} A \\
    x | \alpha | y &\longmapsto (x\alpha) y
\end{align*}
for these maps. These maps have useful commutation relations with the contracting and the splitting differential, that is, we obtain for instance
$$
[d, m_l](x| \alpha | y) = \epsilon(\alpha)xy - d_s(m_l(x | \alpha | y)),
$$
hence we obtain a homotopy between $\epsilon : B \to A$ and $d_s(m_l(x | \alpha | y)) = x m_l(\alpha' e_i) z_0(e^i \alpha'' y) = x \alpha' e_i z_0(e^i \alpha'' y)$. We will use this to obtain a homotopy between $1y$ and something that lives in the trivalent part in cases when $y$ itself is not trivalent.

\begin{Lem}
\label{lem:contrsquare}
The square
$$
\begin{tikzcd}
A \otimes A \ar[r] \ar[d] & A \ar[d] \\
\C \ar[r] & \U
\end{tikzcd}
$$
is a homotopy pushout. Equivalently, the inclusion $\op{cone}(A \otimes A \to A) \to \op{cone}(\mathcal{C} \to \U)$ is a quasi-isomorphism of $A$-bimodules.
\end{Lem}
\begin{proof}
The map being an inclusion, it suffices to show that the cokernel
$$
\op{cone}( B \to A\tp \oplus \U_\bigcirc \oplus \U_\flower))
$$
is contractible. We consider the short exact sequence
$$
\begin{tikzcd}
\op{cone}( B \to A\tp) \ar{r} & \op{cone}( B \to A\tp \oplus \U_\bigcirc \oplus \U_\flower) \ar{r} & (\U_\bigcirc \oplus \U_\flower)[1]
\end{tikzcd}
$$
and note that the outer terms are contractible.
\end{proof}

The following follows directly from \cite{CamposWillwacher}.
\begin{Prop}\label{prop:cube}
The following cube commutes and all the vertical maps are quasi-isomorphisms.
$$
\begin{tikzcd}[back line/.style={densely dotted}, row sep=3em, column sep=3em]
& A \otimes A \ar{dl} \ar{rr} \ar{dd}[near end]{\simeq} 
  & & A \ar{dd}{\simeq} \ar{dl} \\
\C \ar[crossing over]{rr} \ar{dd}{\simeq} & & \U \\
& \Omega_{PA}^\bullet(M \times M) \ar{rr} \ar{dl} & & \Omega_{PA}^\bullet(M) \ar{dl} \\
\Omega_{PA}^\bullet(\FM_M(2)) \ar{rr} & & \Omega_{PA}^\bullet(UTM) \ar[crossing over, leftarrow]{uu}[near start]{\simeq}
\end{tikzcd}
$$
In other words, the upper face of the cube is indeed a model for the square \eqref{equ:bdry_inclusion} as desired.
\end{Prop}
\begin{proof}
The diagram commutes by construction. Since the vertical maps on the back face are quasi-isomorphisms, and the previous lemma, it is enough to show that the map $cone(A \otimes A \to \C) \to cone(\Omega^\bullet(M\times M) \to \Omega^\bullet(\FM_M(2)))$ is a quasi-isomorphism. This is clear, since by the Thom isomorphism both have cohomology a free $H$-module with generator given by the Thom class, and the vertical map respects these Thom classes by construction.
\end{proof}

We recall that for the construction of the coproduct in the case $\chi(M) = 0$ we used the map (in the derived category) of $A$-bimodules
$$
A \overset{\wedge \nu}{\longrightarrow} \Omega_{PA}^\bullet(UTM) \longrightarrow \op{cone}(\Omega_{PA}^\bullet(\FM_M(2)) \to \Omega_{PA}^\bullet(UTM)) \longleftarrow \op{cone}( A \otimes A \to A),
$$
where $\nu$ is a (in this case closed) fiberwise volume form. The above proposition allows us to replace this map with
$$
A \overset{\wedge \nu}{\longrightarrow} \U \longrightarrow \op{cone}(\C \to \U) \longleftarrow \op{cone}( A \otimes A \to A).
$$
It follows directly from the proposition that the second arrow is a quasi-isomorphism and hence the map is well-defined. More concretely, after tensoring everything from both sides with $B$, we obtain a map
$$
B \otimes_A B \longrightarrow \op{cone}(A \otimes A \to A),
$$
defined by the requirement that the diagram
$$
\begin{tikzcd}
B \otimes_A B \ar[r, "g"] \ar[dr, "\wedge Th"', ""{name=g}] & A\otimes A / A \ar[d] \arrow[Rightarrow, from=g, ""'] \\
& \C / \U.
\end{tikzcd}
$$
commutes up to homotopy. We seek to compute the map $g$ explicitly. For this we essentially spell out the formulas implicit in the proof of Lemma \ref{lem:contrsquare}.

In the case of the reduced coproduct we note that the above proposition also direcly implies the existence of a homotopy commuting diagram
$$
\begin{tikzcd}
QA[-n] \ar[r, "g"] \ar[dr, "\wedge Th"', ""{name=g}] & A\otimes A \ar[d] \arrow[Rightarrow, from=g, "h"'] \\
& A,
\end{tikzcd}
$$
with $QA:=B\otimes_A B$.
Recall that in this case the homotopy is extra data (and not a property of $g$) that appears in the description of the coproduct. Again, we will obtain formulas by spelling out the contracting homotopy of the upper face of the cube in Proposition \ref{prop:cube} at least for the image of the map $\wedge Th$.

\subsection{A contracting homotopy}
In this section we will produce an explicit contracting homotopy of (the total complex of) the diagram in Lemma \ref{lem:contrsquare}. Using the explicit description of $\C$ and $\U$ we can write the square as
\begin{equation}\label{equ:bunt1}
\begin{tikzcd}
{\color{blue} A \otimes A} \ar[r] \ar[d] & {\color{teal} A} \ar[d] \\
\C={\color{purple} B} \oplus {\color{blue} A \otimes A} \ar[r] & \U = {\color{teal} A} \oplus {\color{purple} A\tp} \oplus {\color{orange} \U_\bigcirc \oplus \U_\flower }
\end{tikzcd}
\end{equation}
where same-colored elements correspond to contractible subquotient complexes. Our strategy is to write down a contracting homotopy $h_0$ for each of these complexes and then get an overall contracting homotopy $H$ by the perturbation lemma.
Since we want to construct a contracting homotopy in $A$-bimodules we will take first tensor the diagram with $B$ over $A$ form the left and from the right to obtain
\begin{equation}\label{equ:bunt2}
\begin{tikzcd}
{\color{blue} B \otimes B} \ar[r] \ar[d] & {\color{teal} B\otimes_A B} \ar[d] \\
{\color{purple} B\otimes_A B\otimes_A B} \oplus {\color{blue} B\otimes B} \ar[r] & {\color{teal} B\otimes_A B} \oplus {\color{purple} B\otimes_A (A\tp) \otimes_A B} \oplus {\color{orange} B\otimes_A (\U_\bigcirc \oplus \U_\flower)\otimes_A B }
\end{tikzcd}
\end{equation}

We now describe the components of the contracting homotopy $h_0$, that is contracting homotopy of each of the colored subquotient complexes (tensored with $B \otimes B$ if our homotopy makes use of it).
\begin{itemize}
    \item ${\color{purple} B \otimes_A B \otimes_A B \to B \otimes_A B}$. The maps
    \begin{align*}
    B \otimes_A B &\longrightarrow B \otimes_A B \otimes_A B \\
    \alpha |x| \beta &\longmapsto \alpha | x | \beta' | 1 | \beta''
    \end{align*}
    and
    \begin{align*}
    B \otimes_A B \otimes_A B &\longrightarrow B \otimes_A B \otimes_A B \\
    \alpha |x| \beta |y| \gamma  &\longmapsto \alpha | x | \beta y \gamma' | 1 | \gamma''
    \end{align*}
    define a strict (the homotopies commute with the maps) homotopy inverse to $B \otimes_A B \otimes_A B \to B \otimes_A B$ where the second component is the action of the fundamental cycle of the interval via the reparametrization map (see \ref{sec:splittingmap}).
    Thus the contracting homotopy $h_0: B \otimes_A B \otimes_A B  \oplus  B \otimes_A B[1]$ has two components given by the formulas above.

    \item ${\color{blue} B \otimes B \overset{\op{id}}{\to} B \otimes B}$. We could simply chose the identity to be our homotopy. However, since we want to obtain a formula that lives in trivalent graphs (i.e. all $Com_\infty$-multiplications have been carried out), we choose a slighty more complicated contracting homotopy, namely we define
    \begin{align*}
    m : B \otimes B &\longrightarrow (B \otimes B)[1]\\
    \alpha \otimes \beta &\longmapsto \alpha' m_l(\alpha'') \otimes \beta [1],
    \end{align*}
    which is a degree $-2$ map on $\op{cone}( B \otimes B \to B \otimes B)$, that is we map $(u,v) \mapsto (m(v), 0)$,
    where $u \in B \otimes B$ and $v \in B \otimes B[1]$. We then take $h_0 = \op{id}[1] + [d, m]$ as our contracting homotopy, where $\op{id}[1]$ is the canonical contraction. Thus $h_0(u,v) = (mu + (1 -[d_{B\otimes B},m])(v), mv)$. Similarly to the discussion in \eqref{equ:leftmult} we obtain that $m$ gives a homotopy between the identity and the map
    $$
    (\alpha \otimes \beta) \longmapsto (\alpha'|\alpha''e_i) z_0(e^i\alpha''') \otimes \beta,
    $$
    which we will \todo{or won't we} abbreviate to $(\alpha \otimes \beta) \mapsto (\alpha'|d_s(\alpha'')) \otimes \beta$.
    Thus the homotopy $h_0$ is given by the formula
    $$
    h_0( \alpha_1 \otimes \beta_1 , \alpha_2 \otimes \beta_2) = (m(\alpha_1 \otimes \beta_1) + (\alpha_2'| d_s( \alpha_2'')) \otimes \beta_2 , m(\alpha_2\otimes \beta_2) ).
    $$
    \item ${\color{orange} \U_\bigcirc \to \U_\flower}$. We identify $\U_\bigcirc$ with $B^{\geq 1}_{\Z_2} \otimes_{A^{\otimes 2}} A$ and $\U_\flower$ with $B {}_A \otimes_{A^{\otimes 2}} B^{\geq 1}_{\Z_2}$.
    Similarly to above we get a contracting homotopy using deconcatenation and the interval action, namely it is given by components
    \begin{align*}
        h_0: B^{\geq 1} \otimes_{A^{\otimes 2}} A &\longrightarrow B {}_A \otimes_{A^{\otimes 2}} B^{\geq 1} \\
        \beta|x &\longmapsto \pm x|\beta' \Sha (\beta''')^* |1| \beta''.
    \end{align*}
    and
    \begin{align*}
        h_0: B {}_A \otimes_{A^{\otimes 2}} B^{\geq 1} &\longrightarrow B {}_A \otimes_{A^{\otimes 2}} B^{\geq 1} \\
        \alpha|x|\beta &\longmapsto \pm\alpha x \beta' \Sha (\beta''')^* |1| \beta'',
    \end{align*}
    which both descend to $\Z_2$-coinvariants since the shuffle product is commutative.
    \todo[inline]{It feels like this actually depends on what the sign of the reflection is.....and I failed to properly check this.
    Seems to be true.}
    
    \item ${\color{teal} B \otimes_A B \overset{\op{id}}{\to} B \otimes_A B}$. Similarly to the blue homotopy, we want to "carry out one $Com_\infty$-multiplication". We again twist the canonical contracting homotopy by the map 
    \begin{align*}
        m_m : B \otimes_A B &\longrightarrow B \otimes_A B  \\
        \alpha|x|\beta &\longmapsto \alpha' m_l(\alpha'' \Sha (\beta')^*|x) \beta'',
    \end{align*}
    to obtain the homotopy
    $$
    h_0(u,v) = (m_mu + (1 -[d_{B\otimes B},m_m])(v), m_mv),
    $$
    where
    $$
    (1 -[d_{B\otimes B},m_m])(\alpha|x|\beta) = \pm z_0((\alpha'' \Sha \beta') e^i)\alpha'| (\alpha'' \Sha (\beta'')^* e_i) | \beta'''.
    $$
    \todo[inline]{This is again very "suggestive notation". Maybe there is a better way to write these.}
\end{itemize}
Let us note that the only parts of the differential disregarded in the homotopy $h_0$ are those from the horizontal maps in our diagram and all the "splitting differentials" by cutting edges in graphs. Thus decomposing the differential of the total complex into $d = d_0 + d_1$, where now $[d_0,h_0] = \op{id}$, we obtain the contracting homotopy $H$ by
$$
H = h_0 + h_0 d_1 h_0 + h_0 d_1 h_0 d_1 h_0 + \dots,
$$
where one checks that the sum is finite, namely
$$
H = h_0 + h_0 d_1 h_0 + h_0 d_1 h_0 d_1 h_0 + h_0 d_1 h_0 d_1 h_0 d_1 h_0.
$$
Let us universally denote by $\pi:B\otimes_A (-) \otimes_A B \to (-)$ the projection undoing the tensor products with $B$ that we introduced above. 
Then we decompose $\pi \circ H = H_{A\otimes A} + H_{A} + H_{\C} + H_{\U}$ according to the target space. For instance $H_\C : A \oplus A\otimes A \oplus \C \oplus \U \to \C$ (with degree shifts ignored).

\subsection{Explicit formulas for the homotopies of Lemmas \ref{lem:hocop}, \ref{lem:frcop}}
Recall that for the $\chi(M)=0$-case of the coproduct we need to produce a map $g: QA[-d] \to A \otimes A / A$ (with $QA:=B\otimes_A B$ our chosen cofibrant resolution of the $A$-bimodule $A$) that makes the diagram
$$
\begin{tikzcd}
B \otimes_A B \ar[r, "g"] \ar[dr, "\wedge Th"', ""{name=g}] & A\otimes A / A \ar[d] \arrow[Rightarrow, from=g, ""'] \\
& \C / \U.
\end{tikzcd}
$$
homotopy commute, see Lemma \ref{lem:frcop}. Using the homotopy $H$ from the previous subsection we can choose
\begin{align*}
g &= (H_{A\otimes A} + H_{A}) \circ (\wedge Th).
\end{align*}

For the string coproduct in the reduced case (i.e., on $H(LM,M)$) we need to find $g$ and $h$ such that
\begin{equation}
\begin{tikzcd}
QA[-n] \ar[d] \ar[r, "g"] & (A \otimes A) / \C \ar[d]  \\
A[-n] \ar[r, "\wedge Th"] \ar[Rightarrow, ur, "h"] & A / \U,
\end{tikzcd}
\end{equation}
is a homotopy commuting square. We choose
\begin{align*}
g &= (H_{A \otimes A} + H_\C) \circ (\wedge Th) \\
h &= (H_A + H_\U ) \circ (\wedge Th)
\end{align*}
Composing with the projections $(A\otimes A)/\C\to A\otimes A$ and $A/\U\to A$ we obtain the diagram
$$
\begin{tikzcd}[column sep=huge,row sep=large]
B \otimes_A B[-n] \ar[r, "H_{A\otimes A} \circ (\wedge Th)"] \ar[d] & A \otimes A \ar[d]\\
A[-n] \ar[Rightarrow, ur, "H_A \circ (\wedge Th)"] \ar[r, "\wedge \chi(M)\omega"] & A.
\end{tikzcd}
$$

Thus in both cases ($\chi(M)=0$ or working modulo constant loops) it remains to compute
\begin{align*}
&H_{A\otimes A} \circ (\wedge Th)& &\text{and}
& & H_{A} \circ (\wedge Th).
\end{align*}
We note that $\Th = T + Y + \chi(M)\omega$ has three components, with $T\in \U$ the tadpole graph (first term in \eqref{equ:nudef}), $Y\in \U$ the second term in \eqref{equ:nudef} and $\omega\in H^n(M)\subset A$ the top dimensional cohomology class, normalized so that $M$ has volume 1.
The major contribution to $g,h$ above comes from the piece $\tp\wedge$, it image lies in the red summand of the lower right corner of \eqref{equ:bunt2}. On this summand, the homotopy $H$ is nontrivial, and we shall evaluate it now.

\subsubsection{Image of the tadpole graph}
We evaluate the formulas for $H_{A\otimes A} \circ (\wedge T)$ and $H_{A} \circ (\wedge T)$ step by step on a typical element
$$
\alpha | x | \beta = 1|\alpha_1 \alpha_2 \ldots \alpha_k| x | \beta_1 \beta_2 \ldots \beta_l | 1 \in A \otimes \bar{A}[1]^{\otimes k} \otimes A \otimes \bar{A}[1]^{\otimes l}\otimes A \subset B \otimes_A B.
$$
Note that this is enough since $H$ is a map of $A$-bimodules.
$$
\alpha | x | \beta =
\begin{tikzpicture}[scale = 2]
\filldraw (-1,0) circle (0.05);
\filldraw (0,0) circle (0.05);
\filldraw (1,0) circle (0.05);
\draw [thick] (-1,0) --(1,0);
\draw (-0.8,0) --(-0.8,0.1) node[above]{$\alpha_1$};
\draw (-0.6,0) --(-0.6,0.1) node[above]{$\alpha_2$};
\draw (-0.6,0) to node[above]{$\ldots$}(-0.2,0) ;
\draw (-0.2,0) --(-0.2,0.1) node[above]{$\alpha_k$};
\draw (0,0) --(0,0.2) node[above]{$x$};
\draw (0.2,0) --(0.2,0.1) node[above]{$\beta_1$};
\draw (0.4,0) --(0.4,0.1) node[above]{$\beta_2$};
\draw (0.4,0) to node[above]{$\ldots$} (0.8,0);
\draw (0.8,0) --(0.8,0.1) node[above]{$\beta_k$};
\end{tikzpicture}
\quad=\quad
\begin{tikzpicture}[scale = 2]
\filldraw (-1,0) circle (0.05);
\filldraw (0,0) circle (0.05);
\filldraw (1,0) circle (0.05);
\draw [thick] (-1,0) --(1,0);
\draw (0,0) --(0,0.2) node[above]{$x$};
\draw [thick] (-1,0) to node[above]{$\alpha$} (0,0);
\draw [thick] (0,0) to node[above]{$\beta$} (1,0);
\end{tikzpicture}
$$
The terms of $H(\alpha | x| \beta \wedge T)$ are obtained by iteratively applying $h_0$ and $d_1$,
\begin{align*}
    h_0( \alpha | x | \beta \wedge T) &= \Psi_1^\C \\
    d_1 h_0( \alpha | x | \beta \wedge T) &= \Phi_1^\C + \Phi_1^\U \\
    h_0 d_1 h_0( \alpha | x | \beta \wedge T) &= \Psi_2^{A \otimes A} + \Psi_2^\C + \Psi_2^\U \\
    d_1 h_0 d_1 h_0( \alpha | x | \beta \wedge T) &= \Phi_2^{A} + \Phi_2^\U \\
    h_0 d_1 h_0 d_1 h_0( \alpha | x | \beta \wedge T) &= \Psi_3^{A} + \Psi_3^\U \\
    d_1 h_0 d_1 h_0 d_1 h_0( \alpha | x | \beta \wedge T) &= 0,
\end{align*}
where we decomposed the images according to which corner of the square they lie in. In the following we compute each term and show that the missing components are zero. Using this notation we have
\begin{align*}
    H_{A\otimes A}(\alpha | x| \beta \wedge T) &= \pi \Psi_2^{A\otimes A} \\
    H_{A}(\alpha | x| \beta \wedge T) &=  \pi \Psi_3^A
\end{align*}

The term $\Psi_1^\C$ is obtained by applying the {\color{purple}purple homotopy}, that is we obtain 
$$
\Psi_1^\C = \alpha | x | \beta' | 1 | \beta'' \in {\color{purple}B \otimes_A B \otimes_A B} \subset B \otimes_A \mathcal{C} \otimes_A B
$$
$$
\Psi_1^\C = 
\begin{tikzpicture}[scale = 2]
\filldraw (-1,0) circle (0.05);
\filldraw (0,0) circle (0.05);
\filldraw (1,0) circle (0.05);
\filldraw (2,0) circle (0.05);
\draw [thick] (-1,0) to node[above]{$\alpha$}(0,0);
\draw (0,0) to node[above]{$\beta'$}(1,0);
\draw [thick] (1,0) to node[above]{$\beta''$}(2,0);
\draw (1,0) --(1,0.2) node[above]{$1$};
\draw (0,0) --(0,0.2) node[above]{$x$};
\end{tikzpicture}
\quad=\quad
\begin{tikzpicture}[scale = 2]
\filldraw (-1,0) circle (0.05);
\filldraw (0,0) circle (0.05);
\filldraw (1,0) circle (0.05);
\filldraw (2,0) circle (0.05);
\draw [thick] (-1,0) to node[above]{$\alpha$}(0,0);
\draw (0,0) to node[above]{$\beta'$}(1,0);
\draw [thick] (1,0) to node[above]{$\beta''$}(2,0);
\draw (0,0) --(0,0.2) node[above]{$x$};
\end{tikzpicture}
$$
where the middle edge is drawn in a different manner to remember that we think of it as a "graph edge" and not a "Hochschild edge", i.e. there is a splitting differential.
Applying $d_1$ gives two components, one coming from splitting the middle edge and one from the horizontal map $\evd:\C \to \U$. We write
$$ d_1 \Psi_1^\C = {\color{blue} \Phi_1^\C} +  {\color{orange} \Phi_1^\U},$$
where
$$
\Phi_1^\C = {\color{blue} d_s \Phi_2} = \alpha|x\beta' e_q) \otimes (e_q^* \beta''|1|\beta''') \in B \otimes_A \C \otimes_A B
$$
$$
{\color{blue} d_s \Phi_2} = 
\begin{tikzpicture}[scale = 2]
\filldraw (-1,0) circle (0.05);
\filldraw (0,0) circle (0.05);
\filldraw (1.8,0) circle (0.05);
\filldraw (2.8,0) circle (0.05);
\draw [thick] (-1,0) to node[above]{$\alpha$}(0,0);
\draw [thick] (1.8,0) to node[above]{$\beta'''$}(2.8,0);
\draw (0,0) to node[above]{$\beta'$}(0.6,0.2) node[right]{$\omega_i$};
\draw (1/4*0.6, 1/4*0.2) to (1/4*0.6 - 0.2/10, 1/4*0.2 + 0.6/10);
\draw (2/4*0.6, 2/4*0.2) to (2/4*0.6 - 0.2/10, 2/4*0.2 + 0.6/10);
\draw (3/4*0.6, 3/4*0.2) to (3/4*0.6 - 0.2/10, 3/4*0.2 + 0.6/10);
\draw (1.8,0) to node[above]{$\beta''$} (1.2,0.2) node[left]{$\omega^i$};
\draw (1.8-1/4*0.6, 1/4*0.2) to (1.8-1/4*0.6 + 0.2/10, 1/4*0.2 + 0.6/10);
\draw (1.8-2/4*0.6, 2/4*0.2) to (1.8-2/4*0.6 + 0.2/10, 2/4*0.2 + 0.6/10);
\draw (1.8-3/4*0.6, 3/4*0.2) to (1.8-3/4*0.6 + 0.2/10, 3/4*0.2 + 0.6/10);
\draw (0,0) --(0,0.2) node[above]{$x$};
\end{tikzpicture}
\quad=\quad
\begin{tikzpicture}[scale = 2]
\filldraw (-1,0) circle (0.05);
\filldraw (0,0) circle (0.05);
\filldraw (1.8,0) circle (0.05);
\filldraw (2.8,0) circle (0.05);
\draw [thick] (-1,0) to node[above]{$\alpha$}(0,0);
\draw [thick] (1.8,0) to node[above]{$\beta'''$}(2.8,0);
\draw (0,0) to node[above]{$\beta'$}(0.6,0.2);
\draw (1.8,0) to node[above]{$\beta''$} (1.2,0.2);
\draw (0,0) --(0,0.2) node[above]{$x$};
\draw [dotted](0.6,0.2) to [bend left = 18] (1.2,0.2);
\end{tikzpicture}
$$
and 
$$
\Phi_1^\U = \evd (\Phi_1).
$$

To obtain the $\Psi_2$'s we have to apply the blue homotopy to $\Phi_1^\C$ and the orange homotopy to $\Phi_1^\U$
Then $\Psi_2^{A \otimes A} + \Psi_2^\C = h_0 {\color{blue} \Psi_1^\C}$ are the components after applying the blue homotopy. That is
$$
\Psi_2^{A \otimes A} = z_0(e^i \alpha''' x \beta' e_j) (\alpha'|\alpha'' e_i) \otimes (e^j \beta''| \beta''') \in {\color{blue} B \otimes B}
$$
$$ 
\Psi_2^{A \otimes A} \quad=\quad 
\begin{tikzpicture}[scale = 2, baseline={([yshift=-50pt]current bounding box.north)}]
\draw [dotted] (-.8,0) to (-.5,0);
\draw (-.5,0) to node[above]{$\alpha'''$} (0,0);
\draw (0,0) to node[above]{$\beta'$} (.5,0);
\draw (0,0) to (0, 0.2) node[above]{$x$};
\draw [dotted] (.5,0) to (.8,0);
\draw [dashed] (0,0.2) circle (0.7);
\draw (0,0.85) node[above]{$z_0$};
\draw (-1.1,0) to node[above]{$\alpha''$}(-.8,0);
\draw [thick](-1.6,0) to node[above]{$\alpha'$}(-1.1,0);
\filldraw (-1.6,0) circle (0.05);
\filldraw (-1.1,0) circle (0.05);
\draw (.8,0) to node[above]{$\beta''$}(1.1,0);
\draw [thick](1.1,0) to node[above]{$\beta'''$}(1.6,0);
\filldraw (1.6,0) circle (0.05);
\filldraw (1.1,0) circle (0.05);
\end{tikzpicture}
$$
and
$$
\Psi_2^\C = \alpha' | \alpha ''x d_s'(\beta') \otimes d_s''(\beta') | 1 | \beta'' \in B \otimes B \subset B \otimes_A \C \otimes_A B
$$
$$
\Psi_2^\C \quad=\quad 
\begin{tikzpicture}[scale = 2]
\filldraw (-1,0) circle (0.05);
\filldraw (-.5,0) circle (0.05);
\filldraw (1.8,0) circle (0.05);
\filldraw (2.3,0) circle (0.05);
\draw [thick] (-1,0) to node[above]{$\alpha'$}(-.5,0);
\draw [thick] (1.8,0) to node[above]{$\beta'''$}(2.3,0);
\draw (-.5,0) to node[above]{$\alpha''$}(0.05,0.1);
\draw (0.05,0.1) to node[above]{$\beta'$}(0.6,0.2);
\draw (-.5 + 4/8*1.1, 4/8*0.2) to node[above]{$x$}((-.5 + 4/8*1.1 - 0.2/10, 4/8*0.2 + 1.1/10);
\draw (1.8,0) to node[above]{$\beta''$} (1.2,0.2);
\draw [dotted](0.6,0.2) to [bend left = 10] (1.2,0.2);
\end{tikzpicture}
$$

The term $\Psi_2^\U$ obtained by applying the orange homotopy to $\Phi_1^\U$ and hence
$$
\Psi_2^\U = \alpha | x | \beta' \Sha (\beta''')^* |1| \beta'' |1| \beta'''' \in B \otimes_A \U \otimes_A B.
$$
$$
\Psi_2^\U = 
\begin{tikzpicture}[scale = 2]
\filldraw (-.5,0) circle (0.05);
\filldraw (.5,0) circle (0.05);
\draw [thick] (-.5,0) to node[above]{$\alpha$} (0,0);
\draw [thick] (0,0) to node[above]{$\beta''''$} (.5,0);
\draw (0,0) to node[left]{$\beta'$} node[right]{$\beta'''$} (0,1);
\draw (0,1.5) circle  (0.5) ;
\draw (0.2,2) node[right]{$\beta'' > 0$};
\draw (0,0) to (-0.2, 0.3) node[left]{$x$};
\end{tikzpicture}
$$
Here $\beta'' > 0$ denotes the condition that $\beta''$ contains at least 1 element of $A$, that is we apply the reduced coproduct on this factor.

The term $\Phi_2^A$ is the image of $\Psi_2^{A\otimes A}$ under the horizontal multiplication map, that is
$$
\Phi_2^A = z_0(e^i \alpha''' x \beta' e_j) (\alpha'|(\alpha'' e_i)(e^j \beta'')| \beta''') \in {\color{blue} B \otimes_A B}
$$
$$
\Phi_2^A \quad = \quad
\begin{tikzpicture}[scale = 2, baseline=-.65ex]
\filldraw (-.5,-.5) circle (0.05);
\draw [thick] (-.5,-.5) to node[below]{$\alpha$} (0,-.5);
\filldraw (.5,-.5) circle (0.05);
\draw [thick] (0,-.5) to node[below]{$\beta$}(.5,-.5);
\filldraw (0,-.5) circle (0.05);
\draw (0,-.5) to node[left]{$\alpha$} (-0.3,0.1);
\draw (0,-.5) to node[right]{$\beta$} (0.3,0.1);
\draw[dotted] (-0.3,0.1) to (-0.4, 0.3);
\draw[dotted] (0.3,0.1) to (0.4, 0.3);
\draw (-.4,.3) to [out=120, in=180] node[below left]{$\alpha$} (0,1.0);
\draw (.4,.3) to [out=60, in=0] node[below right]{$\beta$} (0,1.0);
\draw (0,1) to (0,1.2) node[above]{$x$};
\draw [dashed] (0,0.8) ellipse [x radius = .7, y radius = .7];
\draw (0,1.5) node[above]{$z_0$};
\end{tikzpicture}
$$

The two components of the term $\Phi_2^\U = \evd(\Psi_2^\C) + d_s \Psi_2^\U$ are given by
$$
\evd{\Psi_2^\C} = \alpha'| (\alpha''x\beta'e_i)(e^i \beta'')| \beta'''
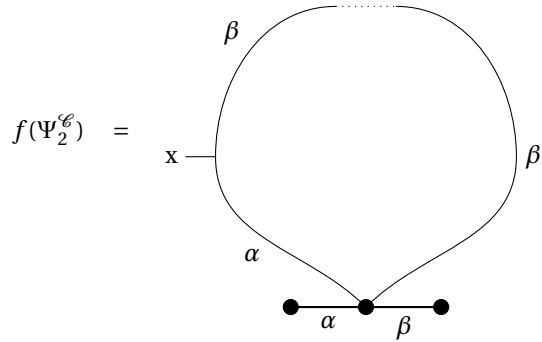
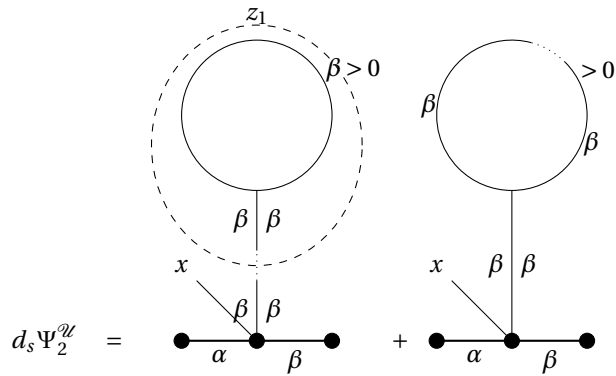
$$
$$
\evd(\Psi_2^\C) \quad = \quad
\begin{tikzpicture}[scale = 2, baseline={([yshift=-50pt]current bounding box.north)}]
\filldraw (-.5,0) circle (0.05);
\draw [thick] (-.5,0) to node[below]{$\alpha$} (0,0);
\filldraw (.5,0) circle (0.05);
\draw [thick] (0,0) to node[below]{$\beta$}(.5,0);
\filldraw (0,0) circle (0.05);
\draw (0,0) to [out=135, in=-90] node[below left]{$\alpha$} (-1,1);
\draw (-1,1) to [out=90, in=180] node[above left]{$\beta$} (-.2,2);
\draw (-1,1) --(-1.2,1) node[left] {x};
\draw [dotted] (-.2,2) to (.2,2);
\draw (0,0) to [out=45, in=-90] (1,1) node[right]{$\beta$};
\draw (1,1) to [out=90, in=0] (.2,2);
\end{tikzpicture}
$$
and
$$
d_s \Psi_2^\U \quad = \quad
\begin{tikzpicture}[baseline=-.65ex,scale = 2]
\filldraw (-.5,0) circle (0.05);
\draw [thick] (-.5,0) to node[below]{$\alpha$} (0,0);
\filldraw (.5,0) circle (0.05);
\draw [thick] (0,0) to node[below]{$\beta$}(.5,0);
\filldraw (0,0) circle (0.05);
\draw (0,0) to (-.4,0.4) node[above left]{$x$};
\draw (0,0) to node[left]{$\beta$} node[right]{$\beta$} (0,.4);
\draw [dotted](0,0.4) to (0,.6);
\draw (0,0.6) to node[left]{$\beta$} node[right]{$\beta$} (0,1);
\draw (0,1.5) circle (0.5);
\draw (.4, 1.8) node[right]{$\beta > 0$};
\draw [dashed] (0,1.3) ellipse [x radius = .7, y radius = .8];
\draw (0,2.05) node[above]{$z_1$};
\end{tikzpicture}
 + 
\begin{tikzpicture}[baseline=-.65ex,scale = 2]
\filldraw (-.5,0) circle (0.05);
\draw [thick] (-.5,0) to node[below]{$\alpha$} (0,0);
\filldraw (.5,0) circle (0.05);
\draw [thick] (0,0) to node[below]{$\beta$}(.5,0);
\filldraw (0,0) circle (0.05);
\draw (0,0) to (-.4,0.4) node[above left]{$x$};
\draw (0,0) to node[left]{$\beta$} node[right]{$\beta$} (0,1);
\draw ([shift=(-90:.5)]0,1.5) arc (-90:45:.5) node[midway]{\phantom{xl}$\beta$};
\draw [dotted]([shift=(45:.5)]0,1.5) arc (45:75:.5);
\draw ([shift=(75:.5)]0,1.5) arc (75:270:.5) node[midway]{$\beta$\phantom{xl}};
\draw (.4, 1.8) node[right]{$ > 0$};
\end{tikzpicture}
$$
Since we have no need for the term $\Psi_3^\U$ we only compute $\Psi_3^A$. It is obtained by applying the green homotopy to $\Phi_2^A$ and $\Phi_2^\U$, that is
$$
\Psi_3^A = h_0^A(\Phi_2^A + \Phi_2^\U)
$$
We obtain
$$
h_0^A(\Phi_2^A) \quad = \quad 
\begin{tikzpicture}[baseline = {(0,-1)}, scale = 2]
\filldraw (-.5,-.5) circle (0.05);
\draw [thick] (-.5,-.5) to node[below]{$\alpha$} (0,-.5);
\filldraw (.5,-.5) circle (0.05);
\draw [thick] (0,-.5) to node[below]{$\beta$}(.5,-.5);
\filldraw (0,-.5) circle (0.05);
\draw (0,-.5) to node[left]{$\alpha$} node[right]{$\beta$} (0,-.2);
\draw [dotted] (0,-.2) to (0, -.05);
\draw (0,-.05) to node[left]{$\alpha$} node[right]{$\beta$} (0,.25);
\draw (0,.25) to node[left]{$\alpha$} (-.1,.45);
\draw (0,.25) to node[right]{$\beta$} (.1,.45);
\draw [dotted] (-.1,.45) to (-.3, .85);
\draw [dotted] (.1, .45) to (.3, .85);
\draw [dashed] (0,0.25) ellipse [x radius = .35, y radius = .35];
\draw (0.4,0.25) node[right]{$z_0$};
\draw (-.3,.85) to [out=120, in=180] node[left]{$\alpha$} (0,1.3);
\draw (.3,.85) to [out=60, in=0] node[right]{$\beta$} (0,1.3);
\draw (0,1.3) to (0,1.5) node[above]{$x$};
\draw [dashed] (0,1.25) ellipse [x radius = .6, y radius = .6];
\draw (0,1.8) node[above]{$z_0$};
\end{tikzpicture}
$$
and\footnote{We apologize for the somewhat cumbersome notation, but hope that the meaning is still clear from the picture below.}
\begin{align*}
\pi \circ h_0^A(\Phi_2^\U) =& \alpha' \Sha (\beta''''''')^* e_i z_0(e^i \alpha'' (\beta'''''')^* x \beta' (\beta''''')^* ((\beta'''')^* e_k) \beta''' e^k) \\
    &+ \alpha' \Sha (\beta''''''')^* e_i z_0(e^i \alpha'' \Sha (\beta'''''')^* x \beta' \Sha (\beta''''')^* e_j) z_1( (\beta'' \Sha (\beta'''')^* e^j) \beta''') \\
    & + \alpha' \Sha (\beta'''')^* e_i z_0(e^i \alpha'' \Sha (\beta''')^* ((\beta'')^* e_k) \alpha''' x \beta' e^k)
\end{align*}
$$
h_0^A(\Phi_2^\U) \quad = \quad 
\begin{tikzpicture}[scale = 2,baseline={(0,-1)}]
\filldraw (-.5,-.5) circle (0.05);
\draw [thick] (-.5,-.5) to node[below]{$\alpha$} (0,-.5);
\filldraw (.5,-.5) circle (0.05);
\draw [thick] (0,-.5) to node[below]{$\beta$}(.5,-.5);
\filldraw (0,-.5) circle (0.05);
\draw (0,-.5) to node[left]{$\alpha$} node[right]{$\beta$} (0,-.1);
\draw [dotted] (0,-.1) to (0, .2);
\draw (0,0.2) to node[left]{$\alpha$} node[right]{$\beta$} (0,1);
\draw ([shift=(-90:.5)]0,1.5) arc (-90:45:.5) node[midway]{\phantom{xl}$\beta$};
\draw [dotted]([shift=(45:.5)]0,1.5) arc (45:75:.5);
\draw ([shift=(75:.5)]0,1.5) arc (75:180:.5) node[midway]{$\beta$\phantom{xll}};
\draw ([shift=(180:.5)]0,1.5) arc (180:270:.5) node[midway]{$\alpha$\phantom{xll}};
\draw (-.5,1.5) --(-.7,1.5) node[left] {$x$};
\draw [dashed] (0,1.1) ellipse [x radius = .9, y radius = 1.1];
\draw (0,2.15) node[above]{$z_0$};
\end{tikzpicture}
+
\begin{tikzpicture}[baseline={(0,-1)},scale = 2]
\filldraw (-.5,-.5) circle (0.05);
\draw [thick] (-.5,-.5) to node[below]{$\alpha$} (0,-.5);
\filldraw (.5,-.5) circle (0.05);
\draw [thick] (0,-.5) to node[below]{$\beta$}(.5,-.5);
\filldraw (0,-.5) circle (0.05);
\draw (0,-.5) to node[left]{$\alpha$} node[right]{$\beta$} (0,-.2);
\draw [dotted] (0,-.2) to (0, -.05);
\draw (0,-.05) to node[left]{$\alpha$} node[right]{$\beta$} (0,.25);
\draw (0,0.25) to (-.2,0.25) node[left]{$x$};
\draw (0,0.25) to node[left]{$\beta$} node[right]{$\beta$} (0,.55);
\draw [dotted] (0,.55) to (0, .7);
\draw (0,0.7) to node[left]{$\beta$} node[right]{$\beta$} (0,1);
\draw (0,1.5) circle (0.5);
\draw (.4, 1.8) node[right]{$\beta > 0$};
\draw [dashed] (0,1.35) ellipse [x radius = .7, y radius = .7];
\draw (0,2.0) node[above]{$z_1$};
\draw [dashed] (0,0.25) ellipse [x radius = .35, y radius = .35];
\draw (0.4,0.25) node[right]{$z_0$};
\end{tikzpicture}
 + 
\begin{tikzpicture}[baseline={(0,-1)},scale = 2]
\filldraw (-.5,-.5) circle (0.05);
\draw [thick] (-.5,-.5) to node[below]{$\alpha$} (0,-.5);
\filldraw (.5,-.5) circle (0.05);
\draw [thick] (0,-.5) to node[below]{$\beta$}(.5,-.5);
\filldraw (0,-.5) circle (0.05);
\draw (0,-.5) to node[left]{$\alpha$} node[right]{$\beta$} (0,-.1);
\draw [dotted] (0,-.1) to (0, .2);
\draw (0,0.2) to node[left]{$\alpha$} node[right]{$\beta$} (0,.6);
\draw (0,0.6) to node[left]{$\beta$} node[right]{$\beta$} (0,1);
\draw (0,0.6) to (-.2,0.6) node[left]{$x$};
\draw ([shift=(-90:.5)]0,1.5) arc (-90:45:.5) node[midway]{\phantom{xl}$\beta$};
\draw [dotted]([shift=(45:.5)]0,1.5) arc (45:75:.5);
\draw ([shift=(75:.5)]0,1.5) arc (75:270:.5) node[midway]{$\beta$\phantom{xl}};
\draw [dashed] (0,1.1) ellipse [x radius = .9, y radius = 1.1];
\draw (0,2.15) node[above]{$z_0$};
\draw (.4, 1.8) node[right]{$ > 0$};
\end{tikzpicture}
$$

\subsubsection{The "other" terms}
It remains to compute $H$ on the terms $\alpha|x|\beta \wedge Y$ and $\alpha|x|\beta \wedge \chi(M)\omega$.
In the first case, we obtain $H(\alpha|x|\beta \wedge Y)$ by applying the green homotopy,
$$
H^A(\alpha|x|\beta \wedge Y) = \pi (h_0^A(\alpha|x|\beta \wedge Y)),
$$
$$
h_0^A(\alpha|x|\beta \wedge Y) \quad = \quad
\begin{tikzpicture}[baseline={(0,-1)},scale = 2]
\filldraw (-.5,-.5) circle (0.05);
\draw [thick] (-.5,-.5) to node[below]{$\alpha$} (0,-.5);
\filldraw (.5,-.5) circle (0.05);
\draw [thick] (0,-.5) to node[below]{$\beta$}(.5,-.5);
\filldraw (0,-.5) circle (0.05);
\draw (0,-.5) to node[left]{$\alpha$} node[right]{$\beta$} (0,-.1);
\draw [dotted] (0,-.1) to (0, .2);
\draw (0,0.2) to node[left]{$\alpha$} node[right]{$\beta$} (0,.6);
\draw (0,0.6) to  (0,.9) node[above]{$e_i e^i$};
\draw (0,0.6) to (-.2,0.6) node[left]{$x$};
\draw [dashed] (0,.6) ellipse [x radius = .6, y radius = .6];
\draw (0,1.2) node[above]{$z_0$};
\end{tikzpicture}
\quad = \quad
\begin{tikzpicture}[baseline={(0,-1)},scale = 2]
\filldraw (-.5,-.5) circle (0.05);
\draw [thick] (-.5,-.5) to node[below]{$\alpha$} (0,-.5);
\filldraw (.5,-.5) circle (0.05);
\draw [thick] (0,-.5) to node[below]{$\beta$}(.5,-.5);
\filldraw (0,-.5) circle (0.05);
\draw (0,-.5) to node[left]{$\alpha$} node[right]{$\beta$} (0,-.1);
\draw [dotted] (0,-.1) to (0, .2);
\draw (0,0.2) to node[left]{$\alpha$} node[right]{$\beta$} (0,.6);
\draw (0,0.6) to  (0,.9);
\draw [dotted] (0,1.1) circle [radius = 0.2];
\draw (0,0.6) to (-.2,0.6) node[left]{$x$};
\draw [dashed] (0,.7) ellipse [x radius = .7, y radius = .7];
\draw (0,1.4) node[above]{$z_0$};
\end{tikzpicture}
$$
Note that this is the "missing" term in the last summand of $h_0^A(\Phi_2^\U)$ if we take the ordinary coproduct instead of the reduced one in that term.

The term $H^A(\alpha|x|\beta \wedge \chi(M)\omega)$ is obtained by applying the green homotopy, i.e.
$$
H^A(\alpha|x|\beta \wedge \chi(M)\omega) = \pi (h_0^A(\alpha|x|\beta \wedge \chi(M)\omega)),
$$
$$
h_0^A(\alpha|x|\beta \wedge \chi(M)\omega) \quad = \quad
\begin{tikzpicture}[baseline = {(0,1)},scale = 2]
\filldraw (-.5,0) circle (0.05);
\draw [thick] (-.5,0) to node[below]{$\alpha$} (0,0);
\filldraw (.5,0) circle (0.05);
\draw [thick] (0,0) to node[below]{$\beta$}(.5,0);
\filldraw (0,0) circle (0.05);
\draw (0,0) to node[left]{$\alpha$} node[right]{$\beta$} (0,1);
\draw (0,1) node[right]{$\chi(M)\omega$} to (-0.2,1.2) node[above left]{$x$};
\end{tikzpicture}
$$

\subsubsection{The collected terms}
Let us summarize all of the above in the following two pictures.
$$
H_{A\otimes A}\circ (\wedge \Th) \quad = \quad
\begin{tikzpicture}[baseline = {(0,0)}, scale = 2]
\draw [dotted] (-.8,0) to (-.5,0);
\draw (-.5,0) to node[above]{$\alpha$} (0,0);
\draw (0,0) to node[above]{$\beta$} (.5,0);
\draw (0,0) to (0, 0.2) node[above]{$x$};
\draw [dotted] (.5,0) to (.8,0);
\draw [dashed] (0,0.2) circle (0.7);
\draw (0,0.85) node[above]{$z_0$};
\draw (-1.1,0) to node[above]{$\alpha$}(-.8,0);
\filldraw (-1.1,0) circle (0.05);
\draw (.8,0) to node[above]{$\beta$}(1.1,0);
\filldraw (1.1,0) circle (0.05);
\end{tikzpicture}
$$
$$
H_{A}\circ (\wedge \Th)  \quad = \quad
\quad 
\begin{tikzpicture}[baseline = {(0,-1)}, scale = 2]
\filldraw (0,-.5) circle (0.05);
\draw (0,-.5) to node[left]{$\alpha$} node[right]{$\beta$} (0,-.2);
\draw [dotted] (0,-.2) to (0, -.05);
\draw (0,-.05) to node[left]{$\alpha$} node[right]{$\beta$} (0,.25);
\draw (0,.25) to node[left]{$\alpha$} (-.1,.45);
\draw (0,.25) to node[right]{$\beta$} (.1,.45);
\draw [dotted] (-.1,.45) to (-.3, .85);
\draw [dotted] (.1, .45) to (.3, .85);
\draw [dashed] (0,0.25) ellipse [x radius = .35, y radius = .35];
\draw (0.4,0.25) node[right]{$z_0$};
\draw (-.3,.85) to [out=120, in=180] node[left]{$\alpha$} (0,1.3);
\draw (.3,.85) to [out=60, in=0] node[right]{$\beta$} (0,1.3);
\draw (0,1.3) to (0,1.5) node[above]{$x$};
\draw [dashed] (0,1.25) ellipse [x radius = .6, y radius = .6];
\draw (0,1.8) node[above]{$z_0$};
\end{tikzpicture}
+
\begin{tikzpicture}[baseline = {(0,-1)}, scale = 2]
\filldraw (0,-.5) circle (0.05);
\draw (0,-.5) to node[left]{$\alpha$} node[right]{$\beta$} (0,-.1);
\draw [dotted] (0,-.1) to (0, .2);
\draw (0,0.2) to node[left]{$\alpha$} node[right]{$\beta$} (0,1);
\draw ([shift=(-90:.5)]0,1.5) arc (-90:45:.5) node[midway]{\phantom{xl}$\beta$};
\draw [dotted]([shift=(45:.5)]0,1.5) arc (45:75:.5);
\draw ([shift=(75:.5)]0,1.5) arc (75:180:.5) node[midway]{$\beta$\phantom{xll}};
\draw ([shift=(180:.5)]0,1.5) arc (180:270:.5) node[midway]{$\alpha$\phantom{xll}};
\draw (-.5,1.5) --(-.7,1.5) node[left] {$x$};
\draw [dashed] (0,1.1) ellipse [x radius = .9, y radius = 1.1];
\draw (0,2.15) node[above]{$z_0$};
\end{tikzpicture}
+
\begin{tikzpicture}[baseline = {(0,-1)}, scale = 2]
\filldraw (0,-.5) circle (0.05);
\draw (0,-.5) to node[left]{$\alpha$} node[right]{$\beta$} (0,-.2);
\draw [dotted] (0,-.2) to (0, -.05);
\draw (0,-.05) to node[left]{$\alpha$} node[right]{$\beta$} (0,.25);
\draw (0,0.25) to (-.2,0.25) node[left]{$x$};
\draw (0,0.25) to node[left]{$\beta$} node[right]{$\beta$} (0,.55);
\draw [dotted] (0,.55) to (0, .7);
\draw (0,0.7) to node[left]{$\beta$} node[right]{$\beta$} (0,1);
\draw (0,1.5) circle (0.5);
\draw (.4, 1.8) node[right]{$\beta > 0$};
\draw [dashed] (0,1.35) ellipse [x radius = .7, y radius = .7];
\draw (0,2.0) node[above]{$z_1$};
\draw [dashed] (0,0.25) ellipse [x radius = .35, y radius = .35];
\draw (0.4,0.25) node[right]{$z_0$};
\end{tikzpicture}
 + 
\begin{tikzpicture}[baseline = {(0,-1)}, scale = 2]
\filldraw (0,-.5) circle (0.05);
\draw (0,-.5) to node[left]{$\alpha$} node[right]{$\beta$} (0,-.1);
\draw [dotted] (0,-.1) to (0, .2);
\draw (0,0.2) to node[left]{$\alpha$} node[right]{$\beta$} (0,.6);
\draw (0,0.6) to node[left]{$\beta$} node[right]{$\beta$} (0,1);
\draw (0,0.6) to (-.2,0.6) node[left]{$x$};
\draw ([shift=(-90:.5)]0,1.5) arc (-90:45:.5) node[midway]{\phantom{xl}$\beta$};
\draw [dotted]([shift=(45:.5)]0,1.5) arc (45:75:.5);
\draw ([shift=(75:.5)]0,1.5) arc (75:270:.5) node[midway]{$\beta$\phantom{xl}};
\draw [dashed] (0,1.1) ellipse [x radius = .9, y radius = 1.1];
\draw (0,2.15) node[above]{$z_0$};
\end{tikzpicture}
+
\begin{tikzpicture}[baseline = {(0,0)},scale = 2]
\filldraw (0,0) circle (0.05);
\draw (0,0) to node[left]{$\alpha$} node[right]{$\beta$} (0,1);
\draw (0,1) node[right]{$\chi(M)\omega$} to (-0.2,1.2) node[above left]{$x$};
\end{tikzpicture}
$$

\begin{Rem}
In the 1-framed case the propagator was chosen compatible with the given 1-framing. Alternatively, since any 1-framings differ by a class $f \in H^{n-1}(M)$ we could simply add this term to $\Th$. The extra term in the homotopies above can then be absorbed in the middle term, if we now also take the ordinary (not the reduced) coproduct there, and define $z_1$ on tadpole graphs to be dual to $f$.
\end{Rem}

\subsubsection{Simpler formulas on a smaller model}\label{sec:simpler formulas}
Recall that the differentials on $B \otimes_A B$ and $A$ consist of an edge contraction differential $d_0$ and a differential depending on the $Com_\infty$ structure. The cohomology with respect to $d_0$ is concentrated in the purely trivalent part and given by $T\bar{H} \otimes H \otimes T\bar{H}$ and $H$, respectively. 

Since $H^{A\otimes A}$ and $H^A$ send trivalent graphs to trivalent ones, we can readily project onto these spaces to obtain maps
\begin{align*}
    H^{A \otimes A}\circ(\wedge \Th) :    T\bar{H} \otimes H \otimes T\bar{H} &\longrightarrow H \otimes H \\
    H^A\circ(\wedge \Th) : T\bar{H} \otimes H T\bar{H} &\longrightarrow H
\end{align*}
given by (essentially) the same formulas. Note that the term coming from $\wedge \chi(M) \omega$ vanishes.

The formulas simplify vastly if $x = 1$, namely
\begin{align}
\label{eqn:simplpsi}
   H^{A \otimes A}\circ(\alpha | 1 | \beta \wedge \Th) &= \epsilon(\alpha) \epsilon(\beta) e_i \otimes e^i \\
   \label{eqn:simplpsi2}
   H^A\circ( \alpha | 1 | \beta \wedge \Th) &= \epsilon(\alpha)e_i z_1(e^i \beta).
\end{align}

\section{Putting it all together, and proof of Theorem \ref{thm:main_3b} }\label{sec:thm3proof}
We are now ready to describe the string bracket and cobracket on the cyclic words $\bCyc(\bar H)$ computing $\bar H_{S^1}(LM)$ and in particular prove Theorem \ref{thm:main_3b}.

We will proceed essentially as in section \ref{sec:thm12proofs} for the computation of the product and coproduct.
Here we will use the dgca model $A$ of $M$ given by graphs as in section \ref{sec:graphical version}.
The cochain complex computing $H(LM)$ is then the reduced Hochschild complex
\[
\bar C(A)=B\otimes_{A^e} A = \left( \bigoplus_{k\geq 0}\bar A^{\otimes k} \otimes A, d_H \right).
\]

The cohomology $H=H^\bullet(M)$ forms a $\hCom_\infty$-algebra, and $A$ is its canonical resolution as a $\Com$-algebra.
Hence, there is a natural $\hCom_\infty$-map $H\to A$.
It follows that we also have a canonical map between the Hochschild complexes, cf. section \ref{sec:intro hochschild}
\[
\bar C(H) = \left( \bigoplus_{k\geq 0}\bar H^{\otimes k} \otimes H, d_H \right) \to \bar C(A).
\]

Unfortunately, we do not know a natural way to construct an (explicit) $\Com_\infty$ or $\hCom_\infty$-map $A\to H$.
Nevertheless, the natural projection of graded vector spaces $A\to H$ induces a well defined map 

$$
B \otimes_{A^e} A \supset (B \otimes_{A^e} A)^{\text{triv, cl}} \overset{\text{pr}}{\longrightarrow} T\bar{H} \otimes H.
$$
on the subspace of closed with respect to edge contraction and trivalent elements where the first inclusion is a quasi-isomorphism and the projection is a bijection.
\todo[inline]{The statement about bijectivity seems more clear for the dual picture in terms of graphs / IHX.}
\todo[inline]{...maybe add some detail here}
\begin{Rem}
The last formula is stating the fact that given a $\Com_\infty$ algebra $H$ one can construct Hochschild chains either by bar-cobar resolving $H$ as a $\Com_\infty$-algebra or by doing associative cobar and taking Hochschild chains on the resulting (cofree) coalgebra (and recalling that Hochschild chains of a (co)free (co)algebra $TV$ are given by $TV \otimes (V \oplus \R)$).
\end{Rem}
Since the Connes' operator is compatible with that projection the discussion about negative cyclic homology of section\ref{sec:intro hochschild} carries over to the model $(T\bar{H} \otimes H)$ and we obtain similarly a quasi-isomorphism
$$
\Cyc(\bar{A}) \supset \Cyc(\bar{A})^{\text{triv,cl}} \overset{\text{pr}}{\longrightarrow} \Cyc(\bar{H}),
$$
with the projection being a bijection.

\subsection{String product and bracket}
We now obtain a description of the string product (cohomology coproduct). Let us first consider the operation $H^\bullet(LM) \otimes H^\bullet(LM) \longleftarrow H^\bullet(\Map(8))$. By Lemma \ref{lem:hoprod} it is obtained from
$$
A \longleftarrow QA \overset{H_{A \otimes A}}{\longrightarrow} A \otimes A,
$$
by taking tensor product with $B \otimes B$ over $A^{\otimes 4}$. Note that $(B \otimes B) \otimes_{A^{\otimes 4}} QA = B \otimes_{A^e} B \otimes_A B \otimes_{A^e} B$ is the Hochschild homology on the dumbbell (whose handle consists of two edges) graph. We now define an inverse to the quasi-isomorphism $(B \otimes B) \otimes_{A^{\otimes 4}} QA \to (B \otimes B) \otimes_{A^{\otimes 4}} A$ induced by mapping the figure eight to the dumbbell,
$$
\begin{tikzcd}[row sep=1ex]
B \otimes_{A^e} A \otimes_{A^e} B \ar[r] & B \otimes_{A^e} B \otimes_A B \otimes_{A^e} B\\
\alpha|x|\beta \arrow[u,symbol=\in] \ar[r,mapsto]& \alpha'' |1| (\alpha''')^*\Sha \alpha' | x| \beta' \Sha (\beta''')^* |1| \beta'' \arrow[u,symbol=\in].
\end{tikzcd}
$$
Composing with the map $\Map(8) \to LM$ described in \eqref{equ:concatloop} we now obtain a description of the string product (cohomology coproduct) as the composition
$$
\begin{tikzcd}[row sep=1ex]
B \otimes_{A^e} A \ar[r] & B \otimes_{A^e} A \otimes_{A^e} B \ar[r] & B \otimes_{A^e} B \otimes_A B \otimes_{A^e} B \ar[r, "H^{A\otimes A}\circ(\wedge Th)"] & (B \otimes_{A^e} A) \otimes (A \otimes_{A^e} B) \\
\alpha x \arrow[u,symbol=\in] \ar[r,mapsto]& \alpha' x \alpha'' \arrow[u,symbol=\in] \ar[r,mapsto]& \alpha^{(2)} \otimes (\alpha^{(3)})^* \Sha \alpha^{(1)} x \alpha^{(4)} \Sha (\alpha^{(6)})^* \otimes \alpha^{(5)} \arrow[u,symbol=\in] \ar[r,mapsto]& 
\scriptstyle
\alpha^{(2)} \otimes H^{A\otimes A}((\alpha^{(3)})^* \Sha \alpha^{(1)} x \alpha^{(4)} \Sha (\alpha^{(6)})^* \wedge \Th) \otimes \alpha^{(5)}
\arrow[u,symbol=\in]
\end{tikzcd}\,.
$$
Since the homotopy $H_{A \otimes A}$ was constructed to respect trivalent graphs, we can readily restrict the above map to closed (with respect to edge contraction) trivalent graphs. Identifying the corresponding spaces with $T\bar{H} \otimes H$ we obtain the following
\begin{Thm}
Under the natural map $T\bar{H} \otimes H \to \Omega^\bullet(LM)$ the string product is given by
\begin{equation}
\begin{tikzcd}[row sep=1ex]
T\bar{H} \otimes H \ar[r] & (T\bar{H} \otimes H) \otimes (T\bar{H} \otimes H) \\
\alpha|x \ar[u, symbol=\in] \ar[mapsto]{r} 
&
z_0(e^i (\alpha^{(3)})^* \Sha \alpha^{(1)} x \alpha^{(4)} \Sha (\alpha^{(6)})^* e^j) (\alpha^{(2)}|e_i)\otimes (\alpha^{(5)}|e_j) \ar[u, symbol=\in] .
\end{tikzcd}
\end{equation}
\end{Thm}

The string bracket (cohomology cobracket) is given by the composition \eqref{equ:defbracket} where the maps are modelled by Lemma \ref{lem:equivtohoch}. In particular, in the above formula $x=1$ and we get the simplified formula in
\begin{Thm}\label{thm:bracket}
Under the map $\Cyc{\bar H} \to \Omega^\bullet_{S^1}(LM)$ the string bracket is modelled by 
\begin{equation}
\begin{tikzcd}[row sep=1ex]
\Cyc{\bar H} \ar[r] & \Cyc{\bar H} \otimes \Cyc{\bar H} 
\\
\alpha_1 \dots \alpha_k \ar[u, symbol=\in] \ar[mapsto,r] 
& \sum_j \alpha_1 \dots \alpha_j e_i \otimes e^i \alpha_{j+1} \dots \alpha_k 
\ar[u, symbol=\in]
\end{tikzcd}
\end{equation}
\end{Thm}

\subsubsection{A remark on homotopy automorphisms}
Let us record some basic facts about the above formula for later. For this, recall that $H^\bullet$ is a $\Com_\infty$ algebra, and as such has a cobar-dual dg Lie algebra $\mathbb{L} = \Lie(\bar{H}_{\bullet}[-1])$ with a differential that sends generators $\bar{H}_\bullet[-1] \to \Lie(\bar{H}_\bullet[-1])$. Considering the $\Com_\infty$-algebra $H^\bullet$ as an $A_\infty$-algebra and taking its (dual) cobar construction one obtains the dg algebra (the tensor algebra in $\bar{H}_\bullet[-1]$) $T(\bar{H}_\bullet) = U\mathbb{L}$. We identify
$$
\Cyc(\bar{H})^* = (T\bar{H}_\bullet)_\L = (T\bar{H}_\bullet)^\L
$$
with $\L$-(co)invariants of $T\bar{H}_\bullet$. Then the following double complex is contractible
$$
\begin{tikzcd}
(T\bar{H}_\bullet)^\L \ar[r] & T\bar{H}_\bullet \otimes \bar{H}_\bullet \ar[d, "{[\cdot, \cdot]}"] \\
& T\bar{H}_\bullet \otimes \R \ar[r] & (T\bar{H}_\bullet)_\L.
\end{tikzcd}
$$
Which we write as
$$
\begin{tikzcd}
\Cyc \bar{H}_\bullet[-1] \ar[r, "\pi_!"] & T\bar{H}_\bullet \otimes H_\bullet \ar[r] & \Cyc \bar{H}_\bullet,
\end{tikzcd}
$$
where the maps are dual to the ones in Lemma \ref{lem:equivtohoch}. By identifying $H_\bullet = H^\bullet[-n]$ using the inner product we obtain the contractible complex
\begin{equation}\label{equ:contrgysin}
\begin{tikzcd}
\Cyc \bar{H}_\bullet [-1] \ar[r, "\pi_!"] & T\bar{H}_\bullet \otimes H^\bullet[-n] \ar[r] & \Cyc \bar{H}_\bullet,
\end{tikzcd}
\end{equation}
where the middle term is Hochschild cohomology of the $Com_\infty$-algebra $H^\bullet$. One checks that the map
$$
\Cyc \bar{H}_\bullet \overset{\pi_!}{\longrightarrow} T\bar{H}_\bullet \otimes H^\bullet[1-n]
$$
is a map of Lie algebras that is furthermore compatible with the natural action of the Hochschild-Gerstenhaber Lie algebra $T\bar{H}_\bullet \otimes H^\bullet$. More concretely, we write $T\bar{H}_\bullet \otimes H^\bullet = T\bar{H}_\bullet \otimes \R \ \oplus \ T\bar{H}_\bullet \otimes \bar{H}^\bullet$ and identify the two summands with the space of inner derivations $\op{Inn}$ and all derivations $\Der(T\bar{H})$, respectively. Now the natural action of $(\op{Inn} \to \Der(T\bar{H}_\bullet))$ on $\Cyc(\bar{H}_\bullet)$ composed with the map $\pi_!$ coincides with the adjoint action (with respect to the string bracket) of $\Cyc(\bar{H}_\bullet)$ on itself.

Recall that there is a Hodge decomposition coming from writing $U\L = S\L = \oplus_k S^k\L$ (i.e. PBW) and the maps in complex \eqref{equ:contrgysin} respect that decomposition (can for instance be seen by identifying $T\bar{H}_\bullet \otimes \bar{H}_\bullet = \Omega^1_{nc}$ with non-commutative 1-forms). Let us use the notation
$$
    \Cyc \bar{H}_\bullet = \oplus_k \Cyc_{(k)} \bar{H}_\bullet = \oplus_k (S^k \L)_{\L}
$$
and similarly
$$
    H_\bullet^{S^1}(LM) = \oplus_k H_\bullet^{S^1}(LM)_{(k)}
$$
for the corresponding decomposition on the equivariant loop space.
Complex \eqref{equ:contrgysin} is then the direct sum of contractible complexes
$$
\begin{tikzcd}
\Cyc_{(k)} \bar{H}_\bullet [-1] \ar[r, "\pi_!"] & S^{k-1} \L \otimes H^\bullet[-n] \ar[r] & \Cyc_{(k-1)} \bar{H}_\bullet,
\end{tikzcd}
$$
the $k=2$ term of which is
$$
\begin{tikzcd}
\Cyc_{(2)} \bar{H}_\bullet [-1] \ar[r, "\pi_!"] & (\op{Inn} \to \Der(\L))[-n] \ar[r, "u \to u(\omega)"] & \Cyc_{(1)} \bar{H}_\bullet = \bar{H}_\bullet,
\end{tikzcd}
$$
where $(\op{Inn} \to \Der(\L))$ is the Harrison complex of $H^\bullet$. After shifting by $n-1$ the positive truncation gives us that
$$
(\Cyc_{(2)} \bar{H}_\bullet)[n-1]^+ \longrightarrow  (\op{Inn} \to \Der(\L))^+
$$
is a quasi-isomorphism (see also Proposition 65 in \cite{CamposWillwacher} for a direct proof).
We have thus obtained the following
\begin{Lem}
The identification $\pi_!$ of $(\op{Inn} \to \Der(\Lie\bar{H}_\bullet))^+$ with a direct summand of $\Cyc(\bar{H}_\bullet)$ intertwines the natural action of the Harrison complex on the cyclic complex with the adjoint action.
\end{Lem}

\subsection{String coproduct and cobracket}
As before, we will distinguish two cases, the $1$-framed case in which a nonvanishing vector field is chosen (and necessarily $\chi(M)=0$), and the reduced case for general $M$.
\subsubsection{The 1-framed case}
By combining Proposition \ref{prop:splitmap} and Lemma  \ref{lem:frcop} we obtain the description of the coproduct $\Omega^\bullet(LM) \longleftarrow \Omega^{\bullet +1 -n}(\Map(8))$ as
$$
\begin{tikzcd}
B \otimes_{A^{\otimes 2}} A & \ar[l, "s"] (B\otimes B) \otimes_{A^{\otimes 4}} (A \otimes A / A) & \ar[l, "{(H_{A\otimes A}, H_{A})}"] (B\otimes B) \otimes_{A^{\otimes 4}} (B \otimes_A B) \ar[d, "\simeq"] \\
&& (B \otimes B) \otimes_{A^{\otimes 4}} A.
\end{tikzcd}
$$
An inverse of the last quasi-isomorphism is given by
\begin{equation}
\label{eqn:8totheta}
\alpha \otimes \beta \otimes x \longmapsto \pm \alpha'' \otimes \beta'' \otimes \alpha''' \Sha \beta''' x \alpha' \Sha \beta',
\end{equation}
i.e. the homotopy equivalence of the mapping spaces of the figure 8 and the theta graph. Finally, we compose it with the map in diagram \eqref{equ:splitloop} to obtain the coproduct. We thus obtain the formula
$$
\begin{tikzcd}[row sep=1ex]
B \otimes_{A^{\otimes 2}} A & \ar[l] B \otimes_{A^{\otimes 2}} A \otimes B \otimes_{A^{\otimes 2}} A \\
 \begin{aligned} (\beta'' \Psi' \alpha'' | \Psi'')  \\ + (\beta'' |(H_{A} )(\alpha'' \Sha \beta''' xy \alpha' \Sha \beta' \wedge \Th)) \\ + (\alpha'' |(H_{A})(\alpha''' \Sha \beta'' xy \alpha' \Sha \beta' \wedge \Th))\end{aligned}   \arrow[u,symbol=\in] & \ar[l,mapsto] \arrow[u,symbol=\in] (\alpha | x) \otimes (\beta | y), \\
\end{tikzcd}
$$
where $\Psi' \otimes \Psi'' = H_{A \otimes A}(\alpha''' \Sha \beta''' x \alpha' \Sha \beta' \wedge \Th)$. We note that the formula again preserves $(B \otimes_{A^e} A)^{\text{triv, cl}} \subset B \otimes_{A^{\otimes 2}} A$, however, it does not at the moment commute with the projection $T\bar{H} \otimes H$, since we evaluate $H_{A\otimes A}$ and $H_A$ on a graph containing a 4-valent vertex. To obtain formulas one could precompose it with another homotopy (and another evaluation of $z_0$) similar to the ones constructed in the previous chapter. Since the final formula is not very enlightening we choose not to do so. Instead we note that this difficulty does not arise on the image of $\Cyc(A) \to B \otimes_{A^{\otimes 2}} A$. And hence from \eqref{eqn:simplpsi} we obtain the description of the cobracket as in Theorem \ref{thm:liebialg}.
\begin{Thm}\label{thm:cobracket framed}
The cobracket in the 1-framed case is described by the map
$$
\begin{tikzcd}[row sep=1ex]
\Cyc(\bar{H}) \otimes \Cyc(\bar{H}) \ar[r] & T\bar{H} \otimes H \ar[r, "pr"] & \Cyc(\bar{H}) \\
\alpha \otimes \beta \arrow[u,symbol=\in] \ar[r,mapsto]& \arrow[u,symbol=\in] (\alpha \omega_i \beta | \omega^i) + (\alpha' |\omega_i) z_1(\omega^i \alpha'' \beta) + (\beta' | \omega_i) z_1(\omega^i \beta'' \alpha).
\end{tikzcd}
$$
\end{Thm}

\subsubsection{The reduced case}
Similarly by Propositions \ref{prop:splitmap} and Lemma \ref{lem:frcop} we obtain a description of the coproduct $\Omega^\bullet(LM) \longleftarrow \Omega^{\bullet +1 -n}(\Map(8), F)$ as
$$
\begin{tikzcd}[column sep= huge]
B \otimes_{A^{\otimes 2}} A & \ar[l, "s"] \Tot\left(    \begin{tikzpicture}
        \node(1) at (0,.5) {$(B \otimes B) \otimes_{A^{\otimes 4}} (A \otimes A)$};
        \node(3) at (0,-.5) {$(B \overset{A}{\oplus} B) \otimes_{A^{\otimes 2}} A$};
        \draw[->](1)-- (3);
    \end{tikzpicture} \right) &[25pt]
    \ar[l, "1 \otimes H_{A \otimes A}"', start anchor={[yshift=3ex]},end anchor={[yshift=3ex]}]
    \ar[l, "{(\epsilon \otimes 1 - 1 \otimes \epsilon) \otimes H_A}" description, start anchor={[yshift=2ex]},end anchor={[yshift=-2ex]}]
    \ar[l, "1", start anchor={[yshift=-3ex]},end anchor={[yshift=-3ex]}]
    \Tot\left(    \begin{tikzpicture}
        \node(1) at (0,.5) {$(B \otimes B) \otimes_{A^{\otimes 4}} (B \otimes_A B)$};
        \node(3) at (0,-.5) {$(B \overset{A}{\oplus} B) \otimes_{A^{\otimes 2}} A$};
        \draw[->](1)-- (3);
    \end{tikzpicture} \right) \ar{d}{\simeq}
  \\
& & \Tot\left(    \begin{tikzpicture}
        \node(1) at (0,.5) {$(B \otimes B) \otimes_{A^{\otimes 4}} A$};
        \node(3) at (0,-.5) {$(B \overset{A}{\oplus} B) \otimes_{A^{\otimes 2}} A$};
        \draw[->](1)-- (3);
    \end{tikzpicture} \right)
\end{tikzcd}
$$
Since the homotopy inverse \eqref{eqn:8totheta} is a strict right inverse and the boundary map in the upper complex factors through this projection, we can use it again to invert the last quasi-isomorphism in the diagram. Moreover, since $(B \otimes B) \otimes_{A^{\otimes 4}} A \to (B \overset{A}{\oplus} B) \otimes_{A^{\otimes 2}} A$ is onto, we can simplify the last term by taking the kernel of this map, that is our model for $\Omega^\bullet(\Map(8), F)$ is $(B^{\geq 1} \otimes B^{\geq 1}) \otimes_{A^{\otimes 4}} A$. We then see that the composite
$$
(B^{\geq 1} \otimes B^{\geq 1}) \otimes_{A^{\otimes 4}} A \to B \otimes_{A^{\otimes 2}} A,
$$
is the same as in the 1-framed case, and hence so are the formulas.
\begin{Thm}\label{thm:cobracket reduced}
The reduced cobracket is described by the map
$$
\begin{tikzcd}[row sep=1ex]
\bCyc(\bar{H}) \otimes \bCyc(\bar{H}) \ar[r] & T\bar{H} \otimes H \ar[r, "pr"] & \bCyc(\bar{H}) \\
\alpha \otimes \beta \arrow[u,symbol=\in] \ar[r,mapsto]& \arrow[u,symbol=\in] (\alpha e_i \beta | e^i) + (\alpha' | e_i) z_1( e^i \alpha'' \beta) + (\beta' | e_i) z_1( e^i \beta'' \alpha) &.
\end{tikzcd}
$$
\end{Thm}

\newcommand{\FFM}{\FM^{fr}}
\newcommand{\Aut}{\mathrm{Aut}}

\section{Discussion -- string topology, invariants of manifolds, and the diffeomorphism group}\label{sec:discussion}

There is the hope that string topology can be used to study manifolds, and that it is sensitive to the structure of $M$, beyond its (rational) homotopy type.
One can ask two related questions:
\begin{enumerate}
    \item How strong is the invariant of $M$ given by (a version of) $H(LM)$ or $H_{S^1}(LM)$, together with some chosen set of algebraic (string topology-)operations?
    \item A version of the diffeomorphism group acts on (a version of) $H(LM)$, preserving the chosen set of algebraic structure. Hence we get a map from the diffeomorphism group of $M$ to the algebraically defined (homotopy) automorphism group of $H(LM)$, with the algebraic structure considered. How non-trivial is this map?
\end{enumerate}
Of course, similar questions can in principle be asked on the chain level, but we only consider (co)homology here.
In this paper we connect string topology to configuration spaces. Hence it makes sense to compare the "strength" of string topology to similar invariants build from configuration spaces of points, via the Goodwillie-Weiss manifold calculus.
More concretely, to every manifold $M$ one can associate the framed configuration spaces $\FFM_M$, as right modules over the (fulton-MacPherson-version of the) framed little disks operad $\FFM_n$. The homotopy type of those forms an invariant, and it is acted upon by the diffeomorphisms of $M$. More generally, we can truncate at arity (number of points) $\leq k$ and get a tower of approximations to the diffeomorphism group, and a hierarchy of invariants:
\begin{equation*}\label{eq:towers}
\begin{tikzcd}
 & & & & \Diff(M)\ar{d} \\
T_1\Diff(M):=\Aut^h_{\FFM_n,\leq 1}(\FFM_M)
\ar{d}
&
\cdots \ar{d}
\ar{l} &
T_k\Diff(M):=\Aut^h_{\FFM_n,\leq k}(\FFM_M)
\ar{d}\ar{l} &
\cdots \ar{d}
\ar{l} &
T_\infty\Diff(M):=\Aut^h_{\FFM_n}(\FFM_M)
\ar{l}\ar{d}
\\
\Aut^h_{(\FFM_n)^\Q,\leq 1}((\FFM_M)^\Q)
&
\cdots
\ar{l} &
\Aut^h_{(\FFM_n)^\Q,\leq k}((\FFM_M)^\Q)
\ar{l} &
\cdots
\ar{l}
&
\Aut^h_{(\FFM_n)^\Q}((\FFM_M)^\Q)
\ar{l}
\end{tikzcd}.
\end{equation*}

Here $\Aut^h_{\FFM_n}(-)$ refers to homotopy automorphisms of the right $\FFM_n$-module and $\Aut^h_{\FFM_n,\leq k}(-)$ to its arity-$k$-truncated version, seeing only configuration spaces of $\leq k$ points.
We also added the corresponding rationalized versions in the lower row. 
We could also consider instead the non-framed configuration spaces $\FM_M$, as a right modules over the fiberwise $E_n$-operads $\FM_n^M$, or replace the diagram with an version for configuration categories \cite{BoavidaWeiss}.
In any case the arrows in the diagram above are not well understood, and the usual embedding calculus convergence estimates do not apply.

\subsection{Diffeomorphism invariance}
Let $G = \Diff_1(M)$ be the identity component of the group of diffeomorphisms. We consider $G$-invariance of the (reduced) coproduct. In the definition \eqref{eqn:defredcop} every map except the last one is a map of spaces, and moreover equivariant with respect to $G$. That is there is a commuting diagram of spaces
$$
\begin{tikzcd}
G \times (I, \partial I) \times (LM,M) \ar[r]\ar[d] & G \times \frac{\Map(\bigcirc_2) / \Map^\prime(8)}{ F/ F|_{UTM}} \ar[d] & \ar[l, "\simeq"] \ar[d] G \times  \frac{\Map(8) / \Map^\prime(8)}{F / F|_{UTM}} \\
 (I, \partial I) \times (LM,M) \ar[r]& \frac{\Map(\bigcirc_2) / \Map^\prime(8)}{ F/ F|_{UTM}} & \ar[l, "\simeq"]  \frac{\Map(8) / \Map^\prime(8)}{F / F|_{UTM}}
\end{tikzcd}
$$
where the vertical arrows are the action. Thus we get an induced diagram in homology. The last step in the definition of the coproduct was taking cap product with a Thom class in $H^n(M, UTM)$. For this recall that the cap product is natural and the map $\frac{\Map(8) / \Map^\prime(8)}{F / F|_{UTM}} \to M/UTM$ is $G$-equivariant, it thus follows that
$$
\begin{tikzcd}
H_\bullet(G \times \frac{\Map(8) / \Map^\prime(8)}{F / F|_{UTM}}) \ar[d]\ar[r, "\cap"] & \op{Hom} (H^n(G \times (M/UTM)), H_\bullet( G \times \frac{\Map(8)}{F}) ) \ar[d]\\
H_\bullet(\frac{\Map(8) / \Map^\prime(8)}{F / F|_{UTM}}) \ar[r, "\cap"] & \op{Hom} (H^n(M/UTM), H_\bullet( \frac{\Map(8)}{F} ) )
\end{tikzcd}
$$
commutes. For degree reasons the image of the Thom class $\Th$ under $H^n(M/UTM) \to H^n(G \times M/UTM)$ is simply $1 \otimes \Th$. Thus by compatibility of capping with products we get that
$$
\begin{tikzcd}
H_\bullet(G) \otimes H_\bullet(\frac{\Map(8) / \Map^\prime(8)}{F / F|_{UTM}})) \ar[d]\ar[r, "\iid \otimes \cap \Th"] & \op{Hom} H_\bullet(G) \otimes H_\bullet( \frac{\Map(8)}{F}) \ar[d]\\
H_\bullet(\frac{\Map(8) / \Map^\prime(8)}{F / F|_{UTM}}) \ar[r, "\cap \Th"] & H_\bullet(\frac{\Map(8)}{F})
\end{tikzcd}
$$
Since the map $\Map(8) \to LM \times LM$ is again $G$-equivariant we obtain that
$$
\begin{tikzcd}
H_\bullet(G) \otimes H_\bullet(LM,M) \ar[d] \ar[r] & H_\bullet(LM,M) \ar[d] \\
H_\bullet(G) \otimes H_\bullet(LM, M) \otimes H_\bullet(LM, M) \ar[r] & H_\bullet(LM, M) \otimes H_\bullet(LM, M).
\end{tikzcd}
$$
commutes. Finally, we obtain similar commuting diagrams for the maps in the Gysin sequence associated to $LM \to LM_{S^1}$ and $G \times LM \to G \times LM_{S^1}$, respectively, which follows from naturality of the Gysin sequence.
Thus we obtain
$$
\begin{tikzcd}
H_\bullet(G) \otimes \bar{H}^{S^1}_\bullet(LM) \ar[d] \ar[r] & \bar{H}^{S^1}_\bullet(LM) \ar[d] \\
H_\bullet(G) \otimes \bar{H}^{S^1}_\bullet(LM) \otimes \bar{H}^{S^1}_\bullet(LM) \ar[r] & \bar{H}^{S^1}_\bullet(LM) \otimes \bar{H}^{S^1}_\bullet(LM).
\end{tikzcd}
$$
Or in formula
$$
\delta(g.x) = g' \delta'(x) \otimes g'' \delta''(x)
$$
for any $g \in H_\bullet(G)$ and $x \in \bar{H}^{S^1}(LM)$ where $\Delta(g) = g' \otimes g''$ is the image of $g$ under the diagonal $H_\bullet(G) \to H_\bullet(G) \otimes H_\bullet(G)$. In particular, primitive elements in $H_\bullet(G)$ act by derivations of the string cobracket. Recall that by Milnor-Moore (Theorem 21.5 in \cite{FHT2}) $H_\bullet(G)$ is the universal enveloping algebra of the Lie algebra $\pi_*(G) \otimes \R$.
From which it follows that we get
$$
( \pi_*(G), [,]) \to Der_{[,],\delta}(\bar{H}^{S^1}(LM)).
$$
That is $\pi_*(\Diff_1(M))$ acts on $\bar{H}^{S^1}(LM)$ by derivations of the Lie bialgebra structure. That the cobracket is preserved will give us a non-trivial term measuring the difference between $\Diff_1(M)$ and $aut_1(M)$, where $aut_1(M)$ is the identity component of the monoid of self-maps. Namely, we have
$$
\begin{tikzcd}
\pi_*(\Diff_1(M)) \ar[r] \ar[d] & Der_{[,],\delta}(\bar{H}^{S^1}(LM)) \ar[d] \\
\pi_*(aut_1(M)) \ar[r] & End(\bar{H}^{S^1}(LM)).
\end{tikzcd}
$$
Let now $M$ be simply-connected, such that we have rational models.

Then the lower arrow factors through
$$
\pi_*(aut_1(M)) \to H(\Der(\Lie\bar{H}_\bullet)/\op{Inn}(\Lie\bar{H}_\bullet)) \to End(H^{S^1}(LM)),
$$
where $\Lie\bar{H}_\bullet$ (the cobar construction of the $\Com_\infty$ coalgebra $H_\bullet$) is the underlying Quillen model for our free model $A$ for cochains on $M$. 
As noted in Lemma \ref{lem:compaction} we can identify $H((\Der(\Lie\bar{H}_\bullet)/\op{Inn}(\Lie\bar{H}_\bullet))^+)$ with the ($>n$-degree part of the) lowest Hodge-degree summand of $H^{S^1}(LM)$ such that the action becomes the adjoint action.
It follows that the map $\pi_*(aut_1(M)) \to End(H^{S^1}(LM))$ factors as the adjoint action of the string bracket
$$
\pi_*(aut_1(M)) \to H^{S^1}(LM)_{(2)}[n-1] \to End(H^{S^1}(LM)),
$$

In particular, we see that while the string bracket is preserved by every element in $\pi_*(aut_1(M))$, in contrast the string cobracket is (generally) not. More concretely, let $x \in \pi_*(aut_1(M)) = H^{S^1}_{(2)}(LM)$. Using the 5-term relations for the Lie bialgebra we then see that $[x, \cdot]$ is a derivation for the Lie bialgebra structure if and only if
$$
[y, \delta(x)] = 0 \text{ for all $y \in \bar{H}^{S^1}(LM)$}.
$$
\todo[inline]{Mention that Graph-complexes and/or Burghelea say there is a map $\pi_*(aut(X)) \to H_{S^1}(LM)_{\Z_2}$ and that $\Diff$ is in its kernel. We do not get that it is in the kernel, but that it acts trivially on $H_{S^1}(LM) \otimes H_{S^1}(LM)$. This is still a non-trivial condition, and in good cases no weaker.}

\subsection{Example}
To get a concrete example we proceed as follows. Let us restrict ourselves further to the case where $\chi(M) = 0$ and we have chosen a fixed non-vanishing vector field $\xi$. Moreover, we only consider $\Diff^\xi_1(M)$ and $aut^\xi(M)$ that preserve $\xi$. (More precisely, the action of $\Diff$ on the lift of the Thom class in $H^{n-1}(UTM) \to H^n(M,UTM)$ determines a cocycle $H_\bullet(G) \to H^\bullet(M)$ and we take the kernel of that cocycle.) Thus we can work in the non-reduced setting, which is $\Diff^\xi_1(M)$-equivariant by a similar argument as above.  And can in particular apply the counit $H^{S^1}(LM) \to \R$ on one factors of the above equation.
We refer to \cite{AKKN} for the fact that by applying the counit we obtain
$$
(1 \otimes \epsilon) \delta(x) \in Z( H^{S^1}(LM,M) , [,]).
$$
To illustrate the nontriviality of this condition. Let us now take $M = (S^n \otimes S^n)^{\#g}$ for $n$ odd. It is formal and coformal (i.e. the cohomology algebra is Koszul). In that case the Lie bialgebra is the same as the one obtained from a surface (or the one constructed by Schedler) see also \cite{AKKN}. Moreover, it has trivial center and we thus obtain that $\pi_*(\Diff(M))$ lie in the kernel of $(1 \otimes \epsilon) \delta$. 
We moreover identify
$$
\pi_*(aut_1(M)) \otimes \R = \op{OutDer}^+(\mathbb{L}),
$$
where $\mathbb{L} = \Lie( x_1,\dots,x_g, y_1,\dots,y_g)/(\omega))$ and $\omega = \sum_i [x_i, y_i]$ (see for instance Theorem 5.7. in \cite{BerglundMadsen}). The action on the framing gives a map $\pi_*(aut_1(M)) \to H$ and hence $\pi_*(aut^\xi_1(M))$ differs from $\pi_*(aut_1(M))$ by at most a factor of $H$.
\todo[inline]{The map $Der(L) \to H$ should just be the "divergence" on tripods. It's probably too cumbersome to actually prove that. We only need that $Der(L)^+$ is much larger than $H$ and still contains elements with non-zero divergence.}
The elements that preserve the cobracket are now
$$
\op{ker} (1\otimes \epsilon)\delta = \{ u \in \op{OutDer}^+(\mathbb{L}) \ | \ (1\otimes \epsilon) \delta(u) = 0 \},
$$
and $(1\otimes \epsilon) \delta(u)$ can be identified with some non-commutative divergence as in \cite{AKKN}. Note that this last Lie algebra is very closely related to $\mathfrak{krv}^{g,1}$ (see loc. cit. for a definition). By definition
$$
\op{ker} (1\otimes \epsilon)\delta \to \op{OutDer}^+(\mathbb{L}) \overset{(1\otimes \epsilon)\delta}{\longrightarrow} U\mathbb{L}/[U\mathbb{L}, U\mathbb{L}]
$$
is exact in the middle and one checks that the second map is non-trivial. Thus we indeed get obstructions for $\pi_*(\Diff)$ in this case.

\subsection{Outlook, and a conjecture}
Note that in our definition the string topology operations on the rational (or real) cohomology of the free loop space depend on $M$ only the rational (or real) homotopy type of the configuration spaces of up to two points on $M$.
Hence string topology, to the extend considered here, contains at most as much "information" about $M$ and $\Diff(M)$ as can be obtained through at the second stage of the rationalized Goodwillie-Weiss tower, appearing at the beginning of this section.

In general, one might conceivably consider a larger set of string topology operations, as has been done for example in \cite{ChasSullivan2}. One might also lift these operations to the (say rational) (co)chain level.
We raise the admittedly very vague conjecture, that nevertheless string topology is at most as strong an invariant of the manifold as can be obtained via the configuration spaces.
More concretely, we conjecture that the higher string topology operations on rational cohomology or cochains of $LM$ can all be described using the rational homotopy type of $\FM_M$ as a module over the fiberwise little disks operad $\FM_n^M$. (Alternatively, this can be replaced by the configuration category of $M$ of \cite{BoavidaWeiss}.)
Furthermore, in this manner one will likely obtain a morphism of topological groups
\[
\Diff(M)\to \text{``}\Aut^h_{\text{string topology}}(C(LM))\text{''}
\]
from the diffeomorphism group to the group of homotopy automorphisms of the cochains on $LM$, preserving all string topology operations. The precise formulation of the right-hand side is to be found, there are several nice subsets of operations one can consider.
We conjecture that at least for for simply connected $M$ the above morphism of topological groups factors through the homotopy automorphisms of the rationalized version $\FM_M$ as a module over the fiberwise little disks operad $\FM_n^M$. Equivalently we may also take $\Aut^h_{(\FFM_n)^\Q}((\FFM_M)^\Q)$.
The homotopy type of the latter object will be determined in \cite{WillwacherTBD}, with the result that a Lie model is given by a slight extension of the graph complex $\GC_M$ of section \ref{sec:GCM}, which also naturally acts on our model $\Graphs_M$ of the configuration spaces.
Hence the appearance of the one-loop term $Z_1$ in the formulas of Theorems \ref{thm:cobracket framed} and \ref{thm:cobracket reduced} above can be seen as a reflection of the appearance of configuration spaces in string topology.

In the other direction, there have been several results or announced results that string topology only depends on the (rational) homotopy type of $M$, for some subsets of string topology operations considered.
Our work indicates to the contrary that string topology, with the correct set of operations, can be used to access information of $M$ beyond its rational homotopy type.

\end{document}